%% file: PHI_algebra_Dynkin.tex
\documentclass{ws-ijac}

\usepackage[english]{babel}
\usepackage[utf8]{inputenc}

\usepackage[fixlanguage]{babelbib}

\usepackage{tikz-cd}

\usepackage{color}
\usepackage{tensor} 
\usepackage{amsmath} 
\usepackage{amsfonts} 
\usepackage{amssymb} 
\usepackage{tabto}
\usepackage{multicol}
\usepackage{url}
\usepackage{paralist} 
\usepackage{enumitem} 

\usepackage{tikz-cd}
\usetikzlibrary{decorations.pathreplacing} 
\usepackage{mathtools}
\usepackage{mathdots}
\makeatletter
\DeclareRobustCommand{\rvdots}{%
    \vbox{
        \baselineskip4\p@\lineskiplimit\z@
        \kern-\p@
        \hbox{.}\hbox{.}\hbox{.}
    }}

\usepackage{MnSymbol} 

\newcommand{\md}{\operatorname{mod}}

\newcommand{\Coh}{\operatorname{Coh}}

\newcommand{\Hom}{\operatorname{Hom}}
\newcommand{\Ext}{\operatorname{Ext}}
\newcommand{\End}{\operatorname{End}}

\newcommand{\ramo}{\text{\textasteriskcentered}}

\newcommand{\dime}{\operatorname{dim}}
\newcommand{\menas}{\scalebox{0.75}[1.0]{\( - \)}}

\begin{document}

\title{PIECEWISE HEREDITARY INCIDENCE ALGEBRAS OF DYNKIN AND EXTENDED DYNKIN TYPE}

\author{EDUARDO DO N. MARCOS}
\address{Dto de Matem\'atica, Instituto de Matem\'atica e Estat\'istica, Universidade de S\~ao Paulo, Rua do Mat\~ao 1010, Cidade Universit\'aria\\ S\~ao Paulo-SP, CEP 05508-090, Brasil}
\email{enmarcos@ime.usp.br}

\author{MARCELO MOREIRA}
\address{Dto de Matem\'atica, 
Instituto de Ci\^encias Exatas, Universidade Federal de Alfenas, Campus Sede, 
Rua Gabriel Monteiro da Silva 700, Centro\\
Alfenas-MG, CEP 37130-001, Brasil}
\email{marcelo.moreira@unifal-mg.edu.br}

\maketitle

\begin{history}
\received{(Day Month 2019)}
\accepted{(Day Month Year)}
\comby{[editor]}
\end{history}

\begin{abstract}
Let $K\Delta$ be the incidence algebra associated with a finite poset $(\Delta,\preceq)$ over the algebraically closed field $K$. We present a study of incidence algebras $K\Delta$ that are piecewise hereditary, which we denominate PHI algebras. 

We describe the quiver with relations of the PHI algebras of Dyn\-kin type and introduce a new family of PHI algebras of extended Dynkin type, which we call ANS family, in reference to Assem, Neh\-ring, and Skowro\'nski. In this description, the important method was the one of cutting sets on trivial extensions, inspired by this we made  a computer program which shows exactly the cutting sets on the given trivial extension that result on incidence algebras.
\end{abstract}

\keywords{incidence algebra; piecewise hereditary algebra; trivial extension.}

\ccode{Mathematics Subject Classification 2010: 16D10 16G10 16G20}

\section{Introduction}

Throughout the paper, $K$ denotes an algebraically closed field. All algebra will be finite dimensional basic associative $K$-algebra. Using Gabriel's theorem we will assume all algebras to be of the form $KQ/I$, where $Q$ is a finite quiver and $I$ is an admissible ideal. All modules will be finite dimensional right modules.

Incidence algebras were introduced in the mid-1960s as a natural way of studying some combinatorial problems. In the representation theory of finite dimensional algebras, the incidence algebras have been the subject of many investigations (see, for instance, \cite{lad1}, \cite{red1}, \cite{ass-pla-red-tre} and \cite{ass-cas-mar-tre}). We will focus the study on 
incidence algebras $K\Delta$ associated with a finite poset $\Delta$ over $K$.
We remark that $K\Delta$ is isomorphic to the algebra $KQ/I$, where the quiver $Q$ is the Hasse diagram and $I$ is the ideal generated by all commuting relations, i.e. the difference of any pair of parallel paths are in $I$.

The classification of algebras is one of the most important problems in the representation of finite-dimensional algebras. In the 1980s, Assem and Happel began to classify the piecewise hereditary algebras, see \cite{ass-hap}. Many mathematicians contributed to the classification of piecewise hereditary algebras such as Keller \cite{kel}, Fern\'andez \cite{fern}, among others.

In the article \cite{mar-mor}, we consider the incidence algebras that are piecewise hereditary, which we called PHI algebras, \emph{piecewise hereditary incidence algebras}. The purpose of this paper is to give a description of some classes of the PHI algebras. We describe the PHI algebras through their quivers with relations. 

Let $\mathcal{A}$ and $\mathcal{B}$ be an abelian categories. In this paper the notation $\mathcal{A} \cong' \mathcal{B}$ means that $\mathcal{A}$ is derived equivalently to $\mathcal{B}$, that is $D^b(\mathcal{A})$ and $D^b(\mathcal{B})$ are equivalent as triangulated categories, we also use the notation 
$D^b (A)$ for the category
$D^b(\md A)$, where $\md A$ denotes the category of finitely generated modules over a finite dimensional algebra $A$.

The paper is organized as follows:

Section \ref{sec2} is devoted to fixing the notation and we briefly recall some definitions and results. In Section \ref{sec3}, we characterize the PHI algebras of type $\mathbb{A}_n$ and the PHI algebras of type $\widetilde{\mathbb{A}}_n$. The Sections \ref{sec4} and \ref{sec5} are dedicated to the description of PHI algebras such that $K\Delta \cong' KQ$ where $\overline{Q}=\mathbb{D}_n$ and $\overline{Q}=\mathbb{E}_6$, respectively.

In Sections \ref{sec6} and \ref{sec7}, we give a presentation by quivers with relations of PHI algebras such that $K\Delta \cong' KQ$ where $\overline{Q}=\mathbb{E}_7$ and $\overline{Q}=\mathbb{E}_8$, respectively. In Section \ref{sectame}, we introduce a family PHI algebras of extended Dynkin type called the ANS family, in reference to Assem, Nehring and Skowro\'nski. In this description of the PHI algebras, the important method was that of  cutting sets in trivial extensions, introduced by Fern\'andez and Platzeck, \cite{fern-plat}, this method inspired the elaboration of a computer program that produces exactly the cutting sets in the given trivial extension that result in incidence algebras. 

We also give the code of this computer program in the appendix,the code was written in java script.

This is part of the PhD of the second named author, he thanks CAPES(Brazil), for financial help during this time of the PhD.

\section{Preliminaries} \label{sec2}
\input{preli}

\section{Type $\mathbb{A}_n$ and type $\widetilde{\mathbb{A}}_n$} \label{sec3}
\input{an}

\section{Type $\mathbb{D}_n$} \label{sec4}
\input{dn}

\section{Type $\mathbb{E}_6$} \label{sec5}
\input{e6}

\section{Type $\mathbb{E}_7$} \label{sec6}
\input{e7}

\section{Type $\mathbb{E}_8$} \label{sec7}
\input{e8}

\section{Type $\widetilde{\mathbb{D}}_n$, type $\widetilde{\mathbb{E}}_6$, type $\widetilde{\mathbb{E}}_7$, type $\widetilde{\mathbb{E}}_8$} \label{sectame}
\input{tame}

\appendix

\input{ape-programa2} 

\section*{Acknowledgements}

The first named author has been supported by the tematic project of Fapesp 2014/09310-5. 

The second named author acknowledges support from CAPES, in the form of a PhD Scholarship, PhD made at programa de Matem\'atica, IME-USP, (Brazil).

\bibliography{ref}

\end{document}

%% file: preli.tex
In this Section, for the sake of completeness, we will recall some definitions. The reader should see the references for more detail.

We begin with the definition of incidence algebras. There are several equivalent ways of defining incidence algebras of finite posets, we give one of them below. 

\begin{definition}[incidence algebra]
Let $(\Delta,\preceq)$ be a poset with $n$ elements. The incidence algebra $K\Delta$ is a quotient of the path algebra of the following quiver $Q$. 
The set of vertices, $Q_0$, is in bijection with the elements of the poset $\Delta$ and the set of arrows $Q_1$ is defining by declaring that there is an arrow $\alpha$ from a vertex $a$ to a vertex $b$, whenever $a \preceq b$ and there is no $a \preceq c \preceq b$, with $c \neq a$ and $c \neq b$.
Let $I$ be the ideal generated by all commutativity relations $\gamma - \gamma'$, with $\gamma$ and $\gamma'$ parallel paths. 
The incidence algebra $K\Delta$ is $KQ/I$.
\end{definition}

The quiver $Q$ of the incidence algebra, in the former definition, is also called the Hasse quiver of the poset.

We are going to assume always that our incidence algebras are connected, that is the Hasse quiver is connected.

For more details in the subject of incidence algebras we refer to \cite{lou} and \cite{bau-vil}. 

We want to define next the notion of piecewise hereditary algebras. In order to do this we need to introduce, very briefly, some previous notions. 

Given an abelian category $\mathcal{A}$ we denote $D^b(\mathcal{A})$ its bounded derived category, as usual if $A$ is a $K$-algebra then $D^b(A)$ denotes the bounded derived category of $\md A$.
  
An abelian category $\mathcal{H}$ is called \emph{hereditary} if the extension groups $\Ext_\mathcal{H}^n (\!X,Y\!)$ are zero for all $n \geq 2$ for any pair of objects $X$ and $Y$ of $\mathcal{H}$.

\begin{remark}
All hereditary categories considered in this paper have splitting idempotents, finite dimension $\Hom$ spaces, and tilting object. See below the definition of tilting object.
\end{remark}

\begin{definition}[piecewise hereditary algebra]
We say that $A$ is piecewise hereditary algebra of type $\mathcal{H}$ if there exists a hereditary abelian category $\mathcal{H}$, with splitting idempotents, finite dimension $\Hom$ spaces,
such that $D^b(\mathcal{A})$ is triangle-equivalent to the bounded derived category $D^b(\mathcal{H})$.
\end{definition}

For more details in the subject of piecewise hereditary algebra we refer to \cite{hap-rei-sma}, \cite{che-kra}, \cite{hap-ric-scho}, \cite{hap2}, \cite{lad2}, \cite{len}, \cite{len-sko}, \cite{bar-len1}, \cite{bar-len2}, \cite{hug-koe-liu}, \cite{hap-zac2}, \cite{meu} and \cite{alv-lem-mar}.

The definition of tilting modules inspired the definition of \emph{tilting object} that follows:

\begin{definition}[tilting object]
Let $\mathcal{H}$ be a hereditary abelian $K$-category. An object $T \in \mathcal{H}$ is called tilting if
\begin{asparaenum} [\itshape i)]
\item $\Ext^1_{\mathcal{H}} (T,T)=0$, and
\item for every $X \in \mathcal{H}$ the condition $\Ext^1_{\mathcal{H}} (T,X)=0=Hom_{\mathcal{H}}(T,X)$ implies that $X=0$.
\end{asparaenum}
\end{definition}

Let $A$ piecewise hereditary algebra of type $\mathcal{H}$. It follows from Rickard's theorem \cite{ric}, the existence of a tilting object $T$ in $D^b(\mathcal{H})$ such that $A = \End T$. 

Given a sequence  $p_1,\dotsc,p_n$ of positive integers, $\mathbb{X}(p_1,\dotsc,p_n)$, will denote the  weighted projective line of type $p_1,\dotsc,p_n$, in the sense of \cite{gei-len}, and $\Coh \mathbb{X}$  the category of coherent sheaves over $\mathbb{X}(p_1,\dotsc,p_n)$. 
Let $Q$ be a finite, connected quiver without oriented cycles and let $KQ$ denote the path algebra of $Q$.
We state one of the most important theorems about piecewise hereditary algebras.

\begin{theorem}[Happel \cite{hap2}]
Let $\mathcal{H}$ be an abelian hereditary connected $K$-catego\-ry with tilting object. Then $\mathcal{H}$ is derived equivalent to $\md KQ$ or derived equivalent to $\Coh \mathbb{X}$ for some weighted projective line $\mathbb{X}$.
\end{theorem}

An algebra $A$ is called a \emph{piecewise hereditary algebras of quiver type} (or of type $Q$) or \emph{of sheaf type} if $A \cong' KQ$ for some quiver $Q$ or $A \cong' \Coh \mathbb{X} $ for some weighted projective line $\mathbb{X}$, respectively.

Observe that an algebra can be, at the same time, of quiver and sheaf type.

In order to study the PHI algebras of quiver type, it is enough to characterize the iterated tilted incidence algebras of type $Q$, since, according to the following theorem:
\begin{theorem}[Happel-Rickard-Schofield \cite{hap-ric-scho}] \label{HRS}
Let $A$ be a finite dimensional basic associative $K$-algebra and $Q$ be a finite quiver with no oriented cycles. Then $A$ is piecewise hereditary of type $Q$ if and only if $A$ is iterated tilted of type $Q$.
\end{theorem}
We recall that an algebra $A$ is called \emph{iterated tilted} of type $Q$ if there exists a sequence of algebras $A=A_0$, $A_1$, $\ldots$, $A_n$, where $A_n$ is the path algebra $Q$, and a sequence of tilting modules $T^i_{A_i}$, for $0 \leq i < n$, such that $A_{i+1}=\End (T^i_{A_i})$, and  every $A_i$-indecomposable module $M$ satisfies $\Hom_{A_i} (T^i,M)=0$ or $\Ext^1_{A_i}(T^i,M)=0$.

%% file: an.tex
We will devote this section to the study of PHI gentle algebras.

\begin{definition}[gentle algebra \cite{ass-sko3}]
Let $A$ be an algebra with acyclic quiver $Q_A$. The algebra $A \cong KQ_A/I$ is called gentle if the bound quiver $(Q_A,I)$ has the following properties:
\begin{asparaenum} [\itshape i)]
\item each point of $Q_A$ is the source and the target of at most two arrows;
\item for each arrow $\alpha$ of $Q_A$, there is at most one arrow $\beta$ and one arrow $\gamma$ such that $\alpha\beta \notin I$ and $\gamma\alpha \notin I$;
\item for each arrow $\alpha$ of $Q_A$, there is at most one arrow $\delta$ and one arrow $\zeta$ such that $\alpha\delta \in I$ and $\zeta\alpha \in I$;
\item the ideal $I$ is generated by paths of length two.
\end{asparaenum}
\end{definition}

The following proposition, is probably well known, and we will use it to give a characterization of the gentle incidence algebra. Note that to be gentle is preserved by derived
equivalence, \cite{sch-zim}.

\begin{definition} [bypass]
A bypass in a quiver $Q$ is a pair of parallel paths $(\alpha, \gamma)$ where $\alpha$ is an arrow and $\gamma$ is a path distinct from $\alpha$ and parallel to it.
\end{definition}

\begin{remark}
The quiver of an incidence algebra has no bypass.
\end{remark}

\begin{proposition}\label{moninc}
Assume that $Q$ has no bypass and $KQ/I \cong KQ/I'$, for admissible ideals $I$ and $I'$. If a path $\gamma \in I$ then $\gamma \in I'$.
\end{proposition}
\begin{proof}
Since there is no bypass it follows that any isomorphism of $KQ/I \cong KQ/I'$ takes the class of arrow  $\alpha + I$ to $\lambda_{\alpha} \alpha + I'$, where $\lambda_{\alpha}$ is in $K^*$. It follows that the isomorphism takes the class of a path $\gamma + I$ to $\lambda_{\gamma}\gamma + I'$ where $\lambda_{\gamma}$ is in $K^*$. Therefore if $\gamma \in I$ then $\gamma \in I'$.
\end{proof}

This has the following corollary, which characterizes the gentle incidence algebras.

\begin{corollary}
If the incidence algebra $K\Delta \cong KQ/I$ is gentle, then $K\Delta$ is hereditary.  
\end{corollary}
\begin{proof}
The quiver associated with an incidence algebra has no bypass and
a gentle algebra have a presentation with ideal generating by quadratic monomial relations. Then by the proposition \ref{moninc} the generating ideal must be zero.

\end{proof}


We can state the following result:
\begin{corollary}
If a PHI algebra $K\Delta$ is gentle, then $K\Delta \cong KQ$, where $\overline{Q}=\mathbb{A}_n$ or $\overline{Q}=\widetilde{\mathbb{A}}_n$.
\end{corollary}
\begin{proof}
From the previous proposition, we conclude that $K\Delta$ is hereditary, that is, $K\Delta$ is isomorphic to $KQ$. This implies that $KQ$ is gentle, because $K\Delta$ is a gentle algebra, by hypothesis.

We assume that there exists a subquiver of $Q$ of the form:
\begin{center}
\begin{tikzcd}
&\bullet &\\
\bullet \ar[r, dash] &\bullet \ar[u, dash] \ar[r, dash]	&\bullet                 
\end{tikzcd}
\end{center}
with the edges oriented respecting the first condition of the definition of gentle algebra. We will analyze a case of orientation of these edges, the other cases are similar. Let
\begin{center}
\begin{tikzcd}
&\bullet &\\
\bullet	&\bullet \ar{l}{\alpha} \ar{u}{\theta} &\bullet \ar{l}{\beta}                 
\end{tikzcd}
\end{center}
By the fourth condition of the definition of gentle algebra, the ideal $I$ is generated by paths of length two. In the subquiver above, we have the paths $\beta\alpha$ and $\beta \theta$ both of length two. Thus, we have $\beta\alpha \in I$ or $\beta\theta \in I$. Since the algebra $KQ$ has the ideal $I$ empty, this case does not happen.

So we only have the two options below:

\begin{tikzcd} [column sep=small, row sep=small]
\bullet \ar[r, dash] &\bullet \ar[r, dash] &\cdots \ar[r, dash] &\bullet \ar[r, dash] &\bullet \quad \text{ou}
\end{tikzcd}
\begin{tikzcd}[column sep=small, row sep=small]
&\bullet \ar[r, dash] &\: \cdots \: \ar[r, dash] &\bullet &\\
&&&&\\
\bullet \ar[uur, dash, bend left] \ar[ddr, dash, bend right] &&&&\bullet \ar[uul, dash, bend right] \ar[ddl, dash, bend left]\\
&&&&\\
&\bullet \ar[r, dash] &\: \cdots \: \ar[r, dash] &\bullet &
\end{tikzcd}

We conclude that $\overline{Q_A}=\mathbb{A}_n$ or $\overline{Q_A}=\widetilde{\mathbb{A}}_n$.
\end{proof}

These two first results and the fact that the property of algebra is gentle is invariant by derived equivalence, see \cite{sch-zim}, we get these corollaries.

\begin{corollary}
If $K\Delta$ is a PHI algebra of type $\mathbb{A}_n$, then $K\Delta \cong KQ$, where $\overline{Q}=\mathbb{A}_n$.
\end{corollary}

\begin{corollary}
If $K\Delta$ is a PHI algebra of type $\widetilde{\mathbb{A}}_n$, then $K\Delta \cong KQ$, where $\overline{Q}=\widetilde{\mathbb{A}}_n$.
\end{corollary}

%% file: dn.tex
The characterization of the piecewise hereditary algebras of type $\mathbb{D}_{n}$ was done by Keller in the article \cite{kel}. In her thesis \cite{fern}, Elsa Fern\'andez obtained the same characterization through an alternative tool, the concept of trivial extension. We will rewrite the Fern\'andez's theorem  for the case of PHI algebras.

\begin{theorem}
Let $K\Delta \cong KQ/I$ be a PHI algebra of type $\mathbb{D}_n$ such that $\overline{Q}$ is not of tree type. Then $K\Delta$ or $K\Delta^{op}$ is isomorphic to one of the following algebras:

\NumTabs{2}
\begin{inparaenum}
\noindent 
\item
\begin{tikzcd}[column sep=tiny, row sep=tiny]
&&&\ramo &&&\\
&									&					& \bullet  \arrow[dll] \arrow[u, dash]	&&								&\\
&\bullet \arrow[rrrr, dotted, dash]	&					&										&&\bullet \arrow[ull] \arrow[dd]	&\\
&&&&&&\\
&\ramo \arrow[uu] & \ramo \arrow[l] &\cdots \arrow[l] &\ramo \arrow[l] &\ramo \arrow[l] &
\end{tikzcd}
\tab\item 
\begin{tikzcd}[column sep=tiny, row sep=tiny]
									&					&{\phantom{\bullet}}	&					&\\
									&					& \bullet  \arrow[dll]	&					&\\
\bullet \arrow[rrrr, dotted, dash]	&					&						&					& \bullet \arrow[ull] \arrow[dd]\\
&&&&\\
\ramo  \arrow[uu] 					& \ramo \arrow[l]	& \cdots \arrow[l]		&\ramo \arrow[l]	& \ramo \arrow[l]
\end{tikzcd}
\tab\item 
\begin{tikzcd}[column sep=tiny, row sep=tiny]
&&&&&&\\
&&&\bullet \arrow[dll] &&&\\
\ramo &\bullet \arrow[rrrr, dotted, dash] \arrow[l, dash] &&&&\bullet  \arrow[ull] \arrow[dll] & \ramo \arrow[l, dash]\\
&&&\bullet  \arrow[ull] &&&\\
&&&&&&
\end{tikzcd}
\tab\item  
\begin{tikzcd}[column sep=tiny, row sep=tiny]
									&& \bullet  \arrow[dll]	&&&\\
\bullet \arrow[rrrr, dotted, dash]	&&                      && \bullet  \arrow[ull] \arrow[dll]	& \ramo \arrow[l, dash]\\
									&& \bullet  \arrow[ull]	&&&
\end{tikzcd}
\tab\item 
\begin{tikzcd}[column sep=tiny, row sep=tiny]
		&													&& \bullet  \arrow[dll]	&&&\\
\ramo 	&\bullet \arrow[rrrr, dotted, dash] \ar[l, dash] 	&&						&& \bullet  \arrow[ull] \arrow[dll] & \\
		&                                                   && \bullet  \arrow[ull]	&&&
\end{tikzcd}

\end{inparaenum}
\end{theorem}
The notation $\ramo$, in the picture above, means that at the vertices  of the graph  we can attach a diagram of the form
\begin{center}
\begin{tikzcd}[column sep=small, row sep=small]
\bullet \arrow[r] & \bullet \arrow[r]  & \bullet \arrow[r] & \cdots \arrow[r] & \bullet \arrow[r] & \bullet 
\end{tikzcd}
with no new relations.
\end{center}

%% file: e6.tex
The first effort to classify piecewise hereditary algebras of type $\mathbb{E}_6$ was made by Happel, in \cite{hap1}. Years later, in her thesis \cite{fern}, Elsa Fern\'andez obtained some updates in relation to this classification, determining a complete description for this class of algebras. This problem was also treated computationally by the Roggon's project \cite{rog}.
We rewrite the Fern\'andez's theorem restricted to the PHI algebras of type $\mathbb{E}_6$.

\begin{theorem}
Let $K\Delta \cong KQ/I$ be a PHI algebra of type $\mathbb{E}_6$ such that $\overline{Q}$ is not of tree type. Then $K\Delta$ or $K\Delta^{op}$ is isomorphic to one of the following algebras:

\NumTabs{2}
\begin{inparaenum}
\noindent 
\item
\begin{tikzcd}[column sep=tiny, row sep=tiny]
							&					& {\phantom{\bullet}}	&							\\
							& \bullet \arrow[dl]	& \bullet  \arrow[l]		&							\\
\bullet \arrow[rrr, dotted, dash]	&					&						& \bullet \arrow[ul] \arrow[dl]	\\
							& \bullet  \arrow[ul]	& \bullet \arrow[l]		&                                 
\end{tikzcd}
\tab\item 
\begin{tikzcd}[column sep=tiny, row sep=tiny]
								&					&					& \bullet \arrow[d, dash]	&								\\
								& \bullet \arrow[dl] 	&                     		& \bullet  \arrow[ll]		&								\\
\bullet \arrow[rrrr, dotted, dash] 	&                     		&                     		&                     			&  \bullet \arrow[ul] \arrow[dll]	\\
								&                     		& \bullet  \arrow[ull]	&						&
\end{tikzcd}
\tab\item 
\begin{tikzcd}[column sep=tiny, row sep=tiny]
                         				& \bullet          				&		                                 	\\
                         				& \bullet \arrow[dl]  \arrow[u]	&           		                      	\\
\bullet \arrow[rr, dotted, dash]	&                            			&  \bullet \arrow[ul]  \arrow[dl]	\\
                           			& \bullet  \arrow[ul]       		&                       			          	\\
                           			& \bullet \arrow[u, dash]		&                                  
\end{tikzcd}
\tab\item
\begin{tikzcd}[column sep=tiny, row sep=tiny]
                           			& {\phantom{\bullet}}	&								&						\\
                           			& \bullet \arrow[dl]        	&								&                  			\\
\bullet \arrow[rr, dotted, dash]	&                            		&  \bullet \arrow[ul]  \arrow[dl]	& \bullet \arrow[l, dash]	\\
                      				& \bullet  \arrow[ul]		&								&						\\
                           			& \bullet \arrow[u, dash]	&								&
\end{tikzcd}
\tab\item 
\begin{tikzcd}[column sep=tiny, row sep=tiny]
\bullet  \arrow[d] & \bullet  \arrow[l] \arrow[dl, dotted, dash] \arrow[d]  	& \bullet \arrow[l]  \arrow[d] \arrow[dl, dotted, dash]    \\
\bullet           	& \bullet \arrow[l]                              					&  \bullet \arrow[l]                                
\end{tikzcd}
\tab\item 
\begin{tikzcd}[column sep=tiny, row sep=tiny]
\bullet  \arrow[d] & \bullet  \arrow[l] \arrow[dl, dotted, dash] \arrow[d] \arrow[r] \arrow[dr, dotted, dash]	& \bullet \arrow[d] 	\\
\bullet           	& \bullet \arrow[l]  \arrow[r]                            										& \bullet                                 
\end{tikzcd}
\tab\item
\begin{tikzcd}[column sep=tiny, row sep=tiny]
								& \bullet \arrow[dl] 	&                     		& \bullet  \arrow[ll]	&                                  			&						\\
\bullet \arrow[rrrr, dotted, dash] 	&                     		&                     		&                     		&  \bullet \arrow[ul] \arrow[dll]	& \bullet \arrow[l, dash]	\\
								&                     		& \bullet  \arrow[ull]	&                     		&                                 			&
\end{tikzcd}
\end{inparaenum}

In the diagrams when there is a non oriented edge we can have two  quivers, one for each orientation.
\end{theorem}

%% file: e7.tex
The class of iterated tilted algebras of type $\mathbb{E}_7$ has an important relation with a class of representation-finite trivial extensions.
Fern\'andez, in her thesis \cite{fern}, describes in detail these trivial extensions.
We recall the definition of trivial extension of an algebra.

\begin{definition}[trivial extension]
Let $A$ be an $K$-algebra.
The trivial extension of $A$ is the algebra $T(A)=A \ltimes D(A)$, where $D(A)=\Hom_K (A,K)$, whose underlying of $K$-vector space is $A \times D(A)$ with multiplication  given by:
\begin{gather*}
(a,f)(b,g)=(ab,ag+fb), \quad \text{for any }a,b \in A \text{ and } f,g \in D(A).
\end{gather*}
\end{definition}

In this work, we will consider the trivial extensions which are finite representation type. 
The trivial extension is a symmetric algebra. The self injective algebras of finite representation type are divided into classes, called Cartan classes.
The \emph{Cartan class} of an algebra $B$ self injective of finite representation type was introduced by Riedtmann in \cite{rie}. More details, we refer the reader to \cite{rie}, \cite{rie1} and \cite{rie2}. 

The Cartan class is defined, briefly, as follows:

Consider $(\Gamma, \tau)$ the \emph{stable Auslander-Reiten quiver} of the algebra $B$, which is self injective of finite representation type. 
Riedtmann showed that $(\Gamma, \tau)$ is isomorphic to $\mathbb{Z}Q/G$ for an automorphism $G$ of the translation quiver $\mathbb{Z}Q$, where the $G$-action in $\mathbb{Z}Q$ is admissible that is each orbit of $G$ finds $\{ x \} \cup x^-$ in at most one vertex and finds $\{ x \} \cup x^+$ in at most one vertex for each $x$ of $(\mathbb{Z}Q)_0$, she also showed that $Q$ is of Dynkin type. The Dynkin type of $Q$ is uniquely determined and it is called the Cartan class of $B$.

In this next result, we will have an association of trivial extensions finite representation type with iterated tilted algebras.
\begin{theorem}[Assem-Happel-Rold\'an \cite{ass-hap-rol}] \label{AHR}
Let $A$ be an algebra. The following conditions are equivalent:
\begin{asparaenum} [\itshape a)]
\item $T(A)$ is representation-finite of Cartan class $Q$;
\item $A$ is iterated tilted algebra of Dynkin type $Q$.
\end{asparaenum}
\end{theorem}

Recall that an algebra  $A$ is schurian if the following condition is satisfied:
$
\dime_K (\Hom_A (P,P')) \leq 1
$
for any indecomposable projectives $P,P'$.

The focus of the study is the quiver with relations of the trivial extension $T(A)$ of a \emph{schurian algebra} $A \cong KQ_A/I$. Note that all incidence algebras are schurian. 

We define now the notion of maximal path in an algebra $KQ/I$ with $I$ admissible. We remark that when we say that a path is in an algebra of the form $KQ/I$ 
we mean that it is a path in $Q$ whose class is not zero in $KQ/I$.

\begin{definition}[maximal path]
A path $\gamma$ in $KQ_A/I$ will be called maximal if $\gamma \neq 0$ and $\alpha \gamma = 0 = \gamma \alpha$ for every arrow $\alpha$ of $(Q_A)_1$.  
\end{definition}

The quiver of $T(A)$ is described in the following theorem.

\begin{theorem}[Fern\'andez \cite{fern}]
If $A \cong KQ_A/I$ is a schurian algebra, then the quiver of $T(A)$ is given by:
\begin{asparaenum}[\itshape a)]
\item $(Q_{T(A)})_0 = (Q_A)_0$,
\item $(Q_{T(A)})_1 = (Q_A)_1 \cup \{ \beta_{\gamma_1}, \dotsc, \beta_{\gamma_t} \}$, where $\{ \gamma_1, \dotsc, \gamma_t \}$ 
is a maximal set of linearly independent maximal paths. Moreover, for each $i$, $\beta_{\gamma_i}$ is an arrow such that $s(\beta_{\gamma_i})=t(\gamma_i)$ and $t(\beta_{\gamma_i})=s(\gamma_i)$.   
\end{asparaenum}
\end{theorem}

It is necessary to describe the relations of the presentation associated with $T(A)$, for this we will need some definitions.

\begin{definition}[elementary cycle]
Let $M = \{ \gamma_1, \dotsc, \gamma_t \}$ be a maximal set of linearly independent maximal paths. We say that a oriented cycle $C$ of $Q_{T(A)}$ is elementary if $C = \beta_{\gamma_i}\alpha_1\alpha_2\dotsc\alpha_n$, where $\gamma_i \in M$ and $\alpha_1 \dotsc \alpha_n=k \gamma_i$, for some $k \in K^*$. 
\end{definition} 

Note that the length of elementary cycle $C$ is at least two. 

We will say that the supplement of $\gamma$ in $C$ is the trivial path $e_{s(\gamma)}$, if $s(\gamma)=t(\gamma)$, otherwise the supplement of $\gamma$ is the path in $C$ whose origin is the terminus of $\gamma$ and it has terminus the origin of $\gamma$.

\begin{theorem}[Fern\'andez \cite{fern}] \label{relacao}
Let $A \cong KQ_A/I_A$ be a schurian triangular algebra such that parallel paths in $Q_A$ are equal in $A$. Then $T(A) = KQ_{T(A)}/I_{T(A)}$ where the admissible ideal $I_{T(A)}$ is generated by:
\begin{asparaenum}[\itshape a)]
\item the paths consisting of $n + 1$ arrows containing all arrows of an elementary cycle of length $n$,
\item the paths whose arrows does not belong to a unique elementary cycle,
\item the difference $\gamma - \gamma'$ of paths $\gamma, \gamma'$ with the same origin and the same endpoint and having a common supplement in elementary cycles of $Q_{T(A)}$.
\end{asparaenum}
\end{theorem}

Note that every path $\gamma$ in $T(A)$ is contained in an elementary cycle $C$. 

For future references in the text, the relations defined by items a), b) and c) are called relations of type 1, 2, and 3, respectively. 
Observe that the relations of type 1 and 2 are monomial.

We will use an important procedure on a trivial extension $T(A)$. The two definitions below are required in order to give a description 
of  this method.

\begin{definition} [cutting set \cite{fern-plat}]
Let $T(A) \cong KQ_{T(A)}/I_{T(A)}$ be a trivial extension and let $\Sigma$ be a set of arrows of $Q_{T(A)}$. We say that $\Sigma$ is a cutting set if it consists of exactly one arrow in each  elementary cycle of $T(A)$.
\end{definition}

\begin{definition}[\cite{fern-plat}]
Let $T(A) \cong KQ_{T(A)}/I_{T(A)}$ be a trivial extension given by a quiver $KQ_{T(A)}$ and an admissible ideal $I_{T(A)}$. We say that $A'$ is defined by the cutting set $\Sigma$ of $T(A)$, if $A'$ is isomorphic to  $KQ_{T(A)}/ <I_{T(A)} \cup \Sigma>$.
\end{definition}

Looking at the Theorem \ref{relacao} and the definition of the cutting set, we see that the relations of  type $1$, are eliminate with the cutting and they have no influence 
on the relations of the ideal which define $A'$. Therefore, in the study of these algebras, we will pay attention only  relations of type $2$ and type $3$. We can now state the following theorem of Fern\'andez.

\begin{proposition}[Fern\'andez \cite{fern}]
Let $A$ be a schurian triangular algebra. Then an algebra $A'$ is defined by a cutting set of $T(A)$ if and only if $T(A) \cong T(A')$.
\end{proposition}

Let $\Gamma$ be a representation-finite trivial extension of Cartan class $\mathbb{E}_7$. All algebras $A$ defined by a cutting set of
$\Gamma$ have their trivial extensions isomorphic to $\Gamma$. Therefore, the trivial extension of $A$ is representation-finite of Cartan class $\mathbb{E}_7$, implying that the algebra $A$ is iterated tilted of type $\mathbb{E}_7$ by Theorem \ref{AHR}.

We want to classify all PHI algebras such that its trivial extension is a given selfinjective finite representation type algebra with Cartan class $\mathbb{E}_7$ 
In particular we can assume  that $A =KQ/I$ where $Q$ has no bypass and  $I$ is generated by the all commutative relations. 
 
We state now the following lemma.

\begin{lemma} \label{cortephia}
Let $A \cong KQ_A/I_A$ be a schurian triangular algebra such that parallel paths in $Q_A$ are equal in $A$. Let  $A' \cong KQ_A'/ I'$ be an algebra defined by a cutting set of $T(A)$. If an arrow of the cutting set belongs to two elementary cycles of $Q_{T(A)}$ then there is, at least, a commutativity relation in $I'$.
\end{lemma}

\begin{proof}

 We give a proof only in the case of a trivial extension $T(A)$ associated with a quiver which is the union of the vertices and arrows belonging to  two elementary cycles, as in the picture.
 The general case reduces to this case.
 
We use the following notation: $C_1=\theta_1\theta_2\dotsc\theta_l\alpha_{l+1}\dotsc\alpha_n$ and \\  $C_2=\theta_1\theta_2\dotsc\theta_l\alpha'_{l+1}\dotsc\alpha'_m$.
\begin{center}
\begin{tikzcd} [column sep=large, row sep=large]
\bullet \ar{d}[swap]{\alpha_n} &\cdots \ar{l}[swap]{\alpha_{n\menas 1}}	&\bullet \ar{l}[swap]{\alpha_{l+2}}\\
\bullet \ar{r}{\theta_1} &\cdots \ar{r}{\theta_l} &\bullet \ar{u}[swap]{\alpha_{l+1}} \ar{d}{\alpha'_{l+1}}\\
\bullet \ar{u}{\alpha'_m} &\cdots \ar{l}{\alpha'_{m\menas 1}} &\bullet \ar{l}{\alpha'_{l+2}}
\end{tikzcd}
\end{center}
with  relations of type  $2$ and $3$:
\begin{align*}
&\alpha'_m \theta_1 \dotsc \theta_l \alpha_{l+1}; \: \alpha_n \theta_1 \dotsc \theta_l \alpha'_{l+1} \\
&\alpha_{l+1} \dotsc \alpha_n - \alpha'_{l+1} \dotsc \alpha'_m
\end{align*}

Now, let $\Sigma = \{ \theta_1, \dotsc , \theta_l  \}$ be a cutting set such that $\theta_i$ belonging to the two elementary cycles $C_1$ and $C_2$.

Therefore the presentation of $A'$ has only the relation $\alpha_{l + 1}$ $\alpha_{l + 2} \dotsc \alpha_n$ $\menas$ $\alpha'_{l + 1}$ $\alpha'_{l + 2} \dotsc \alpha'_m$, as desired.
\end{proof}

A question arises on the existence of a cutting set of trivial extension that defines an incidence algebra. We come across various forms of trivial extension diagrams, and we can find some particular patterns in the cutting sets that define incidence algebras. We join these patterns  in some lemmas that we will call cutting lemmas. Before stating these lemmas, we will introduce some necessary concepts.

For each vertex $h$ of quiver $Q_{T(A)}$, let $C_h$ be the set of all elementary cycles $C$,  with origin and terminus in $h$.

\begin{example} \label{corte1}
Let $T(A)$ be a trivial extension whose quiver $Q_{T(A)}$ is the following:

\begin{center}
\begin{tikzcd} [column sep={1.5cm,between origins} , row sep={1.5cm,between origins} ] 
&&&\bullet \ar{dl}[swap]{\theta_1} \ar{dr}{\delta_1} &&&\\
&&\bullet \ar{dl}[swap]{\theta_2} &\rvdots \ar{u}{\lambda_{m_1}} &\bullet \ar{dr}{\delta_2} &&\\
&\iddots \ar{dl}[swap]{\theta_{n_1}}&&\bullet \ar{u}{\lambda_1}	&&\ddots \ar{dr}{\delta_{n_2}} &\\
\bullet \ar{r}{\mu_1} &\cdots \ar{r}{\mu_{m_2}} &\bullet \ar{r}{\beta'}	&h \ar{u}{\alpha} \ar[swap]{d}{\beta}	&\bullet \ar{l}[swap]{\alpha'} &\cdots \ar{l}[swap]{\gamma_{m_3}} &\bullet \ar{l}[swap]{\gamma_1}\\
&\ddots \ar{ul}{\rho_{n_3}}	&&\bullet \ar{d}[swap]{\eta_1} &&\iddots \ar{ur}[swap]{\sigma_{n_4}}&\\
&&\bullet \ar{ul}{\rho_2} &\rvdots \ar{d}[swap]{\eta_{m_4}}	&\bullet \ar{ur}[swap]{\sigma_2} &&\\
&&&\bullet \ar{ul}{\rho_1} \ar{ur}[swap]{\sigma_1} &&&
\end{tikzcd}
\end{center}
assume that  $C_h = \{ C_1, C_2, C_3, C_4 \}$. Then there are only two cutting sets defining two incidence algebras, which are $\{\alpha, \beta\}$ and
$\{\alpha', \beta`\}$.
\end{example}

\begin{proof}

We use the notation $\lambda = \lambda_1 \dotsc \lambda_{m_1}$, $\theta=\theta_1 \dotsc \theta_{n_1}$, $\mu = \mu_1 \dotsc \mu_{m_2}$, $\eta = \eta_1 \dotsc \eta_{m_4}$, $\rho = \rho_1 \dotsc \rho_{n_3}$, $\delta = \delta_1 \dotsc \delta_{n_2}$, $\gamma = \gamma_1 \dotsc \gamma_{m_3}$ and $\sigma = \sigma_1 \dotsc \sigma_{n_4}$. 

The relations of type $2$ and $3$ are:
\begin{align*}
& \alpha \lambda \theta - \beta \eta \rho; \; \theta \mu \beta' - \delta \gamma \alpha'; \; \alpha \lambda \delta - \beta \eta \sigma; \; \rho \mu \beta' - \sigma \gamma \alpha'; \\
& \theta_{n_1} \mu \beta' \beta; \; \rho_{n_3} \mu \beta' \alpha; \; \delta_{n_2} \gamma \alpha' \beta; \; \sigma_{n_4} \gamma \alpha' \alpha ; \\
& \beta' \alpha \lambda \delta_1; \; \alpha' \alpha \lambda \theta_1; \; \beta' \beta \eta \sigma_1; \; \alpha' \beta \eta \rho_1
\end{align*}

We consider the cutting set with the arrows $\alpha$ and $\beta$. This cutting set define the quiver with commutativity relations $\theta \mu \beta' - \delta \gamma \alpha'$ and $\rho \mu \beta' - \sigma \gamma \alpha'$.

The second incidence algebra with is defined by cutting set $\{\alpha' ,\beta' \}$, is isomorphic to quiver with commutativity relations $\alpha \lambda \theta - \beta \eta \rho$ and $\alpha \lambda \delta - \beta \eta \sigma$.
\end{proof}

In the section \ref{sectame}, we present two more cutting lemmas \ref{cortenphia} and \ref{corteAn}.

These lemmas inspired the search for a general algorithm for finding good cutting sets. We have been able to go one step further, in a computational way, we elaborate a program that shows exactly the cutting sets of given trivial extension that define incidence algebras.
In addition, we implemented this algorithm in the site \url{http://www.ime.usp.br/~celovg/}. 
In the appendix, we present the source code of the program. We briefly describe the ideas used in the program.

The first step is to name the arrows of quiver associated with the given trivial extension. We start labeling by  $\alpha_1$, $\alpha_2$ , $\alpha_3, \dotsc$ the arrows which are in the support
of relations  of type $2$.
The other arrows are named arbitrarily.

In the next step we put the necessary information from the trivial extension into the program. These initial data are:
\begin{itemize}
\item number of relations of type $2$;
\item number of elementary cycles;
\item number of different arrows in these relations of type $2$.
\end{itemize}
Then the program displays a table with checkboxes. In this table, the columns represent the relations of type $2$ and the elementary cycles, respectively, and the lines represent the arrows involved in these elements.
The user clicks on checkbox if the arrow is in the relation(s) or if it belongs to the cycle(s), and leaves unchecked otherwise. Press the ``Ready!'' button, and then the solution is shown.
The solution is the cutting sets that define incident algebras.

Fern\'andez's thesis \cite{fern} showed all the representation-finite trivial extension of Cartan class $\mathbb{E}_7$, totalling $72$.
Thus, for each trivial extension of the list $\mathbb{E}_7$, we look for the cutting sets that define incidence algebras.

Using the program we can see the following:

\begin{theorem} \label{teoe7}
Let $K\Delta \cong KQ/I$ be a PHI algebra of type $\mathbb{E}_7$ such that $\overline{Q}$ is not of tree type. Then $K\Delta$ or $K\Delta^{op}$ is isomorphic to one of the following algebras:

\NumTabs{2}
\begin{inparaenum}

\noindent 
\item
\begin{tikzcd}[column sep=tiny, row sep=tiny]
                            				& \bullet \arrow[dl] 	& 					& \bullet  \arrow[ll]	&\\
\bullet \arrow[rrrr, dotted, dash]	&                     		& 					&                         	& \bullet \arrow[ul] \arrow[dd] \\
								&					&					&					&\\
\bullet \arrow[uu]     				&					& \bullet \arrow[ll]	& 					& \bullet \arrow[ll]
\end{tikzcd}
\tab\item 
\begin{tikzcd}[column sep=tiny, row sep=tiny]
&&\bullet \arrow[dl] &&\\
\bullet &\bullet \arrow[l] \ar[rr, dotted, dash] &&\bullet \arrow[dd] \arrow[ul] &\\
&&&&\\
&\bullet \arrow[uu]	&\bullet \arrow[l] &\bullet \arrow[l] & 
\end{tikzcd}
\tab\item 
\begin{tikzcd}[column sep=tiny, row sep=tiny]
&&&\bullet \arrow[dll] &&\\
\bullet \ar[d] &\bullet \ar[l] \arrow[rrrr, dotted, dash]  &&&&\bullet \arrow[ull] \arrow[ddl] \\
\bullet &&&&&\\
&&\bullet \arrow[uul] &&\bullet \arrow[ll] & 
\end{tikzcd}
\tab\item 
\begin{tikzcd}[column sep=tiny, row sep=tiny]
&&\bullet \arrow[dll] && \\
\bullet \ar[dd] \arrow[rrrr, dotted, dash] &&&&\bullet \arrow[ull] \arrow[ddl] \\
&&&& \\
\bullet &\bullet \arrow[uul] &&\bullet \arrow[ll] & \bullet \ar[uu]
\end{tikzcd}
\tab\item 
\begin{tikzcd}[column sep=tiny, row sep=tiny]
				&					& 													& \bullet  			&\\
\bullet \arrow[dd]	&  					& \bullet \arrow[ll] \arrow[rr] \arrow[dd, dotted, dash]	&                         	& \bullet \arrow[ul] \arrow[dd] \\
				&					&													&					&\\
\bullet \arrow[rr]	&					& \bullet 											& 					& \bullet \arrow[ll]
\end{tikzcd}
\tab\item 
\begin{tikzcd}[column sep=tiny, row sep=tiny]
				&					& 								& \bullet \arrow[ld]	&\\
\bullet \arrow[dd]	&  					& \bullet \arrow[rr, dotted, dash]	&                         	& \bullet \arrow[ul] \arrow[dd] \\
				&					&								&					&\\
\bullet \arrow[rr]	&					& \bullet \arrow[uu]				& 					& \bullet \arrow[ll]
\end{tikzcd}
\tab\item 
\begin{tikzcd}[column sep=tiny, row sep=tiny]
\bullet \arrow[dr]	&					& \bullet \arrow[ll] \arrow[ld, dotted, dash] \arrow[rd]	& 							& \bullet \\
				& \bullet 			& 													& \bullet \arrow[dd] \arrow[ur]	&		\\
				&					&													&							&		\\
				& \bullet \arrow[uu]	& 													& \bullet \arrow[ll]			& 
\end{tikzcd}
\tab\item 
\begin{tikzcd}[column sep=tiny, row sep=tiny]
\bullet 			& \bullet \arrow[d] \arrow[rr] \arrow[ddrr, dotted, dash]	& &\bullet \arrow[dd] 	\\
				& \bullet \arrow[d]									& &						\\
\bullet \arrow[uu]	& \bullet \arrow[l] \arrow[rr]							& & \bullet
\end{tikzcd}
\tab\item 
\begin{tikzcd}[column sep=tiny, row sep=tiny]
\bullet \arrow[r] &\bullet \arrow[dd, dotted, dash] &&\bullet \arrow[dd] \\
&&& \\
\bullet \arrow[uu] &\bullet \arrow[l] \arrow[rr] &&\bullet \ar[uull]\\
&\bullet \ar[u] &&
\end{tikzcd}
\tab\item 
\begin{tikzcd}[column sep=tiny, row sep=tiny]
&&\bullet \arrow[dr] &\\
\bullet \arrow[r] &\bullet \arrow[dd, dotted, dash] &&\bullet \arrow[dd]	\\
&&&\\
\bullet \ar[r] &\bullet \arrow[luu] \arrow[rr] &&\bullet \arrow[uull]
\end{tikzcd}
\tab\item 
\begin{tikzcd}[column sep=tiny, row sep=tiny]
				& 													& \bullet \arrow[dr] 	& 					\\
\bullet 			& \bullet \arrow[d] \arrow[ddrr, dotted, dash] \arrow[ur]	& 					& \bullet \arrow[dd] 	\\
				& \bullet \arrow[d]									& 					&					\\
				& \bullet \arrow[luu] \arrow[rr]							& 					& \bullet	
\end{tikzcd}
\tab\item 
\begin{tikzcd}[column sep=tiny, row sep=tiny]
\bullet \arrow[dd] &&\bullet \arrow[ll] \arrow[ddr] &\bullet \\
&&& \\
\bullet \ar[d] \arrow[urur, dotted, dash] &&&\bullet \arrow[dl] \arrow[uu]	\\
\bullet &&\bullet \arrow[ull] & 
\end{tikzcd}
\tab\item 
\begin{tikzcd}[column sep=tiny, row sep=tiny]
\bullet \arrow[dd]	& 					& \bullet \arrow[ddr, dotted, dash] 	& \bullet \arrow[l]			\\
				& \bullet \arrow[ur] 	& 								& 							\\
\bullet \arrow[ur]	& 					&								& \bullet \arrow[dl] \arrow[uu]	\\
				& 					& \bullet \arrow[ull]				& 
\end{tikzcd}
\tab\item 
\begin{tikzcd}[column sep=tiny, row sep=tiny]
\bullet \arrow[rr] &&\bullet \arrow[ddd, dotted, dash] &&&\bullet \arrow[dddl] \\
&&&&& \\
&&&&& \\
\bullet \ar[r] &\bullet \ar[r] &\bullet \arrow[lluuu] \arrow[rr] &&\bullet \arrow[uuull] & 
\end{tikzcd}
\tab\item 
\begin{tikzcd}[column sep=tiny, row sep=tiny]
\bullet 	&& \bullet \arrow[d] \arrow[rrr]	&& 								 	& \bullet \arrow[dddl]	\\
		&& \bullet \arrow[d]				&& 									& 					\\
		&& \bullet \arrow[d]				&&									&					\\
		&& \bullet \arrow[lluuu] \arrow[rr]	&& \bullet \arrow[uuull, dotted, dash]	& 
\end{tikzcd}
\tab\item 
\begin{tikzcd}[column sep=tiny, row sep=tiny]
					&& 												 	& & \bullet 			 			\\
\bullet \arrow[drr]	&& \bullet \arrow[ll] \arrow[drr] \arrow[d, dotted, dash]	& & \bullet \arrow[u]				\\
\bullet \arrow[u]		&& \bullet 											& & \bullet \arrow[ll] \arrow[u]
\end{tikzcd}
\tab\item 
\begin{tikzcd}[column sep=tiny, row sep=tiny]
\bullet \arrow[dd]	 \arrow[rrr]	&					& 							& \bullet \arrow[ddl, dotted, dash]	& \bullet \arrow[dd] 	\\
							& \bullet \arrow[ul]	& 							&								&					\\
\bullet						&					& \bullet \arrow[ul] \arrow[rr]	&								& \bullet \arrow[uul]
\end{tikzcd}
\tab\item 
\begin{tikzcd}[column sep=tiny, row sep=tiny]
&\bullet \arrow[ddl] &&\bullet \ar[ll] \\
&&&\\
\bullet \arrow[rr, dotted, dash] && \bullet \ar[uul] \arrow[d] &\bullet \ar[uu] \ar[l] \arrow[uull, dotted, dash] \\
\bullet \arrow[u] &&\bullet \arrow[ll] &
\end{tikzcd}
\tab\item  
\begin{tikzcd}[column sep=tiny, row sep=tiny]
\bullet &\bullet \arrow[ddl] &&\bullet \arrow[ll] &	\\
&&&&\\
\bullet \ar[uu] \arrow[rrrr, dotted, dash] &&&&\bullet \arrow[d] \arrow[uulll]	\\
\bullet \arrow[u] &&&&\bullet \arrow[llll]
\end{tikzcd}
\tab\item 
\begin{tikzcd}[column sep=tiny, row sep=tiny]
&&\bullet \arrow[rr] &&\bullet \arrow[ddll, dotted, dash] \\
&&&&\\
\bullet \arrow[rr] &&\bullet \arrow[rr] \arrow[uu] &&\bullet \arrow[d] \arrow[uu] \\
\bullet &&&&\bullet \arrow[llll]
\end{tikzcd}
\tab\item 
\begin{tikzcd}[column sep=tiny, row sep=tiny]
&\bullet \arrow[dddl] &&&\bullet \ar[lll] \\
&&&&\\
&&&&\\
\bullet \arrow[rr, dotted, dash] &&\bullet \arrow[dl] \arrow[uuul] &\bullet \ar[l] &\bullet \ar[uuulll, dotted, dash] \ar[uuu] \ar[l] \\
&\bullet \arrow[ul] &&&
\end{tikzcd}
\tab\item  
\begin{tikzcd}[column sep=tiny, row sep=tiny]
&&\bullet \arrow[dddll] &&\bullet \arrow[ll] \\
\bullet &&&&\bullet \ar[dd] \\
&&&&\\
\bullet \arrow[rrrr, dotted, dash] \ar[uu] &&&&\bullet \arrow[dll] \arrow[uuull] \\
&&\bullet \arrow[ull] && 
\end{tikzcd}
\tab\item 
\begin{tikzcd}[column sep=tiny, row sep=tiny]
&\bullet \arrow[rrr] &&&\bullet \\
&&&&\\
&&&&\\
\bullet \arrow[r] &\bullet \arrow[rr] \ar[uuu] \arrow[uuurrr, dotted, dash] &&\bullet \arrow[r] &\bullet \arrow[d] \arrow[uuu]	\\
&&&& \bullet 
\end{tikzcd}
\tab\item 
\begin{tikzcd}[column sep=tiny, row sep=tiny]
&\bullet \arrow[dddr] \arrow[dddrr, dotted, dash] \arrow[rr] &&\bullet \arrow[ddd] &\\
&&&&\\
&&&&\\
\bullet \arrow[dr] \ar[uuur] \arrow[rr, dotted, dash] &&\bullet \arrow[r] &\bullet \arrow[r] &\bullet \\
&\bullet \arrow[ur] &&&
\end{tikzcd}
\tab\item  
\begin{tikzcd}[column sep=tiny, row sep=tiny]
&\bullet \arrow[dddr] \arrow[dddrr, dotted, dash] \arrow[rr] &&\bullet \arrow[ddd] &\\
&&&&\\
&&&&\\
\bullet \arrow[dr] \ar[uuur] \arrow[rr, dotted, dash] &&\bullet \arrow[r] &\bullet &\\
\bullet \arrow[u] &\bullet \arrow[ur] &&&
\end{tikzcd}
\tab\item 
\begin{tikzcd}[column sep=tiny, row sep=tiny]
&&\bullet \arrow[dddrr] \arrow[rr] &&\bullet \\
\bullet \ar[d] &&&&\\
\bullet \ar[d] &&&&\\
\bullet \arrow[drr] \arrow[uuurr] &&&&\bullet \arrow[llll, dotted, dash] \\
&&\bullet \arrow[urr] && 
\end{tikzcd}
\tab\item 
\begin{tikzcd}[column sep=tiny, row sep=tiny]
&&\bullet \ar[dd] &\bullet \ar[l] \ar[dd] \arrow[ddl, dotted, dash] \\
&&&\\
\bullet &\bullet \ar[l] &\bullet \ar[l] &\bullet \ar[l] \\
&&&\bullet \ar[u]
\end{tikzcd}
\tab\item  
\begin{tikzcd}[column sep=tiny, row sep=tiny]
				& 							& \bullet \arrow[d] 									&					& \bullet \arrow[ll]			\\
\bullet \arrow[rr]	& 				 			& \bullet	\arrow[rr, dotted, dash] \arrow[dl, dotted, dash]	& 					& \bullet \arrow[u] \arrow[dl]	\\
				& \bullet \arrow[ul] \arrow[rr]	& 													& \bullet \arrow[ul]	& 
\end{tikzcd}
\tab\item 
\begin{tikzcd}[column sep=tiny, row sep=tiny]
		& 							& \bullet 					&								& \bullet \arrow[ll]			\\
\bullet 	& 				 			& \bullet	\arrow[rr] \arrow[dl]	& 								& \bullet \arrow[u] \arrow[dl]	\\
		& \bullet \arrow[ul] \arrow[rr]	& 							& \bullet \arrow[ul, dotted, dash]	& 
\end{tikzcd}
\tab\item 
\begin{tikzcd}[column sep=tiny, row sep=tiny]
				& 							&  													& \bullet \arrow[dl]	& 							\\
\bullet \arrow[rr] 	& 				 			& \bullet	\arrow[rr, dotted, dash] \arrow[dl, dotted, dash]	& 					& \bullet \arrow[d] \arrow[ul]	\\
				& \bullet \arrow[ul] \arrow[r]	& \bullet	\arrow[u]									&					& \bullet \arrow[ll]
\end{tikzcd}
\tab\item  
\begin{tikzcd}[column sep=tiny, row sep=tiny]
		& 							&  								& \bullet 	& 							\\
\bullet	& 				 			& \bullet	\arrow[rr] \arrow[dl]		& 			& \bullet \arrow[d] \arrow[ul]	\\
		& \bullet \arrow[ul] \arrow[r]	& \bullet	\arrow[u, dotted, dash]	&			& \bullet \arrow[ll]
\end{tikzcd}
\tab\item 
\begin{tikzcd}[column sep=tiny, row sep=tiny]
&&&\bullet \arrow[dr] & \\
\bullet \arrow[d] &&\bullet \arrow[ll] \arrow[ur] \arrow[dr] && \bullet \arrow[ll, dotted, dash] \\
\bullet \arrow[d] \arrow[rur, dotted, dash] &&&\bullet \arrow[ur] \arrow[lll] &\\
\bullet &&&&
\end{tikzcd}
\tab\item 
\begin{tikzcd}[column sep=tiny, row sep=tiny]
				&					&								& \bullet \arrow[dr] 			& 					\\
\bullet \arrow[dd]	&  					& \bullet	\arrow[ddr, dotted, dash]	&  							& \bullet \arrow[ll]	\\
				& \bullet	\arrow[ur]  	& 								& 							& 					\\
\bullet \arrow[ur]	& 					& 								& \bullet \arrow[uur] \arrow[lll]&
\end{tikzcd}
\tab\item  
\begin{tikzcd}[column sep=tiny, row sep=tiny]
					&		&																& \bullet \arrow[dr] 			& 								\\
\bullet \arrow[ddr]	&  		& \bullet	\arrow[ddl, dotted, dash] \arrow[ll] \arrow[ddr] \arrow[ur]	&  							& \bullet \arrow[ll, dotted, dash]	\\
					&   		& 																& 							& 								\\
\bullet \arrow[uu]		& \bullet	& 																& \bullet \arrow[uur] \arrow[ll]	&
\end{tikzcd}
\tab\item 
\begin{tikzcd}[column sep=tiny, row sep=tiny]
&&\bullet \arrow[rr] &&\bullet \\
&&&&\\
\bullet \arrow[dd] \arrow[rr] &&\bullet \arrow[rr] \ar[uu] \arrow[uurr, dotted, dash] && \bullet \arrow[uu] \arrow[ddll] \\
&&&& \\
\bullet \arrow[rr] &&\bullet \arrow[uull, dotted, dash]	&&
\end{tikzcd}
\tab\item 
\begin{tikzcd}[column sep=tiny, row sep=tiny]
&& \bullet \arrow[ddll]	&&&\bullet \arrow[lll] \\
&&&&&\\
\bullet \arrow[rrrr, dotted, dash] &&&& \bullet \arrow[uull] \arrow[ddll] &\bullet \ar[uu] \ar[l] \arrow[uulll, dotted, dash] \\
&&&&& \\
\bullet \arrow[rr] &&\bullet \arrow[uull] &&&
\end{tikzcd}
\tab\item  
\begin{tikzcd}[column sep=tiny, row sep=tiny]
&&&\bullet \arrow[ddll] &&\bullet \arrow[ll] \\
&&&&&\\
\bullet &\bullet \arrow[l] \arrow[rrrr, dotted, dash] &&&&\bullet \arrow[uull] \arrow[ddll] \\
&&&&& \\
&\bullet \arrow[uul] &&\bullet \arrow[uull] \arrow[ll] \arrow[uulll, dotted, dash] &&
\end{tikzcd}
\tab\item 
\begin{tikzcd}[column sep=tiny, row sep=tiny]
&& \bullet \arrow[ddll]	&&&\bullet \arrow[lll] \\
&&&&&\\
\bullet \arrow[rrrr, dotted, dash] &&&& \bullet \arrow[uull] \arrow[ddll] &\bullet \ar[uu] \ar[l] \arrow[uulll, dotted, dash] \\
&&&&& \\
\bullet &&\bullet \arrow[ll] \arrow[uull] &&&
\end{tikzcd}
\tab\item 
\begin{tikzcd}[column sep=tiny, row sep=tiny]
							& \bullet	\arrow[dl]			&										& \bullet  						&							\\
\bullet \arrow[rr, dotted, dash]	&  							& \bullet	\arrow[rr] \arrow[ul] \arrow[dl]	&  								& \bullet \arrow[dl] \arrow[ul]	\\	
							& \bullet \arrow[ul] \arrow[rr]	& 										& \bullet \arrow[ul, dotted, dash] 	&
\end{tikzcd}
\tab\item  
\begin{tikzcd} [column sep=tiny, row sep=tiny]
			& 							& \bullet \arrow[dl]	& \\
\bullet \ar[r]	& \bullet \ar[rr, dotted, dash]	&					& \bullet \arrow[ul] \arrow[d]	\\
\bullet \ar[u]	& \bullet \arrow[u]			&					& \bullet \arrow[ll]
\end{tikzcd}
\tab\item 
\begin{tikzcd} [column sep=tiny, row sep=tiny]
				& 							& \bullet \arrow[dl]	& \\
\bullet \ar[r]		& \bullet \ar[rr, dotted, dash]	&					& \bullet \arrow[ul] \arrow[d]	\\
\bullet \ar[ur]	& \bullet \arrow[u]			&					& \bullet \arrow[ll]
\end{tikzcd}
\tab\item 
\begin{tikzcd} [column sep=tiny, row sep=tiny]
					& 							& \bullet \arrow[dl]	& \\
\bullet \ar[d] \ar[r]	& \bullet \ar[rr, dotted, dash]	&					& \bullet \arrow[ul] \arrow[d]	\\
\bullet 				& \bullet \arrow[u]			&					& \bullet \arrow[ll]
\end{tikzcd}
\tab\item  
\begin{tikzcd}[column sep=tiny, row sep=tiny]
\bullet \arrow[dr]		&								& \bullet \arrow[dl]		& 										& \bullet 	\\
					& \bullet \arrow[rr, dotted, dash] 	&						& \bullet \arrow[ul] \arrow[d] \arrow[ur]	&			\\
					& \bullet \arrow[u]				&						& \bullet \arrow[ll]						&
\end{tikzcd}
\tab\item 
\begin{tikzcd}[column sep=tiny, row sep=tiny]
\bullet \arrow[dr]		&								& \bullet \arrow[dl]		& 							\\
					& \bullet \arrow[rr, dotted, dash] 	&						& \bullet \arrow[ul] \arrow[d]	\\
					& \bullet \arrow[u]				&						& \bullet \arrow[dl]			\\
					&								& \bullet \arrow[ul]		&
\end{tikzcd}
\tab\item 
\begin{tikzcd}[column sep=tiny, row sep=tiny]
\bullet \arrow[rr] &&\bullet \arrow[ddl, dotted, dash] &\bullet \arrow[dd] \\
\bullet &&& \\
\bullet \ar[u] &\bullet \arrow[uul] \arrow[rr] \ar[l] && \bullet \arrow[uul]	
\end{tikzcd}
\tab\item  
\begin{tikzcd} [column sep=tiny, row sep=tiny]
\bullet \arrow[rr] &&\bullet \arrow[ddl, dotted, dash] &\bullet \arrow[dd] \\
\bullet &&&\\
\bullet	&\bullet \arrow[uul] \arrow[rr] \ar[l] \ar[ul] &&\bullet \arrow[uul]	
\end{tikzcd}
\tab\item 
\begin{tikzcd} [column sep=tiny, row sep=tiny]
\bullet \arrow[rr] &&\bullet \arrow[ddl, dotted, dash] &\bullet \arrow[dd] \\
\bullet \ar[d] &&&\\
\bullet	&\bullet \arrow[uul] \arrow[rr] \ar[l] &&\bullet \arrow[uul]
\end{tikzcd}
\tab\item 
\begin{tikzcd}[column sep=tiny, row sep=tiny]
\bullet &\bullet \ar[l] &&\\
\bullet &\bullet \arrow[ddl] \arrow[rr] \ar[u]	&&\bullet \arrow[ddl] \\
&&&\\
\bullet \arrow[uu] \arrow[rr] &&\bullet \arrow[uul, dotted, dash] &
\end{tikzcd}
\tab\item 
\begin{tikzcd} [column sep=tiny, row sep=tiny]
\bullet &\bullet &&\\
\bullet &\bullet \arrow[ddl] \arrow[rr] \ar[u] \ar[ul]	&&\bullet \arrow[ddl] \\
&&&\\
\bullet \arrow[uu] \arrow[rr] &&\bullet \arrow[uul, dotted, dash]	&
\end{tikzcd}
\tab\item 
\begin{tikzcd} [column sep=tiny, row sep=tiny]
\bullet \ar[r] &\bullet &&\\
\bullet &\bullet \arrow[ddl] \arrow[rr] \ar[u]	&&\bullet \arrow[ddl] \\
&&&\\
\bullet \arrow[uu] \arrow[rr] &&\bullet \arrow[uul, dotted, dash]	&
\end{tikzcd}
\tab\item 
\begin{tikzcd}[column sep=tiny, row sep=tiny]
					&								& 								&								&					\\
\bullet 				&								& \bullet \arrow[ddl] \arrow[rr]		& 								& \bullet \arrow[ddl] 	\\
&&&&\\
\bullet \arrow[uu]		& \bullet \arrow[uul] \arrow[rr] 	&								& \bullet \arrow[uul, dotted, dash]	& \bullet \arrow[l]					
\end{tikzcd}
\tab\item  
\begin{tikzcd}[column sep=tiny, row sep=tiny]
					& 								& \bullet 								&								&					\\
&&&&\\
\bullet 				&								& \bullet \arrow[ddl] \arrow[rr] \arrow[uu]	& 								& \bullet \arrow[ddl] 	\\
&&&&\\
\bullet \arrow[uu]		& \bullet \arrow[uul] \arrow[rr]		&										& \bullet \arrow[uul, dotted, dash]	&					\\
\end{tikzcd}
\tab\item 
\begin{tikzcd} [column sep=tiny, row sep=tiny]
				&										& \bullet \arrow[d] 				&					&\\
\bullet \arrow[rr] 	&										& \bullet \arrow[ddl, dotted, dash]  	&					& \bullet \arrow[ddl] \\
&&&&\\
				& \bullet \arrow[uul]  \arrow[d] \arrow[rr] 	&								& \bullet \arrow[uul] 	&\\
				& \bullet 								&								&					&
\end{tikzcd}
\tab\item 
\begin{tikzcd}[column sep=tiny, row sep=tiny]
\bullet \arrow[r]		& \bullet \arrow[dd, dotted, dash]			&  							& \bullet \arrow[ddl]	\\
&&&\\
\bullet \arrow[uu]		& \bullet \arrow[l] \arrow[dd] \arrow[r]		& \bullet \arrow[uul]			&				    	\\
&&&\\
					& \bullet 								& 							&
\end{tikzcd}
\tab\item  
\begin{tikzcd}[column sep=tiny, row sep=tiny]
					& \bullet \arrow[dd]				& 							&					\\
&&&\\
\bullet \arrow[r]		& \bullet \arrow[dd, dotted, dash]	&  							& \bullet \arrow[ddl]	\\
&&&\\
\bullet \arrow[uu]		& \bullet \arrow[l] \arrow[r]		& \bullet \arrow[uul]			&
\end{tikzcd}
\tab\item 
\begin{tikzcd}[column sep=tiny, row sep=tiny]
					& \bullet 								& 								&					\\
&&&\\
\bullet 				& \bullet \arrow[uu] \arrow[dd] \arrow[rr]	&  								& \bullet \arrow[ddl]	\\
&&&\\
\bullet \arrow[uu]		& \bullet \arrow[l] \arrow[r]				& \bullet \arrow[uul, dotted, dash]	&
\end{tikzcd}
\tab\item 
\begin{tikzcd}[column sep=tiny, row sep=tiny]
\bullet				& \bullet \arrow[ddr, dotted, dash] \arrow[dd] \arrow[r]	& \bullet \arrow[dd]	&					\\
&&&\\
\bullet \arrow[uu]		& \bullet \arrow[l] \arrow[r]							& \bullet 			& \bullet \arrow[l]
\end{tikzcd}
\tab\item   
\begin{tikzcd}[column sep=tiny, row sep=tiny]
\bullet \arrow[d] &&\bullet \arrow[dll]	&& \bullet \ar[d] \\
\bullet \arrow[rrrr, dotted, dash] &&&& \bullet \arrow[dl] \arrow[ull]	\\
&\bullet \arrow[ul]	&&\bullet \arrow[ll] & 
\end{tikzcd}
\tab\item 
\begin{tikzcd}[column sep=tiny, row sep=tiny]
\bullet \arrow[dr]	& 								& \bullet	\arrow[ll] \arrow[dr]	&  								& \bullet		\\
				& \bullet \arrow[ur, dotted, dash]	& 							& \bullet \arrow[dl] \arrow[ur]		&			\\
\bullet \arrow[ur]	&								& \bullet \arrow[ul]			&								& 
\end{tikzcd}
\tab\item 
\begin{tikzcd}[column sep=tiny, row sep=tiny]
\bullet \ar[d] &&\bullet \arrow[ddll] && \\
\bullet &&&& \\
\bullet \arrow[u] \arrow[rrrr, dotted, dash] &&&&\bullet \arrow[dl] \arrow[uull] \\
&\bullet \arrow[ul]	&&\bullet \arrow[ll] & 
\end{tikzcd}
\tab\item  
\begin{tikzcd}[column sep=tiny, row sep=tiny]
							&								& \bullet \arrow[dl]	&								\\
\bullet 						& \bullet \arrow[rr, dotted, dash]	&					& \bullet \arrow[dd] \arrow[ul]		\\
							&								&					&								\\
\bullet \arrow[uu] \arrow[r]	& \bullet \arrow[uu] 				&					& \bullet \arrow[ll]  
\end{tikzcd}
\tab\item 
\begin{tikzcd}[column sep=tiny, row sep=tiny]
\bullet \arrow[rr] &&\bullet \arrow[ddl, dotted, dash] &\bullet \\
\bullet &&&\\
\bullet	\ar[u] \ar[r] &\bullet \arrow[uul] \arrow[rr] &&\bullet \arrow[uul] \arrow[uu] 
\end{tikzcd}
\tab\item 
\begin{tikzcd}[column sep=tiny, row sep=tiny]
&& \bullet \arrow[ddll]	&&&\bullet \arrow[lll] \\
&&&&&\\
\bullet \arrow[rrrr, dotted, dash] &&&& \bullet \arrow[uull] \arrow[ddll] &\bullet \ar[uu] \ar[l] \arrow[uulll, dotted, dash] \\
&&&&& \\
\bullet \ar[uu] &&\bullet \arrow[uull] &&&
\end{tikzcd}
\tab\item 
\begin{tikzcd}[column sep=tiny, row sep=tiny]
&&\bullet \arrow[rr] &&\bullet \arrow[ddll, dotted, dash] \\
&&&&\\
\bullet \arrow[rr] &&\bullet \arrow[rr] \arrow[uu] &&\bullet \arrow[d] \arrow[uu] \\
&&\bullet \ar[rr] &&\bullet
\end{tikzcd}
\tab\item 
\begin{tikzcd}[column sep=tiny, row sep=tiny]
&&&&\bullet \ar[d] \\
&&\bullet \arrow[rr] &&\bullet \arrow[ddll, dotted, dash] \\
&&&&\\
\bullet \arrow[rr] &&\bullet \arrow[rr] \arrow[uu] &&\bullet \arrow[d] \arrow[uu] \\
&&&&\bullet
\end{tikzcd}
\tab\item 
\begin{tikzcd}[column sep=tiny, row sep=tiny]
\bullet \arrow[r]				& \bullet \arrow[ddl, dotted, dash]	&  				\\
&&\\
\bullet \arrow[r] \arrow[uu]	& \bullet \arrow[r]				& \bullet \arrow[ddl] \arrow[uul]	\\
							&								&								\\
\bullet \arrow[r]				& \bullet 						&  
\end{tikzcd}
\tab\item  
\begin{tikzcd}[column sep=tiny, row sep=tiny]
\bullet	&\bullet \arrow[l] \arrow[r] &\bullet \arrow[ddl] &\bullet \ar[dd] \\
&&&\\
&\bullet \arrow[rr, dotted, dash] &&\bullet \arrow[ddl] \arrow[uul] \\
&&&\\
&&\bullet \arrow[uul] &  
\end{tikzcd}
\tab\item 
\begin{tikzcd}[column sep=tiny, row sep=tiny]
\bullet \arrow[r] &\bullet \arrow[ddl] &\bullet \\
&&\bullet \ar[d] \ar[u] \\
\bullet \arrow[rr, dotted, dash] &&\bullet \arrow[ddl] \arrow[uul] \\
&&\\
&\bullet \arrow[uul] &  
\end{tikzcd}
\tab\item 
\begin{tikzcd}[column sep=tiny, row sep=tiny]
\bullet \arrow[ddr]	&\bullet \arrow[r] &\bullet \arrow[ddl]	&\bullet \ar[dd] \\
&&&\\
&\bullet \arrow[rr, dotted, dash] &&\bullet \arrow[ddl] \arrow[uul] \\
&&&\\
&&\bullet \arrow[uul] &  
\end{tikzcd}
\tab\item  
\begin{tikzcd}[column sep=tiny, row sep=tiny]
\bullet 	& \bullet \arrow[r]						& \bullet \arrow[ddl, dotted, dash]	&  								\\
&&&\\
		& \bullet \arrow[uul] \arrow[uu] \arrow[r]	& \bullet \arrow[r]				& \bullet \arrow[ddl] \arrow[uul]	\\
&&&\\
		&										& \bullet 						&  
\end{tikzcd}
\tab\item 
\begin{tikzcd}[column sep=tiny, row sep=tiny]
\bullet \arrow[r]	& \bullet \arrow[dd]				& \bullet \arrow[dd] \arrow[l]	& \bullet \arrow[dd]	\\
&&&\\
				& \bullet \arrow[uur, dotted, dash]	& \bullet \arrow[l] \arrow[r] 	& \bullet 
\end{tikzcd}
\tab\item 
\begin{tikzcd}[column sep=tiny, row sep=tiny]
\bullet	& \bullet \arrow[dd] \arrow[r]	& \bullet \arrow[dd] \arrow[r]		& \bullet		\\
&&&\\
		& \bullet \arrow[r] \arrow[uul]	& \bullet \arrow[uul, dotted, dash]	& \bullet  \arrow[uu]
\end{tikzcd}
\tab\item  
\begin{tikzcd}[column sep=tiny, row sep=tiny]
		& \bullet \arrow[r]					& \bullet \arrow[ddl, dotted, dash] \arrow[ddr, dotted, dash]	& \bullet	\arrow[l]				\\
&&&\\
\bullet	& \bullet \arrow[r] \arrow[l] \arrow[uu]	& \bullet \arrow[uu] 										& \bullet  \arrow[uu] \arrow[l]
\end{tikzcd}
\end{inparaenum}
\end{theorem}
\begin{proof} \label{provae7}
We put the necessary information for each representation-finite trivial extension of Cartan class $\mathbb{E}_7$ \cite{fern} in the program and we show the non-hereditary solutions. As an example we will describe here this procedure to a particular  trivial extension the other cases follows the same pattern.

\begin{tikzcd} [column sep=small, row sep=small]
& \bullet \ar{dl}[swap]{\alpha_1} &&\bullet \ar{ll}[swap]{\alpha_6} &\\
\bullet \ar{rrrr}[swap]{\alpha_2} &&&&\bullet \ar{ul}[swap]{\alpha_5} \ar{dd}{\alpha_3} \\
&&&&\\
\bullet \arrow{uu}{\alpha_4} &&\bullet \arrow{ll}{\alpha_7} &&\bullet \arrow{ll}{\alpha_8}
\end{tikzcd}
\hfill
\begin{minipage}{.45\linewidth}
The relations of type $2$ are: $r1=\alpha_1 \alpha_2 \alpha_3$ and $r2=\alpha_4 \alpha_2 \alpha_5$. The elementary cycles are: $C_1 = \alpha_1 \alpha_2 \alpha_5 \alpha_6 $ and $C_2 = \alpha_4 \alpha_2 \alpha_3 \alpha_8 \alpha_7 $. Therefore, the program shows us the cutting set that defines the solution $1$ of the theorem.
\end{minipage}

\vspace{.5cm}
\end{proof}

%% file: e8.tex
In the previous section, we explained the results and methods applied in the classification of all PHI algebras of type $\mathbb{E}_7$. We will repeat the same procedure for the case $\mathbb{E}_8$.

From the thesis of Fern\'andez \cite{fern}, we have the list of all such trivial extensions, counting $251$. From this, we will present all non-hereditary incidence algebras defined by the cutting sets of each trivial extension, as we did for the case $\mathbb{E}_7$. We did the manual work of finding the incidence algebras defined by the cutting sets for each of the trivial extensions, except that in order to write the proof we prefer to use the program.

\begin{theorem}
Let $K\Delta \cong KQ/I$ be a PHI algebra of type $\mathbb{E}_8$ such that $\overline{Q}$ is not of tree type. Then $K\Delta$ or $K\Delta^{op}$ is isomorphic to one of the following algebras:

\NumTabs{2}
\begin{inparaenum}

\noindent
\item
\begin{tikzcd}[column sep=tiny, row sep=tiny]
\bullet  \arrow[d]       				&  					& 					& \bullet \ar[lll]				\\
\bullet \arrow[rrr, dotted, dash]		&              		& 					& \bullet \arrow[u] \arrow[d] 	\\
\bullet \arrow[u]     				& \bullet \arrow[l]	& \bullet \arrow[l]	& \bullet \arrow[l]
\end{tikzcd}
\tab\item 
\begin{tikzcd}[column sep=tiny, row sep=tiny]
\bullet \ar[r] &\bullet \ar[ddl] \ar[rr] \ar[dd, dotted, dash] &&\bullet \ar[d] \\
&&&\bullet \ar[d] \\
\bullet \ar[r] &\bullet &\bullet \ar[l] &\bullet \ar[l] \\
& &
\end{tikzcd}
\tab\item 
\begin{tikzcd}[column sep=tiny, row sep=tiny]
&\bullet \ar[dddl] \ar[dr] \ar[ddddd, dotted, dash]&\\
&&\bullet \ar[d] \\
&&\bullet \ar[d] \\
\bullet \ar[ddr] &&\bullet \ar[d] \\
&&\bullet \ar[dl] \\
&\bullet \ar[d] &\\
&\bullet &
\end{tikzcd}
\tab\item 
\begin{tikzcd}[column sep=tiny, row sep=tiny]
&\bullet \ar[d] &\\
&\bullet \ar[d] &\\
&\bullet \ar[d] &\\
&\bullet \ar[dl] \ar[dr] \ar[ddd, dotted, dash]&\\
\bullet \ar[ddr] &&\bullet \ar[d] \\
&&\bullet \ar[dl] \\
&\bullet &
\end{tikzcd}
\tab\item 
\begin{tikzcd}[column sep=tiny, row sep=tiny]
&\bullet \ar[d] &\\
&\bullet \ar[d] &\\
&\bullet \ar[dl] \ar[dr] \ar[ddd, dotted, dash]&\\
\bullet \ar[ddr] &&\bullet \ar[d] \\
&&\bullet \ar[dl] \\
&\bullet \ar[d] &\\
&\bullet &
\end{tikzcd}
\tab\item 
\begin{tikzcd}[column sep=tiny, row sep=tiny]
&\bullet \ar[d] &\\
&\bullet \ar[dl] \ar[dr] \ar[ddd, dotted, dash]&\\
\bullet \ar[ddr] &&\bullet \ar[d] \\
&&\bullet \ar[dl] \\
&\bullet \ar[d] &\\
&\bullet \ar[d] &\\
&\bullet &
\end{tikzcd}
\tab\item 
\begin{tikzcd}[column sep=tiny, row sep=tiny]
&\bullet \ar[dl] \ar[dr] \ar[ddd, dotted, dash]&\\
\bullet \ar[ddr] &&\bullet \ar[d] \\
&&\bullet \ar[dl] \\
&\bullet \ar[d] &\\
&\bullet \ar[d] &\\
&\bullet \ar[d] &\\
&\bullet &
\end{tikzcd}
\tab\item 
\begin{tikzcd}[column sep=tiny, row sep=tiny]
\bullet \arrow[dr] && \bullet \ar[ll] \ar[dr] &&\bullet \\
&\bullet \arrow[ur, dotted, dash] && \bullet \ar[d] \ar[ur]	& \\
&\bullet \arrow[u] &\bullet \ar[l] &\bullet \arrow[l] & 
\end{tikzcd}
\tab\item 
\begin{tikzcd}[column sep=tiny, row sep=tiny]
\bullet \arrow[dr] &&\bullet \ar[dr, dotted, dash] &&\bullet \arrow[ll] \\
&\bullet \arrow[ur]	&& \bullet \ar[d] \ar[ur] & \\
&\bullet \arrow[u] &\bullet \ar[l] &\bullet \arrow[l] & 
\end{tikzcd}
\tab\item 
\begin{tikzcd}[column sep=tiny, row sep=tiny]
       			& \bullet \arrow[dl] 	& 								& \bullet & 							\\
\bullet \arrow[d]	& 					& \bullet \ar[rr] \ar[ul]			&  		& \bullet \arrow[ul] \arrow[d]	\\
\bullet \arrow[rr]	& 					& \bullet \arrow[u, dotted, dash]	& 		& \bullet \arrow[ll]
\end{tikzcd}
\tab\item 
\begin{tikzcd}[column sep=tiny, row sep=tiny]
       			& \bullet \arrow[dl] 	& 							& \bullet \arrow[dl]	& 							\\
\bullet \arrow[d]	& 					& \bullet \ar[rr, dotted, dash] 	&  					& \bullet \arrow[ul] \arrow[d]	\\
\bullet \arrow[rr]	& 					& \bullet \arrow[u]			& 					& \bullet \arrow[ll]
\end{tikzcd}
\tab\item 
\begin{tikzcd}[column sep=tiny, row sep=tiny]
				& \bullet \arrow[dl]				& 						& \bullet 				& 					\\
\bullet \arrow[dr]	& 								& \bullet \ar[ul] \ar[dr]	& 						& \bullet \arrow[ul]	\\
     				& \bullet \arrow[ur, dotted, dash]	& 						& \bullet \ar[dl] \ar[ur]	& 					\\
				&								& \bullet \ar[lu]			&						&
\end{tikzcd}
\tab\item 
\begin{tikzcd}[column sep=tiny, row sep=tiny]
				& \bullet \arrow[dl]	& 							& \bullet \ar[dl] 			& 					\\
\bullet \arrow[dr]	& 					& \bullet \ar[dr, dotted, dash]	& 						& \bullet \arrow[ul]	\\
     				& \bullet \arrow[ur]	& 							& \bullet \ar[dl] \ar[ur]	& 					\\
				&					& \bullet \ar[lu]				&						&
\end{tikzcd}
\tab\item 
\begin{tikzcd}[column sep=tiny, row sep=tiny]
\bullet \arrow[r]	& \bullet 		 	& 				& \bullet \ar[r] \ar[dl]	& \bullet \arrow[dd]				\\
				& 					& \bullet \ar[dl] 	&  					& 								\\
\bullet \arrow[uu]	& \bullet \ar[l] \ar[rrr]	& 				& 					& \bullet \arrow[uul, dotted, dash]	
\end{tikzcd}
\tab\item 
\begin{tikzcd}[column sep=tiny, row sep=tiny]
&&\bullet \ar[d] &\\
&&\bullet \ar[ddl] \ar[dr] \ar[dddd, dotted, dash]&\\
&&&\bullet \ar[d] \\
\bullet \ar[r] &\bullet \ar[ddr] &&\bullet \ar[d] \\
&&&\bullet \ar[dl] \\
&&\bullet &
\end{tikzcd}
\tab\item 
\begin{tikzcd}[column sep=tiny, row sep=tiny]
\bullet			& 					& 				& \bullet \ar[r] \ar[dl]		 	& \bullet \ar[dd]	\\
				& 					& \bullet \ar[dl]	& 							& 				\\
\bullet \ar[uu]	& \bullet \ar[l] \ar[rr]	& 				& \bullet \ar[uu, dotted, dash]	& \bullet \ar[l]
\end{tikzcd}
\tab\item 
\begin{tikzcd}[column sep=tiny, row sep=tiny]
&&\bullet \ar[d] &\\
&&\bullet \ar[dl] \ar[dr] \ar[ddd, dotted, dash]&\\
\bullet \ar[r] &\bullet \ar[ddr] &&\bullet \ar[d] \\
\bullet \ar[u] &&&\bullet \ar[dl] \\
&&\bullet &
\end{tikzcd}
\tab\item 
\begin{tikzcd}[column sep=tiny, row sep=tiny]
\bullet				&  				& \bullet \ar[dl] \ar[rr]	& 							& \bullet \ar[d]	\\
					& \bullet \ar[dl]	& 						&  							& \bullet \ar[d]	\\
\bullet \ar[uu] \ar[rrr]	& 				& 						& \bullet \ar[uul, dotted, dash]	& \bullet \ar[l]	
\end{tikzcd}
\tab\item 
\begin{tikzcd}[column sep=tiny, row sep=tiny]
&&\bullet \ar[d] &\\
&&\bullet \ar[dl] \ar[dr] \ar[ddd, dotted, dash]&\\
\bullet \ar[r] &\bullet \ar[ddr] &&\bullet \ar[ddl] \\
\bullet \ar[u] &&& \\
\bullet \ar[u] &&\bullet &
\end{tikzcd}
\tab\item 
\begin{tikzcd}[column sep=tiny, row sep=tiny]
&\bullet &&\\
\bullet \ar[d] \ar[ur] &\bullet \ar[l] &\bullet \ar[l] \ar[r] & \bullet \ar[d] \\
\bullet \ar[rr] &&\bullet \ar[u, dotted, dash] &\bullet \ar[l] 
\end{tikzcd}
\tab\item 
\begin{tikzcd}[column sep=tiny, row sep=tiny]
&\bullet \ar[d] &\\
&\bullet \ar[dl] \ar[dr] \ar[ddd, dotted, dash]&\\
\bullet \ar[ddr] &&\bullet \ar[d] \\
&&\bullet \ar[dl] &\bullet \ar[l]\\
&\bullet &&\bullet \ar[u]
\end{tikzcd}
\tab\item 
\begin{tikzcd}[column sep=tiny, row sep=tiny]
&\bullet \ar[d] &&\\
&\bullet \ar[ddl] \ar[dr] \ar[dddd, dotted, dash] &&\\
&&\bullet \ar[d] &\\
\bullet \ar[ddr] &&\bullet \ar[d] &\\
&&\bullet \ar[dl] &\bullet \ar[l] \\
&\bullet &&
\end{tikzcd}
\tab\item 
\begin{tikzcd}[column sep=tiny, row sep=tiny]
\bullet				&  				& \bullet \ar[dl] \ar[rr]	& 				& \bullet \ar[dd]				\\
					& \bullet \ar[dl]	& 						&  				& 							\\
\bullet \ar[dr] \ar[uu]	& 				& 						& 				& \bullet \ar[uull, dotted, dash]	\\
					& \bullet \ar[rr]	&						& \bullet \ar[ur]	&
\end{tikzcd}
\tab\item 
\begin{tikzcd}[column sep=tiny, row sep=tiny]
					&				& \bullet	& 				&							&				\\
\bullet \ar[dd] \ar[urr]	& \bullet \ar[l]	& 		& \bullet \ar[ll] 	& \bullet \ar[dr] \ar[l]			&				\\
					&				&		&				& 							& \bullet \ar[dl] 	\\ 
\bullet \ar[rrrr]		& 				& 		& 				& \bullet \ar[uu, dotted, dash]	& 
\end{tikzcd}
\tab\item 
\begin{tikzcd}[column sep=tiny, row sep=tiny]
&\bullet \ar[d] &&\\
&\bullet \ar[d] &&\\
&\bullet \ar[dl] \ar[dr] \ar[ddd, dotted, dash]&&\\
\bullet \ar[ddr] &&\bullet \ar[d] &\\
&&\bullet \ar[dl] &\bullet \ar[l]\\
&\bullet &&
\end{tikzcd}
\tab\item 
\begin{tikzcd}[column sep=tiny, row sep=tiny]
&\bullet \ar[d] &&\\
&\bullet \ar[d] &&\\
&\bullet \ar[d] &&\\
&\bullet \ar[dl] \ar[dr] \ar[dd, dotted, dash]&&\\
\bullet \ar[dr] &&\bullet \ar[dl] &\bullet \ar[l]\\
&\bullet &&
\end{tikzcd}
\tab\item 
\begin{tikzcd}[column sep=tiny, row sep=tiny]
\bullet \ar[d]		& 							& \bullet \ar[ll]		& 					\\
\bullet \ar[dr]	& 							& \bullet \ar[dr] \ar[u]	& \bullet 			\\
\bullet \ar[u]		& \bullet \ar[ur, dotted, dash]	& 					& \bullet \ar[ll] \ar[u]
\end{tikzcd}
\tab\item 
\begin{tikzcd}[column sep=tiny, row sep=tiny]
\bullet \ar[r] \ar[dd]		& \bullet \ar[dd, dotted, dash]	& \bullet \ar[r]	& \bullet \ar[dd]	\\
						&							&				& 				\\
\bullet					& \bullet \ar[uul] \ar[r]		& \bullet \ar[uul]	& \bullet \ar[l] 						
\end{tikzcd}
\tab\item 
\begin{tikzcd}[column sep=tiny, row sep=tiny]
					&							& \bullet \ar[d]	\\
\bullet \ar[r]			& \bullet \ar[d, dotted, dash]	& \bullet \ar[l]	\\
\bullet \ar[u] \ar[d]	& \bullet \ar[l] \ar[r]			& \bullet \ar[u]	\\
\bullet 				&							&
\end{tikzcd}
\tab\item 
\begin{tikzcd}[column sep=tiny, row sep=tiny]
		& 						& \bullet \ar[dr]			& 							& \bullet \ar[dr]	&				\\
		& \bullet \ar[dl] \ar[ur]	&						& \bullet \ar[dl, dotted, dash]	& 				& \bullet \ar[dl]	\\ 
\bullet 	& 						& \bullet \ar[ul] \ar[rr]	& 							& \bullet \ar[ul]	& 
\end{tikzcd}
\tab\item 
\begin{tikzcd}[column sep=tiny, row sep=tiny]
					& \bullet \ar[dl]	&							&					\\
\bullet \ar[dr]		&				& \bullet \ar[ul] \ar[ddr]		& \bullet 			\\
					& \bullet \ar[dr]	&							& 					\\
\bullet \ar[uu]		& 				& \bullet \ar[uu, dotted, dash] 	& \bullet \ar[uu] \ar[l]	 						
\end{tikzcd}
\tab\item 
\begin{tikzcd}[column sep=tiny, row sep=tiny]
					& 				&							& \bullet 	&					\\
\bullet \ar[dr]		&				& \bullet \ar[ll] \ar[ddrr]		& 			& \bullet \ar[ul]		\\
					& \bullet \ar[dr]	&							&			&				\\
\bullet \ar[uu]		& 				& \bullet \ar[uu, dotted, dash] 	& 			& \bullet \ar[uu] \ar[ll]	 						
\end{tikzcd}
\tab\item 
\begin{tikzcd}[column sep=tiny, row sep=tiny]
\bullet \ar[dr]	& 				& 				& \bullet \ar[lll] \ar[dddr]			& \bullet					\\
				& \bullet \ar[dr]	&				& 								& 						\\ 
 				& 				& \bullet \ar[dr] 	& 								&						\\
\bullet \ar[uuu]	&				&				& \bullet \ar[uuu, dotted, dash]		& \bullet \ar[l] \ar[uuu]
\end{tikzcd}
\tab\item 
\begin{tikzcd} [column sep=tiny, row sep=tiny]
\bullet \ar[dr] 	&				& \bullet \ar[dr] \ar[ll] 		&				& \bullet   \\
				& \bullet \ar[dr] 	&							& \bullet \ar[dr] 	& \\ 
\bullet \ar[uu]	&				& \bullet \ar[uu, dotted, dash] 	&				& \bullet \ar[uu]  \ar[ll] 
\end{tikzcd}
\tab\item 
\begin{tikzcd}[column sep=tiny, row sep=tiny]
&\bullet \ar[dl] \ar[ddr] \ar[dddd, dotted, dash] &&\\
\bullet \ar[d] &&&\\
\bullet \ar[d] &&\bullet \ar[ddl] \ar[ddr] \ar[ddd, dotted, dash] &\\
\bullet \ar[dr] &&&\\
&\bullet \ar[dr] &&\bullet \ar[dl] \\
&&\bullet &
\end{tikzcd}
\tab\item 
\begin{tikzcd}[column sep=tiny, row sep=tiny]
&\bullet \ar[dl] \ar[dr] \ar[dd, dotted, dash]&&\\
\bullet \ar[dr] &&\bullet \ar[dl] &\bullet \ar[l]\\
&\bullet \ar[d] &&\bullet \ar[u] \\
&\bullet &&\bullet \ar[u]
\end{tikzcd}
\tab\item 
\begin{tikzcd}[column sep=tiny, row sep=tiny]
&&\bullet \ar[d] &\\
&&\bullet \ar[ddl] \ar[dr] \ar[dddd, dotted, dash]&\\
&&&\bullet \ar[d] \\
\bullet &\bullet \ar[ddr] \ar[l] &&\bullet \ar[d] \\
&&&\bullet \ar[dl] \\
&&\bullet &
\end{tikzcd}
\tab\item 
\begin{tikzcd}[column sep=tiny, row sep=tiny]
&&&\bullet \ar[dll] \ar[dr] \ar[dddr, dotted, dash, bend right] &\\
&\bullet \ar[dl] \ar[ddrrr] \ar[dddrr, dotted, dash, bend right] &&&\bullet \ar[d] \\
\bullet \ar[d] &&&&\bullet \ar[d] \\
\bullet \ar[drrr] &&&&\bullet \ar[dl] \\
&&&\bullet &
\end{tikzcd}
\tab\item 
\begin{tikzcd}[column sep=tiny, row sep=tiny]
&&\bullet \ar[dl] \ar[dr] \ar[ddd, dotted, dash]&\\
\bullet \ar[r] &\bullet \ar[ddr] &&\bullet \ar[d] \\
\bullet \ar[u] &&&\bullet \ar[dl] \\
&&\bullet \ar[d] &\\
&&\bullet &
\end{tikzcd}
\tab\item 
\begin{tikzcd}[column sep=tiny, row sep=tiny]
&\bullet \ar[dl] \ar[dr] \ar[ddddl, dotted, dash, bend left] &\\
\bullet \ar[ddd] &&\bullet \ar[d] \\
&&\bullet \ar[d] \\
&&\bullet \ar[d] \ar[dll] \ar[ddl, dotted, dash, bend right=10] \\
\bullet \ar[dr] &&\bullet \ar[dl] \\
&\bullet &
\end{tikzcd}
\tab\item 
\begin{tikzcd}[column sep=tiny, row sep=tiny]
&\bullet \ar[d] &&\\
&\bullet \ar[d] &&\\
&\bullet \ar[dl] \ar[dr] \ar[dd, dotted, dash]&&\\
\bullet \ar[dr] &&\bullet \ar[dl] &\bullet \ar[l]\\
&\bullet \ar[d] &&\\
&\bullet &&
\end{tikzcd}
\tab\item 
\begin{tikzcd}[column sep=tiny, row sep=tiny]
&\bullet \ar[d] &&\\
&\bullet \ar[ddl] \ar[dr] \ar[dddd, dotted, dash] &&\\
&&\bullet \ar[d] &\\
\bullet \ar[ddr] &&\bullet \ar[d] &\\
&&\bullet \ar[dl] \ar[r] &\bullet \\
&\bullet &&
\end{tikzcd}
\tab\item 
\begin{tikzcd}[column sep=tiny, row sep=tiny]
&\bullet \ar[dl] \ar[dr] \ar[dddl, dotted, dash, bend left] &\\
\bullet \ar[dd] &&\bullet \ar[d] \\
&&\bullet \ar[d] \ar[dll] \ar[ddl, dotted, dash, bend right=10] \\
\bullet \ar[dr] &&\bullet \ar[dl] \\
&\bullet \ar[d] &\\
&\bullet &
\end{tikzcd}
\tab\item 
\begin{tikzcd}[column sep=tiny, row sep=tiny]
&\bullet \ar[d] &\\
&\bullet \ar[dl] \ar[dr] \ar[dddl, dotted, dash, bend left] &\\
\bullet \ar[dd] &&\bullet \ar[d] \\
&&\bullet \ar[d] \ar[dll] \ar[ddl, dotted, dash, bend right=10] \\
\bullet \ar[dr] &&\bullet \ar[dl] \\
&\bullet &
\end{tikzcd}
\tab\item 
\begin{tikzcd}[column sep=tiny, row sep=tiny]
&\bullet \ar[d] &&\\
&\bullet \ar[dl] \ar[dr] \ar[dd, dotted, dash]&&\\
\bullet \ar[dr] &&\bullet \ar[dl] &\bullet \ar[l]\\
&\bullet \ar[d] &&\\
&\bullet \ar[d] &&\\
&\bullet &&
\end{tikzcd}
\tab\item 
\begin{tikzcd}[column sep=tiny, row sep=tiny]
&\bullet \ar[d] &&\\
&\bullet \ar[d] &&\\
&\bullet \ar[dl] \ar[dr] \ar[ddd, dotted, dash]&&\\
\bullet \ar[ddr] &&\bullet \ar[d] &\\
&&\bullet \ar[dl] \ar[r] &\bullet \\
&\bullet &&
\end{tikzcd}
\tab\item 
\begin{tikzcd}[column sep=tiny, row sep=tiny]
&\bullet \ar[dl] \ar[dr] \ar[ddr, dotted, dash, bend right=10] &\\
\bullet \ar[d] \ar[drr] \ar[ddr, dotted, dash, bend left=10] &&\bullet \ar[d] \\
\bullet \ar[dr] &&\bullet \ar[dl] \\
&\bullet \ar[d] &\\
&\bullet \ar[d] &\\
&\bullet &
\end{tikzcd}
\tab\item 
\begin{tikzcd}[column sep=tiny, row sep=tiny]
&\bullet \ar[d] &\\
&\bullet \ar[dl] \ar[dr] \ar[ddr, dotted, dash, bend right=10] &\\
\bullet \ar[d] \ar[drr] \ar[ddr, dotted, dash, bend left=10] &&\bullet \ar[d] \\
\bullet \ar[dr] &&\bullet \ar[dl] \\
&\bullet \ar[d] &\\
&\bullet &
\end{tikzcd}
\tab\item 
\begin{tikzcd}[column sep=tiny, row sep=tiny]
&\bullet \ar[d] &\\
&\bullet \ar[d] &\\
&\bullet \ar[dl] \ar[dr] \ar[ddr, dotted, dash, bend right=10] &\\
\bullet \ar[d] \ar[drr] \ar[ddr, dotted, dash, bend left=10] &&\bullet \ar[d] \\
\bullet \ar[dr] &&\bullet \ar[dl] \\
&\bullet &
\end{tikzcd}
\tab\item 
\begin{tikzcd}[column sep=tiny, row sep=tiny]
&\bullet \ar[d] &&\\
&\bullet \ar[d] &&\\
&\bullet \ar[d] &&\\
&\bullet \ar[dl] \ar[dr] \ar[dd, dotted, dash]&&\\
\bullet \ar[dr] &&\bullet \ar[dl] \ar[r] &\bullet \\
&\bullet &&
\end{tikzcd}
\tab\item 
\begin{tikzcd} [column sep=tiny, row sep=tiny]  
				&				&												& \bullet \ar[dl] 	&  \\
				& \bullet \ar[r]	& \bullet \ar[dd, dotted, dash] \ar[rr, dotted, dash] 	&				& \bullet \ar[dd] \ar[ul]\\
\bullet \ar[ur] 	&&&& \\
				& \bullet \ar[ul] 	& \bullet \ar[l] \ar[rr] 								&				& \bullet \ar[uull] \\
\end{tikzcd}
\tab\item 
\begin{tikzcd} [column sep=tiny, row sep=tiny]  
				&				&						& \bullet 	 	&  \\
				& \bullet 		& \bullet \ar[dd] \ar[rr] 	&				& \bullet \ar[dd] \ar[ul]\\
\bullet \ar[ur] 	&&&& \\
				& \bullet \ar[ul] 	& \bullet \ar[l] \ar[rr] 		&				& \bullet \ar[uull, dotted, dash] \\
\end{tikzcd}
\tab\item 
\begin{tikzcd} [column sep=tiny, row sep=tiny]  
				&						&												& \bullet \ar[dl] 	&  \\
\bullet \ar[rr]		&						& \bullet \ar[dl, dotted, dash] \ar[rr, dotted, dash] 	&				& \bullet \ar[dl] \ar[ul]\\
				& \bullet \ar[ul] \ar[d] 	&												& \bullet \ar[ul] 	& \\
				& \bullet \ar[rr] 			&												& \bullet \ar[u] 	&\\
\end{tikzcd}
\tab\item 
\begin{tikzcd} [column sep=tiny, row sep=tiny]  
		&						&						& \bullet 					&  \\
\bullet 	&						& \bullet \ar[dl] \ar[rr] 	&							& \bullet \ar[dl] \ar[ul]\\
		& \bullet \ar[ul] \ar[d] 	&						& \bullet \ar[ul, dotted, dash] 	& \\
		& \bullet \ar[rr] 			&						& \bullet \ar[u] 				&\\
\end{tikzcd}
\tab\item 
\begin{tikzcd} [column sep=tiny, row sep=tiny]  
				& \bullet \ar[d] 									& \bullet \ar[dr] \ar[l] 	&\\
\bullet \ar[r] 		& \bullet \ar[dd, dotted, dash] \ar[ur, dotted, dash]	&					& \bullet \ar[dl] \\
				&												& \bullet  \ar[ul] 		&\\
\bullet \ar[uu] 	& \bullet \ar[l] \ar[ur]								&					&\\
\end{tikzcd}
\tab\item 
\begin{tikzcd} [column sep=tiny, row sep=tiny]  
				& \bullet 	 			& \bullet \ar[dr] \ar[l] 				&\\
\bullet  			& \bullet \ar[dd] \ar[ur]	&								& \bullet \ar[dl] \\
				&						& \bullet  \ar[ul, dotted, dash] 		&\\
\bullet \ar[uu] 	& \bullet \ar[l] \ar[ur]		&								&\\
\end{tikzcd}
\tab\item 
\begin{tikzcd} [column sep=tiny, row sep=tiny]  
					&				&									& \bullet 	&\\
\bullet 		 		&				& \bullet \ar[dl] \ar[rr]				&			& \bullet \ar[dd] \ar[ul]\\
					& \bullet  \ar[dl] 	&									&			&\\
\bullet \ar[uu] \ar[rr] 	&				& \bullet \ar[uu, dotted, dash] 		 	&			& \bullet \ar[ll] \\
\end{tikzcd}
\tab\item 
\begin{tikzcd} [column sep=tiny, row sep=tiny]  
&&&&\bullet \ar[dl] &\\
&\bullet \ar[rr] &&\bullet \ar[ddll, dotted, dash] \ar[rr, dotted, dash] &&\bullet \ar[dd] \ar[ul] \\
&&&&&\\
\bullet \ar[r] &\bullet \ar[uu] \ar[rr] &&\bullet \ar[uu] &&\bullet \ar[ll]
\end{tikzcd}
\tab\item 
\begin{tikzcd} [column sep=tiny, row sep=tiny]  
&&&\bullet \ar[d] &&\bullet \ar[ll]\\
&\bullet \ar[rr] &&\bullet \ar[ddll, dotted, dash] \ar[rr, dotted, dash] &&\bullet \ar[ddll] \ar[u]\\
&&&&&\\
\bullet \ar[r] &\bullet \ar[uu] \ar[rr] &&\bullet \ar[uu] &&
\end{tikzcd}
\tab\item 
\begin{tikzcd} [column sep=tiny, row sep=tiny]  
					&				& \bullet 					&& \bullet \ar[ll]\\
\bullet 				&				& \bullet \ar[dl] \ar[rr] 		&& \bullet \ar[ddll] \ar[u]\\
					& \bullet  \ar[dl] 	&							&&\\
\bullet \ar[uu] \ar[rr] 	&				& \bullet \ar[uu, dotted, dash] 	&&\\
\end{tikzcd}

\noindent
\item 
\begin{tikzcd} [column sep=tiny, row sep=tiny]  
					& \bullet 		&							&&\\
\bullet \ar[ur] \ar[dd] 	& \bullet  \ar[l]	& \bullet \ar[l] \ar[ddrr] 		&& \bullet  \\
&&&&\\
\bullet \ar[rr] 		&				& \bullet \ar[uu, dotted, dash] 	&& \bullet \ar[ll] \ar[uu] \\
\end{tikzcd}
\tab\item 
\begin{tikzcd} [column sep=tiny, row sep=tiny]  
&&\bullet \ar[dr] &&&\\
\bullet \ar[r] &\bullet \ar[ur] \ar[dd] &&\bullet \ar[ll, dotted, dash] \ar[ddrr, dotted, dash] &&\bullet \ar[ll] \\
&&&&&\\
&\bullet \ar[rr] &&\bullet \ar[uu] &&\bullet \ar[ll] \ar[uu] \\
\end{tikzcd}
\tab\item 
\begin{tikzcd}[column sep=tiny, row sep=tiny]
&&\bullet \ar[d] &\\
&&\bullet \ar[dl] \ar[dr] \ar[ddd, dotted, dash]&\\
\bullet &\bullet \ar[l] \ar[ddr] &&\bullet \ar[d] \\
&&&\bullet \ar[dl] &\bullet \ar[l]\\
&&\bullet &&
\end{tikzcd}
\tab\item 
\begin{tikzcd} [column sep=tiny, row sep=tiny]  
&\bullet \ar[rrr] &&&\bullet \ar[ddlll, dotted, dash] \ar[ddr, dotted, dash] &\bullet \ar[l] \\
\bullet \ar[d] &&&&&\\
\bullet \ar[r] &\bullet \ar[rrr] \ar[uu] &&&\bullet \ar[uu] &\bullet \ar[l] \ar[uu] \\
\end{tikzcd}
\tab\item 
\begin{tikzcd} [column sep=tiny, row sep=tiny]  
					&& \bullet \ar[d]  								&& \bullet \ar[ll] \\
\bullet \ar[rr] \ar[dd]	&& \bullet \ar[dd, dotted, dash] \ar[rr, dotted, dash] 	&& \bullet  \ar[dd] \ar[u] \\
&&&&\\
\bullet 		 		&& \bullet \ar[uull]  \ar[rr] 						&& \bullet \ar[uull] \\
\end{tikzcd}
\tab\item 
\begin{tikzcd} [column sep=tiny, row sep=tiny]  
\bullet \ar[rr]  			&& \bullet \ar[ddll] 				&&\\
&&&&\\
\bullet \ar[rr, dotted, dash] && \bullet \ar[dd] \ar[rr] \ar[uu]	&& \bullet  \ar[dd]  \\
&&&&\\
\bullet \ar[uu]  			&& \bullet \ar[ll]  \ar[rr] 			&& \bullet \ar[uull, dotted, dash] \\
\end{tikzcd}
\tab\item 
\begin{tikzcd} [column sep=tiny, row sep=tiny]  
					&			&												& \bullet \ar[dl] 	&\\
\bullet \ar[rr] 		&			& \bullet \ar[dd, dotted, dash] \ar[rr, dotted, dash] 	&				& \bullet \ar[dd] \ar[ul] \\
&&&&\\
\bullet \ar[dr] \ar[uu] 	&			& \bullet \ar[ll]  \ar[rr] 							&				& \bullet \ar[uull] \\
					& \bullet 	&&&\\
\end{tikzcd}
\tab\item 
\begin{tikzcd} [column sep=tiny, row sep=tiny]  
\bullet \ar[r] 		& \bullet \ar[dd] 					&& \bullet \ar[ll] \ar[dd] \ar[rr] 	&& \bullet \ar[dd] \\
&&&&\\
\bullet \ar[uu] 	& \bullet \ar[uurr, dotted, dash] 	&& \bullet \ar[ll] \ar[rr] 			&& \bullet \ar[uull, dotted, dash]\\
\end{tikzcd}
\tab\item 
\begin{tikzcd} [column sep=tiny, row sep=tiny]  
&&&\bullet \ar[dl] &&\\
\bullet \ar[rr] \ar[dd] &&\bullet \ar[dd, dotted, dash] \ar[rr, dotted, dash] &&\bullet \ar[ul] \ar[dd] &\bullet \ar[l] \\
&&&&&\\
\bullet &&\bullet \ar[uull] \ar[rr] &&\bullet \ar[uull] &\\
\end{tikzcd}
\tab\item 
\begin{tikzcd}[column sep=tiny, row sep=tiny]
&&\bullet \ar[dl] \ar[dr] \ar[ddl, dotted, dash, bend left] &\\
\bullet \ar[d] &\bullet \ar[l] \ar[d] &&\bullet \ar[d] \ar[dll] \ar[ddl, dotted, dash, bend right=10] \\
\bullet &\bullet \ar[dr] &&\bullet \ar[dl] \\
&&\bullet &
\end{tikzcd}
\tab\item 
\begin{tikzcd} [column sep=tiny, row sep=tiny]  
&\bullet \ar[dl] \ar[dr] \ar[dd, dotted, dash] &&\bullet \ar[dddl, dotted, dash] \ar[dr] \ar[dl] &\\
\bullet \ar[dr] &&\bullet \ar[dl] &&\bullet \ar[d] \\
&\bullet \ar[dr] &&&\bullet \ar[dll] \\
&&\bullet &&\\
\end{tikzcd}
\tab\item 
\begin{tikzcd}[column sep=tiny, row sep=tiny]
&&\bullet \ar[dl] \ar[ddr] \ar[dddl, dotted, dash, bend left] &\\
\bullet &\bullet \ar[l] \ar[d] &&\\
&\bullet \ar[d] &&\bullet \ar[d] \ar[dll] \ar[ddl, dotted, dash, bend right=10] \\
&\bullet \ar[dr] &&\bullet \ar[dl] \\
&&\bullet &
\end{tikzcd}
\tab\item 
\begin{tikzcd} [column sep=tiny, row sep=tiny]  
&&&\bullet \ar[ddddl, dotted, dash] \ar[ddr] \ar[dl] &\\
&\bullet \ar[dl] \ar[dr] \ar[dd, dotted, dash] &\bullet \ar[d] &&\\
\bullet \ar[dr] &&\bullet \ar[dl] &&\bullet \ar[ddll] \\
&\bullet \ar[dr] &&&\\
&&\bullet &&\\
\end{tikzcd}
\tab\item 
\begin{tikzcd}[column sep=tiny, row sep=tiny]
&&\bullet \ar[dl] \ar[dr] \ar[dddl, dotted, dash, bend left=20] &\\
\bullet &\bullet \ar[l] \ar[dd] &&\bullet \ar[d] \ar[ddll] \ar[dddl, dotted, dash, bend right=10] \\
&&&\bullet \ar[d] \\
&\bullet \ar[dr] &&\bullet \ar[dl] \\
&&\bullet &
\end{tikzcd}
\tab\item 
\begin{tikzcd} [column sep=tiny, row sep=tiny] 
&\bullet \ar[dl] \ar[ddr] \ar[ddd, dotted, dash] &&&\\
\bullet \ar[d] &&&\bullet \ar[dddl, dotted, dash] \ar[ddr] \ar[dl] &\\
\bullet \ar[dr] &&\bullet \ar[dl] &&\\
&\bullet \ar[dr] &&&\bullet \ar[dll] \\
&&\bullet &&\\
\end{tikzcd}
\tab\item 
\begin{tikzcd} [column sep=tiny, row sep=tiny] 
&\bullet \ar[ddl] \ar[dr] \ar[ddd, dotted, dash] &&&\\
&&\bullet \ar[d] &&\\
\bullet \ar[dr] &&\bullet \ar[dr] \ar[dl] \ar[dd, dotted, dash] &&\\
&\bullet \ar[dr] &&\bullet \ar[dl] &\bullet \ar[l] \\
&&\bullet &&\\
\end{tikzcd}
\tab\item 
\begin{tikzcd} [column sep=tiny, row sep=tiny] 
&&&\bullet \ar[dll] \ar[ddr] \ar[dddd, dotted, dash] &\\
&\bullet \ar[ddd, dotted, dash] \ar[dr] \ar[dl] &&&\\
\bullet \ar[ddr] &&\bullet \ar[d] &&\bullet \ar[ddl] \\
&&\bullet \ar[dl] \ar[dr] &&\\
&\bullet &&\bullet &
\end{tikzcd}
\tab\item 
\begin{tikzcd} [column sep=tiny, row sep=tiny] 
&&\bullet \ar[ddl] \ar[dr] \ar[ddd, dotted, dash] &&\\
&&&\bullet \ar[d] &\\
\bullet &\bullet \ar[l] \ar[dr] &&\bullet \ar[dr] \ar[dl] \ar[dd, dotted, dash] &\\
&&\bullet \ar[dr] &&\bullet \ar[dl] \\
&&&\bullet &&\\
\end{tikzcd}
\tab\item 
\begin{tikzcd} [column sep=tiny, row sep=tiny]  
&\bullet \ar[d] &&&\\
&\bullet \ar[dl] \ar[dr] \ar[dd, dotted, dash] &&\bullet \ar[dddl, dotted, dash] \ar[ddr] \ar[dl] &\\
\bullet \ar[dr] &&\bullet \ar[dl] &&\\
&\bullet \ar[dr] &&&\bullet \ar[dll] \\
&&\bullet &&\\
\end{tikzcd}
\tab\item 
\begin{tikzcd} [column sep=tiny, row sep=tiny] 
&&\bullet \ar[ddl] \ar[dr] \ar[ddd, dotted, dash] &&\\
&&&\bullet \ar[dr] \ar[ddl] \ar[ddd, dotted, dash] &\\
\bullet \ar[r] &\bullet \ar[dr] &&&\bullet \ar[d] \\
&&\bullet \ar[dr] &&\bullet \ar[dl] \\
&&&\bullet &&\\
\end{tikzcd}
\tab\item 
\begin{tikzcd} [column sep=tiny, row sep=tiny] 
&\bullet \ar[dl] \ar[dr] \ar[dd, dotted, dash] &&&\\
\bullet \ar[dr] &&\bullet \ar[dr] \ar[dl] \ar[dd, dotted, dash] &&\\
&\bullet \ar[dr] &&\bullet \ar[dl] \ar[r] &\bullet \ar[d] \\
&&\bullet &&\bullet \\
\end{tikzcd}
\tab\item 
\begin{tikzcd} [column sep=tiny, row sep=tiny]  
				& \bullet \ar[r] \ar[dddd] 	& \bullet \ar[r] 				& \bullet \ar[ddr] \ar[dddd] 	& \\
&&&&\\
\bullet \ar[ddr] 	&						& \bullet \ar[uul, dotted, dash] 	&							& \bullet \ar[ll] \\
&&&&\\
				& \bullet \ar[rr] 			&							& \bullet \ar[uul] 				& \\
\end{tikzcd}
\tab\item 
\begin{tikzcd} [column sep=tiny, row sep=tiny] 
&\bullet \ar[dl] \ar[dr] \ar[dd, dotted, dash] &&&\\
\bullet \ar[dr] &&\bullet \ar[dr] \ar[dl] \ar[ddd, dotted, dash] &&\\
&\bullet \ar[ddr] &&\bullet \ar[d] &\\
&&&\bullet \ar[r] \ar[dl] &\bullet \\
&&\bullet &&
\end{tikzcd}
\tab\item 
\begin{tikzcd} [column sep=tiny, row sep=tiny]  
							& \bullet \ar[dl] 	&												&				&\\
\bullet \ar[rr, dotted, dash] 	&				& \bullet \ar[rr] \ar[ddll] \ar[ul] 					&				& \bullet \ar[ddll]  \\
&&&& \\
\bullet \ar[rr] \ar[uu] 			&				& \bullet \ar[rr, dotted, dash] \ar[uu, dotted, dash] 	&				& \bullet \ar[uu] \ar[dl]\\
							&				&												& \bullet \ar[ul] 	& \\
\end{tikzcd}
\tab\item 
\begin{tikzcd} [column sep=tiny, row sep=tiny]  
&\bullet \ar[dl] \ar[dr] \ar[dd, dotted, dash] &&\bullet \ar[dddl, dotted, dash] \ar[ddr] \ar[dl] &&\\
\bullet \ar[dr] &&\bullet \ar[dl] &&&\\
&\bullet \ar[dr] &&&\bullet \ar[dll] \ar[r] &\bullet \\
&&\bullet &&&\\
\end{tikzcd}
\tab\item 
\begin{tikzcd} [column sep=tiny, row sep=tiny]  
\bullet \ar[rr] 		&& \bullet \ar[d] 									&&\\
\bullet \ar[dd] \ar[u] 	&& \bullet \ar[ll, dotted, dash] \ar[dd, dotted, dash] 	&& \bullet \ar[dd]  \\
&&&& \\
\bullet \ar[uurr] 		&& \bullet \ar[ll] \ar[rr] 							&& \bullet \ar[uull]\\
\end{tikzcd}
\tab\item 
\begin{tikzcd} [column sep=tiny, row sep=tiny]  
\bullet \ar[rr] 				&& \bullet 	 					&&\\
\bullet \ar[dd] \ar[u] 			&& \bullet \ar[ll] \ar[dd] \ar[rr] 	&& \bullet \ar[dd]  \\
&&&& \\
\bullet \ar[uurr, dotted, dash] 	&& \bullet \ar[ll] \ar[rr] 			&& \bullet \ar[uull, dotted, dash]\\
\end{tikzcd}
\tab\item 
\begin{tikzcd} [column sep=tiny, row sep=tiny]  
\bullet \ar[dd] \ar[rr] 		&& \bullet 	 					&&\\
&&&&\\
\bullet \ar[rr, dotted, dash] && \bullet \ar[uull] \ar[dd] \ar[rr] 	&& \bullet \ar[dd]  \\
&&&& \\
\bullet \ar[uu] 			&& \bullet \ar[ll] \ar[rr] 			&& \bullet \ar[uull, dotted, dash]\\
\end{tikzcd}
\tab\item 
\begin{tikzcd} [column sep=tiny, row sep=tiny]  
\bullet \ar[dd] \ar[rr] 	&& \bullet \ar[dd] 									&&\\
&&&&\\
\bullet \ar[rr] 		&& \bullet \ar[uull, dotted, dash] \ar[dd, dotted, dash] 	&& \bullet \ar[dd]  \\
&&&& \\
\bullet \ar[uu] 		&& \bullet \ar[ll] \ar[rr] 								&& \bullet \ar[uull]\\
\end{tikzcd}
\tab\item 
\begin{tikzcd} [column sep=tiny, row sep=tiny]  
			& \bullet \ar[ddl] 			&												& \bullet \ar[ddl] 		&\\
&&&&\\
\bullet \ar[r] 	& \bullet \ar[r] 			& \bullet \ar[ddl, dotted, dash] \ar[rr, dotted, dash] 	&					& \bullet \ar[ddl] \ar[uul] \\
&&&& \\
			& \bullet \ar[uul] \ar[rr] 	&												& \bullet \ar[uul]  	& \\
\end{tikzcd}
\tab\item 
\begin{tikzcd} [column sep=tiny, row sep=tiny]  
&&\bullet \ar[ddl] &&\bullet &\\
&&&&&\\
\bullet &\bullet \ar[l] \ar[rr, dotted, dash] &&\bullet \ar[uul] \ar[ddl] \ar[rr] &&\bullet \ar[ddl] \ar[uul] \\
&&&&& \\
&&\bullet \ar[uul] \ar[rr] &&\bullet \ar[uul, dotted, dash] & \\
\end{tikzcd}
\tab\item 
\begin{tikzcd} [column sep=tiny, row sep=tiny]  
\bullet \ar[r] &\bullet \ar[ddl, dotted, dash] \ar[ddrr, dotted, dash] &&\bullet \ar[ll] &\bullet \ar[l]\\
&&&&\\
\bullet \ar[uu] \ar[r] &\bullet \ar[uu] &&\bullet \ar[ll] \ar[uu] &\bullet \ar[uu] \ar[l] \ar[uul, dotted, dash] \\
\end{tikzcd}
\tab\item 
\begin{tikzcd} [column sep=tiny, row sep=tiny] 
&&\bullet \ar[ddd, dotted, dash] \ar[dr] \ar[dl] &&&\\
\bullet &\bullet \ar[l] \ar[ddr] &&\bullet \ar[d] \ar[ddr, dotted, dash, bend left=10] \ar[drr] &&\\
&&&\bullet \ar[dl] \ar[dr] &&\bullet \ar[dl] \\
&&\bullet &&\bullet &
\end{tikzcd}
\tab\item 
\begin{tikzcd} [column sep=tiny, row sep=tiny] 
&&\bullet \ar[dl] \ar[ddr] \ar[ddd, dotted, dash] &\\
&\bullet \ar[d] &&\\
&\bullet \ar[dl] \ar[ddd, dotted, dash] \ar[dr] &&\bullet \ar[dl] \\
\bullet \ar[d] &&\bullet \ar[ddl] &\\
\bullet \ar[dr] &&&\\
&\bullet &&\\
\end{tikzcd}
\tab\item 
\begin{tikzcd} [column sep=tiny, row sep=tiny]  
&&\bullet \ar[ddl] &&\bullet \ar[ll] &\\
&&&&& \\
\bullet &\bullet \ar[l] \ar[rrr, dotted, dash] &&& \bullet \ar[uull] \ar[dl] &\bullet \ar[l] \\
&&\bullet \ar[ul] &\bullet \ar[l] && \\
\end{tikzcd}
\tab\item 
\begin{tikzcd} [column sep=tiny, row sep=tiny] 
\bullet \ar[dddr, dotted, dash] \ar[ddd] \ar[r] &\bullet \ar[ddd] \ar[dddrrr, dotted, dash] \ar[rrr] &&&\bullet \ar[ddd] \ar[dddr, dotted, dash] \ar[r] &\bullet \ar[ddd] \\
&&&&&\\
&&&&&\\
\bullet \ar[r] &\bullet \ar[rrr] &&&\bullet \ar[r] &\bullet \\
\end{tikzcd}
\tab\item 
\begin{tikzcd} [column sep=tiny, row sep=tiny] 
&\bullet \ar[ddd] \ar[dddr, dotted, dash] \ar[r] &\bullet \ar[dddrr, dotted, dash] \ar[ddd] \ar[rr] &&\bullet \ar[ddd] &\\
&&&&&\\
&&&&&\\
\bullet \ar[r] &\bullet \ar[r] &\bullet \ar[rr] &&\bullet \ar[r] &\bullet \\
\end{tikzcd}
\tab\item 
\begin{tikzcd} [column sep=tiny, row sep=tiny] 
\bullet \ar[r] &\bullet \ar[ddd] \ar[dddr, dotted, dash] \ar[r] &\bullet \ar[dddrr, dotted, dash] \ar[ddd] \ar[rr] &&\bullet \ar[ddd] \ar[r] &\bullet \\
&&&&&\\
&&&&&\\
&\bullet \ar[r] &\bullet \ar[rr] &&\bullet &\\
\end{tikzcd}
\tab\item 
\begin{tikzcd} [column sep=tiny, row sep=tiny] 
\bullet \ar[d] &				&							& \bullet \ar[dl] 	&\\
\bullet \ar[r] 	& \bullet \ar[r] 	& \bullet \ar[rr, dotted, dash] 	&				& \bullet \ar[dd]  \ar[lu]\\
			&				&							&				&\\
			&				& \bullet \ar[uu] 				&				& \bullet \ar[ll] \\
\end{tikzcd}
\tab\item 
\begin{tikzcd} [column sep=tiny, row sep=tiny] 
\bullet \ar[d]		&							& \bullet \ar[dl] 	&\\
\bullet \ar[r]		& \bullet \ar[rr, dotted, dash] 	&				& \bullet \ar[dd]  \ar[lu]\\
				&							&				&\\
\bullet \ar[uur] 	& \bullet \ar[uu] 				&				& \bullet \ar[ll] \\
\end{tikzcd}
\tab\item 
\begin{tikzcd} [column sep=tiny, row sep=tiny] 
\bullet 				&							& \bullet \ar[dl] 	&\\
\bullet \ar[r] \ar[u]	& \bullet \ar[rr, dotted, dash] 	&				& \bullet \ar[dd]  \ar[lu]\\
					&							&				&\\
\bullet \ar[uur]		& \bullet \ar[uu] 				&				& \bullet \ar[ll] \\
\end{tikzcd}
\tab\item 
\begin{tikzcd} [column sep=tiny, row sep=tiny] 
\bullet \ar[d]	&							& \bullet \ar[dl] 	&							&\\
\bullet \ar[r]	& \bullet \ar[rr, dotted, dash] 	&				& \bullet \ar[dd]  \ar[lu] \ar[r] 	& \bullet \\
			&							&				&							&\\
			& \bullet \ar[uu] 				&				& \bullet \ar[ll] 				&\\
\end{tikzcd}
\tab\item 
\begin{tikzcd} [column sep=tiny, row sep=tiny] 	
\bullet \ar[dr]	&							& \bullet \ar[dl] 	&							&\\
				& \bullet \ar[rr, dotted, dash] 	&				& \bullet \ar[dd]  \ar[lu] \ar[r] 	& \bullet \\
\bullet \ar[ur]	&							&				&							&\\
				& \bullet \ar[uu] 				&				& \bullet \ar[ll] 				&\\
\end{tikzcd}
\tab\item 
\begin{tikzcd} [column sep=tiny, row sep=tiny] 
					&							& \bullet \ar[dl] 	&							&\\
\bullet \ar[r] \ar[d]	& \bullet \ar[rr, dotted, dash] 	&				& \bullet \ar[dd]  \ar[lu] \ar[r] 	& \bullet \\
\bullet 				&							&				&							&\\
					& \bullet \ar[uu] 				&				& \bullet \ar[ll] 				&\\
\end{tikzcd}
\tab\item 
\begin{tikzcd} [column sep=tiny, row sep=tiny] 
			&							& \bullet \ar[dl] 	&\\
\bullet \ar[r] 	& \bullet \ar[rr, dotted, dash]   	&				& \bullet \ar[d]  \ar[lu]\\
			& \bullet \ar[u] 				&				& \bullet \ar[d]\\
			& \bullet \ar[u] 				&				& \bullet \ar[ll]\\
\end{tikzcd}
\tab\item 
\begin{tikzcd} [column sep=tiny, row sep=tiny] 
\bullet \ar[rr] &&\bullet \ar[ddll, dotted, dash] &\\
&&&\bullet \ar[dl] \\
\bullet \ar[uu] \ar[rr] \ar[d] &&\bullet \ar[uu] &\\
\bullet \ar[r] &\bullet \ar[r] &\bullet &
\end{tikzcd}
\tab\item 
\begin{tikzcd} [column sep=tiny, row sep=tiny] 
\bullet \ar[rr] &&\bullet \ar[ddll, dotted, dash] &\\
&&&\bullet \ar[dl] \\
\bullet \ar[uu] \ar[rr] \ar[d] \ar[dr] &&\bullet \ar[uu] &\\
\bullet &\bullet \ar[r] & \bullet &
\end{tikzcd}
\tab\item 
\begin{tikzcd} [column sep=tiny, row sep=tiny] 
\bullet \ar[rr] 					&					& \bullet \ar[ddll, dotted, dash] 	&\\
								&					&								& \bullet \ar[dl]\\
\bullet \ar[uu] \ar[rr] \ar[d] \ar[drr]	&					& \bullet \ar[uu] 					&\\
\bullet 							&\bullet \ar[r]					& \bullet
\end{tikzcd}

\noindent
\item 
\begin{tikzcd} [column sep=tiny, row sep=tiny] 
							& \bullet 	& \bullet \ar[l]				&\\
							&			& \bullet \ar[u]				&\\
\bullet \ar[dd] 				&			& \bullet \ar[ll] \ar[dd] \ar[u]	&\\
							&			&							& \bullet\\
\bullet \ar[uurr, dotted, dash] 	&			& \bullet \ar[ll] \ar[ur] 		&\\
\end{tikzcd}
\tab\item 
\begin{tikzcd} [column sep=tiny, row sep=tiny] 
				& \bullet \ar[r]	& \bullet \ar[d]				&\\
				&				& \bullet \ar[d]				&\\
\bullet \ar[dd] 	&				& \bullet \ar[dd, dotted, dash] 	&\\
				&				&							& \bullet \ar[ul]\\
\bullet \ar[uurr] 	&				& \bullet \ar[ll] \ar[ur] 		&\\
\end{tikzcd}
\tab\item 
\begin{tikzcd} [column sep=tiny, row sep=tiny] 
							& \bullet 	& \bullet \ar[r]						& \bullet \\
\bullet \ar[dd] 				&			& \bullet \ar[ll] \ar[dd] \ar[ul] \ar[u] 	&\\
							&			&									& \bullet\\
\bullet \ar[uurr, dotted, dash] 	&			& \bullet \ar[ll] \ar[ur] 				&\\
\end{tikzcd}
\tab\item 
\begin{tikzcd} [column sep=tiny, row sep=tiny] 
				& \bullet \ar[dr]						& \bullet \ar[d]					& \bullet \ar[l] \\
\bullet \ar[dd] 	&									& \bullet \ar[dd, dotted, dash] 		&\\
				&									&								& \bullet \ar[ul]\\
\bullet \ar[uurr] 	&									& \bullet \ar[ll] \ar[ur] 			&\\
\end{tikzcd}
\tab\item 
\begin{tikzcd} [column sep=tiny, row sep=tiny] 
							& 		&  \bullet \ar[dr] 						& \\
							& \bullet	& 									& \bullet \\
\bullet \ar[dd] 				&		& \bullet \ar[ll] \ar[dd] \ar[ul] \ar[ur] 	&\\
							&		&									& \bullet\\
\bullet \ar[uurr, dotted, dash] 	&		& \bullet \ar[ll] \ar[ur] 				&\\
\end{tikzcd}
\tab\item 
\begin{tikzcd} [column sep=tiny, row sep=tiny] 
				& 				&  \bullet 					& \\
				& \bullet	\ar[dr]	& 							& \bullet \ar[dl] \ar[ul] \\
\bullet \ar[dd] 	&				& \bullet \ar[dd, dotted, dash] 	&\\
				&				&							& \bullet \ar[ul]\\
\bullet \ar[uurr] 	&				& \bullet \ar[ll] \ar[ur] 		&\\
\end{tikzcd}
\tab\item 
\begin{tikzcd} [column sep=tiny, row sep=tiny] 
		& \bullet \ar[rr]				&& \bullet \ar[d]					&\\
		& \bullet \ar[rr] 				&& \bullet \ar[ddll, dotted, dash] 	&\\
		&							&&								& \bullet \ar[dl]\\
\bullet 	& \bullet \ar[uu] \ar[rr] \ar[l] 	&& \bullet \ar[uu] 				&\\
\end{tikzcd}
\tab\item 
\begin{tikzcd} [column sep=tiny, row sep=tiny] 
		& \bullet \ar[drr]				&&								& \bullet \ar[dl]\\
		& \bullet \ar[rr] 				&& \bullet \ar[ddll, dotted, dash] 	&\\
		&							&&								& \bullet \ar[dl]\\
\bullet 	& \bullet \ar[uu] \ar[rr] \ar[l] 	&& \bullet \ar[uu] 				&\\
\end{tikzcd}
\tab\item 
\begin{tikzcd} [column sep=tiny, row sep=tiny] 
		& \bullet 					&& \bullet \ar[d] \ar[ll]			&\\
		& \bullet \ar[rr] 				&& \bullet \ar[ddll, dotted, dash] 	&\\
		&							&&								& \bullet \ar[dl]\\
\bullet 	& \bullet \ar[uu] \ar[rr] \ar[l] 	&& \bullet \ar[uu] 				&\\
\end{tikzcd}
\tab\item 
\begin{tikzcd} [column sep=tiny, row sep=tiny] 
\bullet \ar[rr]			&& \bullet 						&\\
\bullet \ar[dd] \ar[u]	&& \bullet \ar[dd, dotted, dash]	&\\
					&&								& \bullet \ar[ul]\\
\bullet \ar[uurr]	 	&& \bullet \ar[ll] \ar[ur] \ar[d] 		&\\
					&& \bullet 	 					&\\
\end{tikzcd}
\tab\item 
\begin{tikzcd} [column sep=tiny, row sep=tiny] 
\bullet							&& \bullet 						& \\
\bullet \ar[dd]  \ar[u]  \ar[urr]		&& \bullet \ar[dd, dotted, dash]	&\\
								&&								& \bullet \ar[ul]\\
\bullet \ar[uurr]	 				&& \bullet \ar[ll] \ar[ur] \ar[d] 		&\\
								&& \bullet 	 					&\\
\end{tikzcd}
\tab\item 
\begin{tikzcd} [column sep=tiny, row sep=tiny] 
\bullet 				&& \bullet \ar[ll]					&\\
\bullet \ar[dd] \ar[u]	&& \bullet \ar[dd, dotted, dash]	&\\
					&&								& \bullet \ar[ul]\\
\bullet \ar[uurr]	 	&& \bullet \ar[ll] \ar[ur] \ar[d] 		&\\
					&& \bullet 	 					&\\
\end{tikzcd}
\tab\item 
\begin{tikzcd} [column sep=tiny, row sep=tiny] 
\bullet \ar[rr]				&& \bullet 						&\\
						&& \bullet \ar[d] 					&\\
\bullet \ar[dd] \ar[uu]		&& \bullet \ar[dd, dotted, dash] 	&\\
						&&								& \bullet \ar[ul]\\
\bullet \ar[uurr] 			&& \bullet \ar[ll] \ar[ur]			&\\
\end{tikzcd}
\tab\item 
\begin{tikzcd} [column sep=tiny, row sep=tiny] 
\bullet						& \bullet 	& \bullet \ar[d] 					&\\
\bullet \ar[dd] \ar[u]  \ar[ur]	&			& \bullet \ar[dd, dotted, dash] 	&\\
							&			&								& \bullet \ar[ul]\\
\bullet \ar[uurr] 				&			& \bullet \ar[ll] \ar[ur]			&\\
\end{tikzcd}
\tab\item 
\begin{tikzcd} [column sep=tiny, row sep=tiny] 
\bullet && \bullet \ar[ll]					&\\
						&& \bullet \ar[d] 					&\\
\bullet \ar[dd] \ar[uu] 	&& \bullet \ar[dd, dotted, dash] 	&\\
						&&								& \bullet \ar[ul]\\
\bullet \ar[uurr] 			&& \bullet \ar[ll] \ar[ur]			&\\
\end{tikzcd}
\tab\item 
\begin{tikzcd} [column sep=tiny, row sep=tiny] 
							&& \bullet \ar[d] 					&\\
\bullet \ar[rr] 				&& \bullet \ar[ddll, dotted, dash] 	&\\
							&&								& \bullet \ar[dl]\\
\bullet \ar[uu] \ar[rr] \ar[d]	&& \bullet \ar[uu] 				&\\
\bullet \ar[rr]					&& \bullet 						&\\
\end{tikzcd}
\tab\item 
\begin{tikzcd} [column sep=tiny, row sep=tiny] 
									&& \bullet \ar[d] 					&\\
\bullet \ar[rr] 						&& \bullet \ar[ddll, dotted, dash] 	&\\
									&&								& \bullet \ar[dl]\\
\bullet \ar[uu] \ar[rr] \ar[d] \ar[drr] 	&& \bullet \ar[uu] 				&\\
\bullet 								&& \bullet 						&\\
\end{tikzcd}
\tab\item 
\begin{tikzcd} [column sep=tiny, row sep=tiny] 
							&& \bullet \ar[d] 					&\\
\bullet \ar[rr] 				&& \bullet \ar[ddll, dotted, dash] 	&\\
							&&								& \bullet \ar[dl]\\
\bullet \ar[uu] \ar[rr] \ar[d] 	&& \bullet \ar[uu] 				&\\
\bullet 						&& \bullet \ar[ll]					&\\
\end{tikzcd}
\tab\item 
\begin{tikzcd} [column sep=tiny, row sep=tiny] 
		& \bullet \ar[rr]				&& \bullet \ar[d]					&\\
		& \bullet \ar[rr] 				&& \bullet \ar[ddll, dotted, dash]  	&\\
		&							&&								& \bullet \ar[dl]\\
\bullet  	& \bullet \ar[l] \ar[uu] \ar[rr] 	&& \bullet \ar[uu] 				&\\
\end{tikzcd}
\tab\item 
\begin{tikzcd} [column sep=tiny, row sep=tiny] 
		& 							& \bullet \ar[dr]	& 								& \bullet \ar[dl] \\
		& \bullet \ar[rr] 				&				& \bullet \ar[ddll, dotted, dash]  	&\\
		&							&				&								& \bullet \ar[dl]\\
\bullet  	& \bullet \ar[l] \ar[uu] \ar[rr] 	&				& \bullet \ar[uu] 					&\\
\end{tikzcd}
\tab\item 
\begin{tikzcd} [column sep=tiny, row sep=tiny] 
		& \bullet 					&& \bullet \ar[d] \ar[ll]			&\\
		& \bullet \ar[rr] 				&& \bullet \ar[ddll, dotted, dash]  	&\\
		&							&&								& \bullet \ar[dl]\\
\bullet  	& \bullet \ar[l] \ar[uu] \ar[rr] 	&& \bullet \ar[uu] 				&\\
\end{tikzcd}
\tab\item 
\begin{tikzcd} [column sep=tiny, row sep=tiny]
				&& \bullet \ar[d] 					&&\\
\bullet \ar[rr] 	&& \bullet \ar[dd, dotted, dash]  	&& \bullet \ar[dd] \\
&&&&\\
\bullet \ar[uu] 	&& \bullet \ar[ll] \ar[d] \ar[rr] 		&& \bullet \ar[uull] \\
				&& \bullet 						&& \\
\end{tikzcd}
\tab\item 
\begin{tikzcd} [column sep=tiny, row sep=tiny]
				&&								&& \bullet \\
\bullet \ar[rr] 	&& \bullet \ar[dd, dotted, dash]  	&& \bullet \ar[dd] \ar[u]\\
&&&&\\
\bullet \ar[uu] 	&& \bullet \ar[ll] \ar[d] \ar[rr] 		&& \bullet \ar[uull] \\
				&& \bullet 						&& \\
\end{tikzcd}
\tab\item 
\begin{tikzcd} [column sep=tiny, row sep=tiny] 
\bullet 		 		&& \bullet \ar[d] 					&&\\
\bullet \ar[dd] \ar[u] 	&& \bullet \ar[ddrr, dotted, dash] 	&& \bullet \ar[ll]\\
&&&&\\
\bullet \ar[rr] 		&& \bullet \ar[uu] 				&& \bullet \ar[ll] \ar[uu] \\
\end{tikzcd}
\tab\item  
\begin{tikzcd} [column sep=tiny, row sep=tiny] 
				&& \bullet 	 					&& \bullet \ar[d] \\
\bullet \ar[dd] 	&& \bullet \ar[ll] \ar[ddrr] \ar[u] 	&& \bullet \\
&&&&\\
\bullet \ar[rr] 	&& \bullet \ar[uu, dotted, dash] 	&& \bullet \ar[ll] \ar[uu] \\
\end{tikzcd}
\tab\item 
\begin{tikzcd} [column sep=tiny, row sep=tiny] 
\bullet 				&&								&&\\
\bullet \ar[dd] \ar[u] 	&& \bullet \ar[ddrr, dotted, dash] 	&& \bullet \ar[ll]\\
&&&&\\
\bullet \ar[rr] 		&& \bullet \ar[uu] 				&& \bullet \ar[ll] \ar[uu] \ar[d]\\
					&&								&& \bullet\\
\end{tikzcd}
\tab \item 
\begin{tikzcd} [column sep=tiny, row sep=tiny] 
				&				& \bullet \ar[dd] 				&& \bullet \ar[ll] \\
				& \bullet \ar[dr] 	&							&&\\
\bullet \ar[dd] 	&				& \bullet \ar[rr, dotted, dash] 	&& \bullet \ar[uu] \ar[ddll] \\
&&&&\\
\bullet \ar[rr] 	&				& \bullet \ar[uu]  				&& \\
\end{tikzcd}
\tab\item  
\begin{tikzcd} [column sep=tiny, row sep=tiny] 
				&				& \bullet 	 				&& \bullet \ar[ll] \\
				& \bullet 	 	&							&&\\
\bullet \ar[dd] 	&				& \bullet \ar[ul] \ar[ll] \ar[rr] 	&& \bullet \ar[uu] \ar[ddll] \\
&&&&\\
\bullet \ar[rr] 	&				& \bullet \ar[uu, dotted, dash] 	&& \\
\end{tikzcd}
\tab\item 
\begin{tikzcd} [column sep=tiny, row sep=tiny] 
				&& \bullet \ar[rr] 					&& \bullet \ar[dd]\\
&&&&\\
\bullet \ar[rr] 	&& \bullet \ar[dd, dotted, dash] 	&& \bullet \ar[ll]\\
&&&&\\
\bullet \ar[uu] 	&& \bullet \ar[ll] \ar[d] \ar[uurr] 	&& \\
				&& \bullet 						&&\\
\end{tikzcd}
\tab\item 
\begin{tikzcd} [column sep=tiny, row sep=tiny] 
\bullet \ar[d] &&\bullet \ar[dl] &\bullet \ar[d] \\
\bullet \ar[r] &\bullet \ar[rr, dotted, dash] &&\bullet \ar[ul] \ar[d]\\
&\bullet \ar[u] &&\bullet \ar[ll]\\
\end{tikzcd}
\tab\item 
\begin{tikzcd} [column sep=tiny, row sep=tiny] 
\bullet \ar[dr] &&\bullet \ar[dl] &\bullet \ar[d]\\
&\bullet \ar[rr, dotted, dash] &&\bullet \ar[ul] \ar[d]\\
\bullet \ar[ur] &\bullet \ar[u] &&\bullet \ar[ll]\\
\end{tikzcd}
\tab\item 
\begin{tikzcd} [column sep=tiny, row sep=tiny] 
&&\bullet \ar[dl] &\bullet \ar[d]\\
\bullet \ar[d] \ar[r] &\bullet \ar[rr, dotted, dash] && \bullet \ar[ul] \ar[d]\\
\bullet &\bullet \ar[u] &&\bullet \ar[ll]\\
\end{tikzcd}
\tab\item 
\begin{tikzcd} [column sep=tiny, row sep=tiny]
&&&\bullet &\bullet \ar[l] \\
&&\bullet \ar[dll] &&\bullet \ar[u] \ar[d]\\
\bullet \ar[rrrr, dotted, dash] &&&&\bullet \ar[ull] \ar[dd]\\
&&&&\\			
\bullet \ar[uu] &&&&\bullet \ar[llll]\\
\end{tikzcd}
\tab\item 
\begin{tikzcd} [column sep=tiny, row sep=tiny]
&&&\bullet &\bullet \\
&&\bullet \ar[dll] &&\bullet \ar[d] \ar[u] \ar[ul] \\
\bullet \ar[rrrr, dotted, dash] &&&& \bullet \ar[ull] \ar[dd] \\
&&&&\\				
\bullet \ar[uu] &&&&\bullet \ar[llll]\\
\end{tikzcd}
\tab\item 
\begin{tikzcd} [column sep=tiny, row sep=tiny]
&&&\bullet \ar[r] &\bullet \\
&&\bullet \ar[dll] &&\bullet \ar[d] \ar[u] \\
\bullet \ar[rrrr, dotted, dash] &&&&\bullet \ar[ull] \ar[dd]\\
&&&&\\			
\bullet \ar[uu] &&&&\bullet \ar[llll]\\
\end{tikzcd}
\tab\item 
\begin{tikzcd} [column sep=tiny, row sep=tiny]
&&\bullet \ar[ddl] &\bullet \\
&&&\bullet \ar[d] \ar[u] \\
\bullet \ar[r] &\bullet \ar[rr, dotted, dash] && \bullet \ar[uul] \ar[d] \\
&\bullet \ar[u] &&\bullet \ar[ll]\\
\end{tikzcd}
\tab\item 
\begin{tikzcd} [column sep=tiny, row sep=tiny]
				& \bullet \ar[r] 			& \bullet						&\\
				& \bullet \ar[ddl] 	\ar[u] 	&							&\\
&&&\\
\bullet \ar[rr] 	&						& \bullet \ar[dd, dotted, dash] 	&\\
				&						&							& \bullet \ar[ul] \\
\bullet \ar[uu] 	&						& \bullet \ar[ll] \ar[ur] 		&\\
\end{tikzcd}
\tab\item 
\begin{tikzcd} [column sep=tiny, row sep=tiny]
\bullet 	 		& 							& \bullet						&\\
				& \bullet \ar[ddl] \ar[ul] \ar[ur]	&							&\\
&&&\\
\bullet \ar[rr] 	&							& \bullet \ar[dd, dotted, dash] 	&\\
				&							&							& \bullet \ar[ul] \\
\bullet \ar[uu] 	&							& \bullet \ar[ll] \ar[ur] 		&\\
\end{tikzcd}
\tab\item 
\begin{tikzcd} [column sep=tiny, row sep=tiny]
				& \bullet 	 			& \bullet	\ar[l]				&\\
				& \bullet \ar[ddl] 	\ar[u]	&							&\\
&&&\\
\bullet \ar[rr] 	&						& \bullet \ar[dd, dotted, dash] 	&\\
				&						&							& \bullet \ar[ul] \\
\bullet \ar[uu] 	&						& \bullet \ar[ll] \ar[ur] 		&\\
\end{tikzcd}
\tab\item 
\begin{tikzcd} [column sep=tiny, row sep=tiny]
			&							& \bullet \ar[ddl] 	&						&\\
\bullet \ar[d] &&&&\\
\bullet \ar[r]	& \bullet \ar[rr, dotted, dash] 	&				& \bullet \ar[uul] \ar[dd] 	&\\
			&							&				&						& \bullet \\
			& \bullet \ar[uu] 				&				& \bullet \ar[ll] \ar[ur] 	&\\
\end{tikzcd}
\tab\item 
\begin{tikzcd} [column sep=tiny, row sep=tiny]
				&							& \bullet \ar[ddl] 	&						&\\
\bullet \ar[dr]	&&&&\\
				& \bullet \ar[rr, dotted, dash] 	&				& \bullet \ar[uul] \ar[dd] 	&\\
\bullet \ar[ur]	&							&				&						& \bullet \\
				& \bullet \ar[uu] 				&				& \bullet \ar[ll] \ar[ur] 	&\\
\end{tikzcd}
\tab\item 
\begin{tikzcd} [column sep=tiny, row sep=tiny]
					&							& \bullet \ar[ddl] 	&						&\\
					&&&&\\
\bullet \ar[d] \ar[r] 	& \bullet \ar[rr, dotted, dash] 	&				& \bullet \ar[uul] \ar[dd] 	&\\
\bullet 				&							&				&						& \bullet \\
					& \bullet \ar[uu] 				&				& \bullet \ar[ll] \ar[ur] 	&\\
\end{tikzcd}
\tab\item 
\begin{tikzcd} [column sep=tiny, row sep=tiny]
\bullet &\bullet \ar[l] &&\\
&&&\\
\bullet \ar[rr] &&\bullet \ar[uul] \ar[dd, dotted, dash] &\\
&&&\bullet \ar[ul] \\
\bullet \ar[uu] &&\bullet \ar[ll] \ar[d] \ar[ur] &\\
&&\bullet &\\
\end{tikzcd}
\tab\item 
\begin{tikzcd} [column sep=tiny, row sep=tiny] 
\bullet \ar[d] &\bullet \ar[rr] \ar[dd] \ar[l] &&\bullet \ar[r] \ar[dd] &\bullet \ar[dd]\\
\bullet &&&&\\
&\bullet \ar[rr] &&\bullet \ar[r] \ar[ulul, dotted, dash] &\bullet \ar[ulu, dotted, dash]\\
\end{tikzcd}
\tab\item 
\begin{tikzcd} [column sep=tiny, row sep=tiny] 
\bullet \ar[d] &\bullet \ar[dd] &&\bullet \ar[ddll, dotted, dash] \ar[dd] \ar[ll] &\\
\bullet \ar[d] &&&&\\
\bullet &\bullet \ar[l] &&\bullet \ar[ll] &\bullet \ar[l]\\
\end{tikzcd}
\tab\item 
\begin{tikzcd} [column sep=tiny, row sep=tiny] 
\bullet &\bullet \ar[dl] \ar[rr] \ar[dd] \ar[l] &&\bullet \ar[r] \ar[dd] &\bullet \ar[dd]\\
\bullet &&&&\\
&\bullet \ar[rr] &&\bullet \ar[r] \ar[ulul, dotted, dash] &\bullet \ar[ulu, dotted, dash]\\
\end{tikzcd}
\tab\item 
\begin{tikzcd} [column sep=tiny, row sep=tiny] 
&&\bullet \ar[dd] &&\bullet \ar[ddll, dotted, dash] \ar[dd] \ar[ll] &\\
&\bullet \ar[d] &&&&\\
\bullet \ar[r] &\bullet &\bullet \ar[l] &&\bullet \ar[ll] &\bullet \ar[l]\\
\end{tikzcd}
\tab\item 
\begin{tikzcd} [column sep=tiny, row sep=tiny] 
\bullet &\bullet \ar[rr] \ar[dd] \ar[l] &&\bullet \ar[r] \ar[dd] &\bullet \ar[dd]\\
\bullet \ar[u] &&&&\\
&\bullet \ar[rr] &&\bullet \ar[r] \ar[ulul, dotted, dash] &\bullet \ar[ulu, dotted, dash]\\
\end{tikzcd}
\tab\item  
\begin{tikzcd} [column sep=tiny, row sep=tiny] 
\bullet &\bullet \ar[dd] &&\bullet \ar[ddll, dotted, dash] \ar[dd] \ar[ll] &\\
\bullet \ar[u] \ar[d] &&&&\\
\bullet &\bullet \ar[l] &&\bullet \ar[ll] &\bullet \ar[l]\\
\end{tikzcd}
\tab\item 
\begin{tikzcd} [column sep=tiny, row sep=tiny] 
&\bullet \ar[dd] &&\bullet \ar[ddll, dotted, dash] \ar[dd] \ar[ll] &\bullet \\
&&&&\bullet \ar[u] \\
\bullet &\bullet \ar[l] &&\bullet \ar[ll] &\bullet \ar[l] \ar[u] \\
\end{tikzcd}
\tab\item 
\begin{tikzcd} [column sep=tiny, row sep=tiny] 
&\bullet \ar[dd] &&\bullet \ar[ddll, dotted, dash] \ar[dd] \ar[ll] &&\\
&&&&\bullet &\\
\bullet &\bullet \ar[l] &&\bullet \ar[ll] &\bullet \ar[l] \ar[r] \ar[u] &\bullet \\
\end{tikzcd}
\tab\item 
\begin{tikzcd} [column sep=tiny, row sep=tiny] 
&\bullet \ar[dd] &&\bullet \ar[ddll, dotted, dash] \ar[dd] \ar[ll] &\bullet \ar[d] \\
&&&&\bullet \\
\bullet &\bullet \ar[l] &&\bullet \ar[ll] &\bullet \ar[l] \ar[u] \\
\end{tikzcd}
\tab\item 
\begin{tikzcd} [column sep=tiny, row sep=tiny] 
&\bullet \ar[dd] &&\bullet \ar[ddll, dotted, dash] \ar[dd] \ar[r] \ar[ll] &\bullet \ar[d] \\
&&&&\bullet \\
\bullet &\bullet \ar[l] &&\bullet \ar[ll] &\bullet \ar[l] \\
\end{tikzcd}
\tab\item 
\begin{tikzcd} [column sep=tiny, row sep=tiny] 
&\bullet \ar[dd] &&\bullet \ar[ddll, dotted, dash] \ar[dd] \ar[dr] \ar[r] \ar[ll] &\bullet \\
&&&&\bullet \\
\bullet &\bullet \ar[l] &&\bullet \ar[ll] &\bullet \ar[l] \\
\end{tikzcd}
\tab\item 
\begin{tikzcd} [column sep=tiny, row sep=tiny] 
&\bullet \ar[dd] &&\bullet \ar[ddll, dotted, dash] \ar[dd] \ar[r] \ar[ll] &\bullet \\
&&&&\bullet \ar[u] \\
\bullet &\bullet \ar[l] &&\bullet \ar[ll] &\bullet \ar[l] \\
\end{tikzcd}
\tab\item 
\begin{tikzcd} [column sep=tiny, row sep=tiny] 
&\bullet \ar[dd] &&\bullet \ar[ddll, dotted, dash] \ar[dd] \ar[r] \ar[ll] &\bullet \\
\bullet \ar[d] &&&&\\
\bullet &\bullet \ar[l] &&\bullet \ar[ll] &\bullet \ar[l] \\
\end{tikzcd}
\tab\item 
\begin{tikzcd} [column sep=tiny, row sep=tiny] 
\bullet &\bullet \ar[rr] \ar[dd] \ar[l] &&\bullet \ar[r] \ar[dd] &\bullet \ar[dd] &\\
&&&&&\\
&\bullet \ar[rr] &&\bullet \ar[r] \ar[ulul, dotted, dash] &\bullet \ar[ulu, dotted, dash] &\bullet \ar[l] \\
\end{tikzcd}
\tab\item  
\begin{tikzcd} [column sep=tiny, row sep=tiny] 
&\bullet \ar[dd] &&\bullet \ar[ddll, dotted, dash] \ar[dd] \ar[r] \ar[ll] &\bullet \\
&&&&\bullet \\
\bullet &\bullet \ar[l] &&\bullet \ar[ll] &\bullet \ar[l] \ar[u] \\
\end{tikzcd}
\tab\item 
\begin{tikzcd} [column sep=tiny, row sep=tiny] 
			&						& \bullet \ar[dr] 	&\\
			& \bullet \ar[dr] \ar[ur] 	&				& \bullet \ar[ll, dotted, dash]\\
\bullet \ar[d]	&						& \bullet \ar[dr] 	&\\
\bullet \ar[r]	& \bullet  				&				& \bullet \ar[ll] \ar[uu]\\
\end{tikzcd}
\tab\item 
\begin{tikzcd}  [column sep=tiny, row sep=tiny] 
				&						& \bullet \ar[dr] 	&\\
				& \bullet \ar[dr] \ar[ur] 	&				& \bullet \ar[ll, dotted, dash]\\
\bullet \ar[dr]	&						& \bullet \ar[dr] 	&\\
\bullet \ar[r]		& \bullet  				&				& \bullet \ar[ll] \ar[uu]\\
\end{tikzcd}
\tab\item  
\begin{tikzcd}  [column sep=tiny, row sep=tiny] 
					&						& \bullet \ar[dr] 	&\\
					& \bullet \ar[dr] \ar[ur] 	&				& \bullet \ar[ll, dotted, dash]\\
\bullet 				&						& \bullet \ar[dr] 	&\\
\bullet \ar[u] \ar[r]	& \bullet  				&				& \bullet \ar[ll] \ar[uu]\\
\end{tikzcd}
\tab\item 
\begin{tikzcd}  [column sep=tiny, row sep=tiny]  
\bullet &\bullet \ar[l] &\bullet \ar[d] \ar[l] &\\
\bullet \ar[drdr, dotted, dash] &&\bullet \ar[ll] &\\
&&&\\
\bullet \ar[uu] &&\bullet \ar[ll] \ar[uu] &\bullet \ar[l] \\
\end{tikzcd}
\tab\item 
\begin{tikzcd}  [column sep=tiny, row sep=tiny]  
&\bullet &\bullet \ar[d] \ar[l] \ar[r] &\bullet \\
\bullet \ar[drdr, dotted, dash] &&\bullet \ar[ll] &\\
&&&\\
\bullet \ar[uu] &&\bullet \ar[ll] \ar[uu] &\bullet \ar[l] \\
\end{tikzcd}
\tab\item 
\begin{tikzcd}  [column sep=tiny, row sep=tiny]  
\bullet \ar[r] &\bullet &\bullet \ar[d] \ar[l] &\\
\bullet \ar[drdr, dotted, dash] &&\bullet \ar[ll] &\\
&&&\\
\bullet \ar[uu] &&\bullet \ar[ll] \ar[uu] &\bullet \ar[l] \\
\end{tikzcd}
\tab\item 
\begin{tikzcd}  [column sep=tiny, row sep=tiny]  
&&\bullet \ar[d] &\\
\bullet \ar[drdr, dotted, dash] &&\bullet \ar[ll] &\bullet \\
&&&\bullet \ar[u] \\
\bullet \ar[uu] &&\bullet \ar[ll] \ar[uu] &\bullet \ar[l] \ar[u] \\
\end{tikzcd}
\tab\item 
\begin{tikzcd}  [column sep=tiny, row sep=tiny]  
&&\bullet \ar[d] &&\\
\bullet \ar[drdr, dotted, dash] &&\bullet \ar[ll] &&\\
&&&\bullet &\\
\bullet \ar[uu] &&\bullet \ar[ll] \ar[uu] &\bullet \ar[l] \ar[r] \ar[u] &\bullet \\
\end{tikzcd}
\tab\item 
\begin{tikzcd}  [column sep=tiny, row sep=tiny]  
&&\bullet \ar[d] &\\
\bullet \ar[drdr, dotted, dash] &&\bullet \ar[ll] &\bullet \ar[d] \\
&&&\bullet \\
\bullet \ar[uu] &&\bullet \ar[ll] \ar[uu] &\bullet \ar[l] \ar[u] \\
\end{tikzcd}
\tab\item 
\begin{tikzcd} [column sep=tiny, row sep=tiny] 
\bullet \ar[d] &&&\bullet \ar[d] &\\
\bullet \ar[r] &\bullet \ar[drdr, dotted, dash] &&\bullet \ar[ll] &\\
&&&&\\
&\bullet \ar[uu] &&\bullet \ar[ll] \ar[uu] &\bullet \ar[l] \\
\end{tikzcd}
\tab\item 
\begin{tikzcd} [column sep=tiny, row sep=tiny] 
\bullet 		&							& \bullet \ar[dr] 	&\\
\bullet \ar[u]	& \bullet \ar[dr] \ar[ur] \ar[l]	&				& \bullet \ar[ll, dotted, dash]\\
			& 							& \bullet \ar[dr] 	&\\
			& \bullet  					&				& \bullet \ar[ll] \ar[uu]\\
\end{tikzcd}
\tab\item 
\begin{tikzcd} [column sep=tiny, row sep=tiny] 
\bullet \ar[dr] &&\bullet \ar[dr] &&\\
&\bullet \ar[drdr, dotted, dash] &&\bullet \ar[ll] &\\
\bullet \ar[ur]	&&&&\\
&\bullet \ar[uu] &&\bullet \ar[ll] \ar[uu] &\bullet \ar[l] \\
\end{tikzcd}
\tab\item 
\begin{tikzcd} [column sep=tiny, row sep=tiny] 
\bullet 	&									& \bullet \ar[dr] 	&\\
		& \bullet \ar[dr] \ar[ur] \ar[dl] \ar[ul]	&				& \bullet \ar[ll, dotted, dash]\\
\bullet	& 									& \bullet \ar[dr] 	&\\
		& \bullet  							&				& \bullet \ar[ll] \ar[uu]\\
\end{tikzcd}
\tab\item 
\begin{tikzcd} [column sep=tiny, row sep=tiny] 
\bullet &&\bullet \ar[dr] &&\\
\bullet \ar[u] \ar[r] &\bullet \ar[drdr, dotted, dash] &&\bullet \ar[ll] &\\
&&&&\\
&\bullet \ar[uu] &&\bullet \ar[ll] \ar[uu] &\bullet \ar[l]\\
\end{tikzcd}
\tab\item  
\begin{tikzcd} [column sep=tiny, row sep=tiny] 
\bullet \ar[d] &							& \bullet \ar[dr] 	&\\
\bullet		& \bullet \ar[dr] \ar[ur] \ar[l] 	&				& \bullet \ar[ll, dotted, dash]\\
			& 							& \bullet \ar[dr] 	&\\
			& \bullet  					&				& \bullet \ar[ll] \ar[uu]\\
\end{tikzcd}
\tab\item 
\begin{tikzcd} [column sep=tiny, row sep=tiny] 
&&&\bullet \ar[d] &\\
\bullet \ar[r] &\bullet \ar[drdr, dotted, dash] &&\bullet \ar[ll] &\\
&&&&\bullet \\
&\bullet \ar[uu] &&\bullet \ar[ll] \ar[uu] &\bullet \ar[l] \ar[u] \\
\end{tikzcd}
\tab\item 
\begin{tikzcd} [column sep=tiny, row sep=tiny] 
		&							& \bullet \ar[dr] 	&\\
\bullet 	& \bullet \ar[l]  \ar[dr] \ar[ur] 	&				& \bullet \ar[ll, dotted, dash]\\
		&							& \bullet \ar[dr] 	&\\
		& \bullet 					&				& \bullet \ar[ll] \ar[uu] \\
		& \bullet \ar[u] 				&				&
\end{tikzcd}
\tab\item  
\begin{tikzcd} [column sep=tiny, row sep=tiny]  
					&&								& \bullet	& \bullet \ar[l]\\
\bullet \ar[rr] \ar[dd]	&& \bullet \ar[dd, dotted, dash] 	&		& \bullet \ar[dd] \ar[u]\\
&&&&\\
\bullet   				&& \bullet \ar[uull]  \ar[rr] 		&		& \bullet \ar[uull] \\
\end{tikzcd}
\tab\item 
\begin{tikzcd} [column sep=tiny, row sep=tiny]  
					&&								& \bullet	& \bullet \\
\bullet \ar[rr] \ar[dd]	&& \bullet \ar[dd, dotted, dash] 	&		& \bullet \ar[dd] \ar[ul] \ar[u]\\
&&&&\\
\bullet   				&& \bullet \ar[uull]  \ar[rr] 		&		& \bullet \ar[uull] \\
\end{tikzcd}
\tab\item 
\begin{tikzcd} [column sep=tiny, row sep=tiny]  
					&&								& \bullet \ar[r]	& \bullet \\
\bullet \ar[rr] \ar[dd]	&& \bullet \ar[dd, dotted, dash] 	&				& \bullet \ar[dd] \ar[u]\\
&&&&\\
\bullet   				&& \bullet \ar[uull]  \ar[rr] 		&				& \bullet \ar[uull] \\
\end{tikzcd}
\tab\item  
\begin{tikzcd} [column sep=tiny, row sep=tiny]
\bullet \ar[rr] 				&			& \bullet \ar[ddll, dotted, dash] \ar[ddrr, dotted, dash] 	&& \bullet \ar[ll]\\
&&&&\\
\bullet \ar[uu] \ar[rr] \ar[d]	&			& \bullet \ar[uu] 										&& \bullet \ar[ll] \ar[uu] \\
\bullet \ar[r]					& \bullet		&													&&
\end{tikzcd}
\tab\item 
\begin{tikzcd} [column sep=tiny, row sep=tiny]
\bullet \ar[rr] 					&			& \bullet \ar[ddll, dotted, dash] \ar[ddrr, dotted, dash] 	&& \bullet \ar[ll]\\
&&&&\\
\bullet \ar[uu] \ar[rr] \ar[d] \ar[dr]	&			& \bullet \ar[uu] 										&& \bullet \ar[ll] \ar[uu] \\
\bullet							& \bullet		&													&&
\end{tikzcd}
\tab\item 
\begin{tikzcd} [column sep=tiny, row sep=tiny]
\bullet \ar[rr] 				&				& \bullet \ar[ddll, dotted, dash] \ar[ddrr, dotted, dash] 	&& \bullet \ar[ll]\\
&&&&\\
\bullet \ar[uu] \ar[rr] \ar[d]	&				& \bullet \ar[uu] 										&& \bullet \ar[ll] \ar[uu] \\
\bullet						& \bullet	\ar[l]	&													&&
\end{tikzcd}
\tab\item 
\begin{tikzcd} [column sep=tiny, row sep=tiny]
\bullet \ar[d]			& \bullet \ar[l]	&							&& \\
\bullet	 			&				& \bullet \ar[ddll] \ar[ddrr] 	&& \bullet \\
&&&&\\
\bullet \ar[uu] \ar[rr] 	&				& \bullet \ar[uu, dotted, dash] 	&& \bullet \ar[ll] \ar[uu] \\
\end{tikzcd}
\tab\item 
\begin{tikzcd} [column sep=tiny, row sep=tiny]
\bullet \ar[d]			& \bullet \ar[dl]	&							&& \\
\bullet	 			&				& \bullet \ar[ddll] \ar[ddrr] 	&& \bullet \\
&&&&\\
\bullet \ar[uu] \ar[rr] 	&				& \bullet \ar[uu, dotted, dash] 	&& \bullet \ar[ll] \ar[uu] \\
\end{tikzcd}
\tab\item 
\begin{tikzcd} [column sep=tiny, row sep=tiny]
\bullet \ar[d] \ar[r]	& \bullet 	&							&& \\
\bullet	 			&			& \bullet \ar[ddll] \ar[ddrr] 	&& \bullet \\
&&&&\\
\bullet \ar[uu] \ar[rr] 	&			& \bullet \ar[uu, dotted, dash] 	&& \bullet \ar[ll] \ar[uu] \\
\end{tikzcd}
\tab\item 
\begin{tikzcd} [column sep=tiny, row sep=tiny] 
\bullet \ar[dd] &&\bullet \ar[ddll] &&\bullet \ar[d] \\
&&&&\bullet \ar[d] \\
\bullet \ar[rrrr, dotted, dash] &&&&\bullet \ar[d] \ar[uull]\\
\bullet \ar[u] &&&&\bullet \ar[llll]\\
\end{tikzcd}
\tab\item 
\begin{tikzcd} [column sep=tiny, row sep=tiny] 
\bullet \ar[d] &&\bullet \ar[ddll] &&\\
\bullet &&&&\bullet \ar[d] \\
\bullet \ar[u] \ar[rrrr, dotted, dash] &&&&\bullet \ar[dd]  \ar[uull]\\
&&&&\\
\bullet \ar[uu] &&&&\bullet \ar[llll] \\
\end{tikzcd}
\tab\item 
\begin{tikzcd} [column sep=tiny, row sep=tiny] 
&&\bullet \ar[ddll] &\bullet &\bullet \ar[d] \ar[l] \\
&&&&\bullet \ar[d] \\
\bullet \ar[rrrr, dotted, dash] &&&&\bullet \ar[d] \ar[uull]\\
\bullet \ar[u] &&&&\bullet \ar[llll]\\
\end{tikzcd}
\tab\item  
\begin{tikzcd} [column sep=tiny, row sep=tiny]  
\bullet 				&							& \bullet \ar[dl] 	&\\
\bullet \ar[dd] \ar[u] 	& \bullet \ar[rr, dotted, dash] 	&				& \bullet \ar[dd] \ar[ul]\\
					&							&				&\\
\bullet \ar[r]			& \bullet \ar[uu] 				&				& \bullet \ar[ll] \\
\end{tikzcd}
\tab\item 
\begin{tikzcd} [column sep=tiny, row sep=tiny]  
				&							& \bullet \ar[dl] 		&							&\\
\bullet \ar[dd] 	& \bullet \ar[rr, dotted, dash] 	&					& \bullet \ar[dd] \ar[ul] \ar[r] 	& \bullet \\
				&							&					&							&\\
\bullet \ar[r]		& \bullet \ar[uu] 				&					& \bullet \ar[ll] 				&\\
\end{tikzcd}
\tab\item 
\begin{tikzcd} [column sep=tiny, row sep=tiny]  
				&							& \bullet \ar[d] 	&\\
				&							& \bullet			&\\
\bullet \ar[dd]	& \bullet \ar[l] \ar[rr] 			&				& \bullet \ar[dd] \ar[ul]\\
				&							&				&\\
\bullet \ar[r]		& \bullet \ar[uu, dotted, dash] 	&				& \bullet \ar[ll] \\
\end{tikzcd}
\tab\item 
\begin{tikzcd} [column sep=tiny, row sep=tiny]  
\bullet \ar[rr]	 	&							& \bullet \ar[dl, dotted, dash] 	&				& \bullet \ar[dl] \\
				& \bullet \ar[ul] \ar[d]	\ar[dl] 	&							& \bullet \ar[ul] 	& \\
\bullet 			& \bullet \ar[rr] 				&							& \bullet \ar[u] 	&\\
\end{tikzcd}
\tab\item 
\begin{tikzcd} [column sep=tiny, row sep=tiny]  
\bullet \ar[d]		&&						&&\\
\bullet 			&& \bullet \ar[d] \ar[rr] 	&& \bullet \ar[dd] \\
				&& \bullet \ar[d] 			&&\\
\bullet \ar[uu]	&& \bullet \ar[ll] \ar[rr] 	&& \bullet \ar[uull, dotted, dash]\\
\end{tikzcd}
\tab\item 
\begin{tikzcd}  [column sep=tiny, row sep=tiny] 
&&\bullet \ar[d] &&\\
\bullet \ar[rr] &&\bullet \ar[dd, dotted, dash] &&\bullet \ar[dd] \\
&&&&\\
\bullet \ar[uu]	&&\bullet \ar[ll] \ar[rr] &&\bullet \ar[uull]\\
&&\bullet \ar[u] &&
\end{tikzcd}
\tab\item  
\begin{tikzcd}  [column sep=tiny, row sep=tiny]
				&& \bullet 					&&\\
\bullet 			&& \bullet \ar[u] \ar[d] \ar[rr] 	&& \bullet \ar[dd] \\
				&& \bullet \ar[d] 				&&\\
\bullet \ar[uu]	&& \bullet \ar[ll] \ar[rr] 		&& \bullet \ar[uull, dotted, dash]\\
\end{tikzcd}
\tab\item 
\begin{tikzcd}  [column sep=tiny, row sep=tiny]
&&&&\bullet \\
\bullet \ar[rr] &&\bullet \ar[dd, dotted, dash] &&\bullet \ar[u] \ar[dd] \\
&&&&\\
\bullet \ar[uu]	&&\bullet \ar[ll] \ar[rr] &&\bullet \ar[uull]\\
&&\bullet \ar[u] &&
\end{tikzcd}
\tab\item 
\begin{tikzcd} [column sep=tiny, row sep=tiny]  
\bullet \ar[d]			&					&  							&& \\
\bullet  				&					& \bullet \ar[dl] \ar[rr] 		&& \bullet \ar[dd] \\
					& \bullet  \ar[dl] 		&							&&\\
\bullet \ar[uu] \ar[rr] 	&					& \bullet \ar[uu, dotted, dash] 	&& \bullet \ar[ll]\\
\end{tikzcd}
\tab\item  
\begin{tikzcd} [column sep=tiny, row sep=tiny]  
&&&&&\bullet  \\
&\bullet \ar[rr] &&\bullet \ar[dldl, dotted, dash] &&\bullet \ar[dd] \ar[u]\\
&&&&&\\
\bullet \ar[r] &\bullet \ar[uu] \ar[rr] &&\bullet \ar[uu] &&\bullet \ar[ll]\\
\end{tikzcd}
\tab\item 
\begin{tikzcd} [column sep=tiny, row sep=tiny]  
&&\bullet \ar[dr] &&\\
\bullet \ar[r] &\bullet \ar[ur] \ar[dd] &&\bullet \ar[ll, dotted, dash] &\\
&&&&\bullet \\
&\bullet \ar[rr] &&\bullet \ar[uu] &\bullet \ar[l] \ar[u] \\
\end{tikzcd}
\tab\item 
\begin{tikzcd} [column sep=tiny, row sep=tiny]  
\bullet &&\bullet \ar[dr] &&\\
\bullet \ar[u] \ar[r] &\bullet \ar[ur] \ar[dd] &&\bullet \ar[ll, dotted, dash] &\\
&&&&\\
&\bullet \ar[rr] &&\bullet \ar[uu] &\bullet \ar[l] \\
\end{tikzcd}
\tab\item  
\begin{tikzcd} [column sep=tiny, row sep=tiny]  
&\bullet \ar[rr] &&\bullet \ar[dldl, dotted, dash] &&\bullet \ar[ll] \\
\bullet \ar[d] &&&&&\\
\bullet \ar[r] &\bullet \ar[uu] \ar[rr] &&\bullet \ar[uu] &&\bullet \ar[ll]\\
\end{tikzcd}
\tab\item 
\begin{tikzcd} [column sep=tiny, row sep=tiny]  
						&				&				& \bullet 					&\\
\bullet 		 			&				&				& \bullet \ar[dl] \ar[r] \ar[u] 	& \bullet \ar[dddl] \\
						&				& \bullet \ar[dl] 	&							&\\
						& \bullet  \ar[dl] 	&				&							&\\
\bullet \ar[rrr] \ar[uuu] 	&				&				& \bullet \ar[uuu, dotted, dash]	&\\
\end{tikzcd}
\tab\item 
\begin{tikzcd} [column sep=tiny, row sep=tiny]  
\bullet &\bullet \ar[rr] &&\bullet \ar[dldl, dotted, dash] && \\
\bullet \ar[d] \ar[u] &&&&&\\
\bullet \ar[r] &\bullet \ar[uu] \ar[rr] &&\bullet \ar[uu] &&\bullet \ar[ll]\\
\end{tikzcd}
\tab\item 
\begin{tikzcd} [column sep=tiny, row sep=tiny] 
					&			&							&& \bullet \\
\bullet \ar[rr] 		&			& \bullet \ar[dd, dotted, dash] 	&& \bullet \ar[dd] \ar[u]\\
					&			&							&&\\
\bullet \ar[uu] \ar[dr] 	&			& \bullet \ar[ll]  \ar[rr] 		&& \bullet \ar[uull] \\
					& \bullet 	&							&&\\
\end{tikzcd}
\tab\item 
\begin{tikzcd} [column sep=tiny, row sep=tiny]  
\bullet				&& \bullet \ar[rr] \ar[ll] 		&& \bullet \ar[d]\\
\bullet \ar[rr] \ar[dd]	&& \bullet \ar[dd, dotted, dash]&& \bullet \ar[ll] \\
					&&							&&\\
\bullet 				&& \bullet \ar[uull]  \ar[uurr] 	&&  \\
\end{tikzcd}
\tab\item 
\begin{tikzcd} [column sep=tiny, row sep=tiny]
				&				&							&& \bullet \ar[d] \\
\bullet \ar[dr] 	&				& \bullet \ar[ll] \ar[ddrr] 		&& \bullet \\
				& \bullet \ar[dr] 	&							&& \\
\bullet \ar[uu] 	&				& \bullet \ar[uu, dotted, dash] 	&& \bullet \ar[uu] \ar[ll]	 	
\end{tikzcd}
\tab\item 
\begin{tikzcd} [column sep=tiny, row sep=tiny]  
&&\bullet \ar[ddl] &\bullet \ar[l] &\\
&&&&\\
\bullet &\bullet \ar[l] \ar[rr, dotted, dash] &&\bullet \ar[dd] \ar[uul] \ar[r] &\bullet \\
&&&&\\
&\bullet \ar[uu] &&\bullet \ar[ll] &
\end{tikzcd}
\tab\item 
\begin{tikzcd} [column sep=tiny, row sep=tiny]  
&&\bullet \ar[r] \ar[ddr] &\bullet &\bullet \ar[l]\\
&&&&\\
\bullet \ar[r] &\bullet \ar[dd] \ar[uur] &&\bullet \ar[ll, dotted, dash] &\\
&&&&\\
&\bullet \ar[rr] &&\bullet \ar[uu] &\\
\end{tikzcd}
\tab\item 
\begin{tikzcd} [column sep=tiny, row sep=tiny] 
&\bullet \ar[dd] &&\bullet \ar[ddll, dotted, dash] \ar[dd] \ar[ll] &\bullet \\
&&&&\bullet \ar[u] \ar[d] \\
\bullet &\bullet \ar[l] &&\bullet \ar[ll] &\bullet \ar[l] \\
\end{tikzcd}
\tab\item  
\begin{tikzcd} [column sep=tiny, row sep=tiny]
\bullet \ar[d] &&&&\\
\bullet \ar[drdr, dotted, dash] && \bullet \ar[ll] \ar[drdr, dotted, dash] &&\bullet \ar[ll] \\
&&&&\\
\bullet \ar[uu] &&\bullet \ar[uu] \ar[ll] &\bullet \ar[l] &\bullet \ar[l] \ar[uu]
\end{tikzcd}
\tab\item 
\begin{tikzcd} [column sep=tiny, row sep=tiny]
& \bullet \ar[drdr, dotted, dash] &&\bullet \ar[ll] &\\
&&&&\bullet \\
\bullet &\bullet \ar[l] \ar[uu] &\bullet \ar[l] &\bullet \ar[l] \ar[uu] &\bullet \ar[l] \ar[u]
\end{tikzcd}
\tab\item 
\begin{tikzcd} [column sep=tiny, row sep=tiny]
& \bullet \ar[drdr, dotted, dash] &&\bullet \ar[ll] &\\
\bullet \ar[d] &&&&\\
\bullet &\bullet \ar[l] \ar[uu] &\bullet \ar[l] &\bullet \ar[l] \ar[uu] &\bullet \ar[l]
\end{tikzcd}
\tab\item 
\begin{tikzcd} [column sep=tiny, row sep=tiny]
\bullet \ar[drdr, dotted, dash] && \bullet \ar[ll] \ar[drdr, dotted, dash] &&\bullet \ar[ll] \\
&&&&\\
\bullet \ar[uu] &&\bullet \ar[uu] \ar[ll] &\bullet \ar[l] &\bullet \ar[l] \ar[uu] \ar[d] \\
&&&&\bullet \\
\end{tikzcd}
\tab\item 
\begin{tikzcd} [column sep=tiny, row sep=tiny] 
\bullet &&\bullet \ar[dr] &&\\
&\bullet \ar[ul] \ar[drdr, dotted, dash] &&\bullet \ar[ll] &\\
\bullet \ar[uu]	&&&&\\
&\bullet \ar[uu] &&\bullet \ar[ll] \ar[uu] &\bullet \ar[l] \\
\end{tikzcd}
\tab\item 
\begin{tikzcd} [column sep=tiny, row sep=tiny] 
\bullet &&\bullet \ar[dr] &&\\
&\bullet \ar[ul] \ar[drdr, dotted, dash] &&\bullet \ar[ll] &\\
&&&&\bullet \\
&\bullet \ar[uu] &&\bullet \ar[ll] \ar[uu] &\bullet \ar[l] \ar[u]	\\
\end{tikzcd}
\tab\item  
\begin{tikzcd} [column sep=tiny, row sep=tiny] 
\bullet &\bullet \ar[rr] \ar[dd] \ar[l] &&\bullet \ar[r] \ar[dd] &\bullet \ar[dd] &\\
&&&&&\\
&\bullet \ar[rr] &&\bullet \ar[r] \ar[ulul, dotted, dash] &\bullet \ar[ulu, dotted, dash] \ar[r] &\bullet \\
\end{tikzcd}
\tab\item 
\begin{tikzcd} [column sep=tiny, row sep=tiny] 
\bullet \ar[d] &\bullet \ar[dd] &&\bullet \ar[ddll, dotted, dash] \ar[dd] \ar[ll] & \\
\bullet &&&&\\
\bullet \ar[u] &\bullet \ar[l] &&\bullet \ar[ll] &\bullet \ar[l] \\
\end{tikzcd}
\tab\item 
\begin{tikzcd} [column sep=tiny, row sep=tiny] 
&\bullet \ar[dd] &&\bullet \ar[ddll, dotted, dash] \ar[dd] \ar[r] \ar[ll] &\bullet \\
\bullet &&&&\\
\bullet \ar[u] &\bullet \ar[l] &&\bullet \ar[ll] &\bullet \ar[l] \\
\end{tikzcd}
\tab\item  
\begin{tikzcd} [column sep=tiny, row sep=tiny] 
&\bullet \ar[rr] \ar[dd] &&\bullet \ar[r] \ar[dd] &\bullet \ar[dd] &\\
&&&&&\bullet \ar[d] \\
&\bullet \ar[rr] &&\bullet \ar[r] \ar[ulul, dotted, dash] &\bullet \ar[ulu, dotted, dash] \ar[r] &\bullet \\
\end{tikzcd}
\tab\item 
\begin{tikzcd} [column sep=tiny, row sep=tiny]  
					& \bullet \ar[r] 	& \bullet 					&& \bullet \ar[ll] \\
					&				&							&&\\
\bullet 				&				& \bullet \ar[ddll] \ar[rr] 		&& \bullet \ar[ddll] \ar[uu] \\
					&				&							&& \\
\bullet \ar[uu] \ar[rr] 	&				& \bullet \ar[uu, dotted, dash] 	&& \\
\end{tikzcd}
\tab\item 
\begin{tikzcd} [column sep=tiny, row sep=tiny]
							&& \bullet \ar[dd] 									&& \bullet \ar[ll]\\
							&&													&&\\
\bullet \ar[rr] 				&& \bullet \ar[ddll, dotted, dash] \ar[rr, dotted, dash] 	&& \bullet \ar[ddll] \ar[uu] \\
							&&													&& \\
\bullet \ar[uu] \ar[d] \ar[rr] 	&& \bullet \ar[uu] 									&& \\
\bullet   						&&													&& \\ 
\end{tikzcd}
\tab\item  
\begin{tikzcd} [column sep=tiny, row sep=tiny]
					&&								&& \bullet \\
\bullet \ar[d]  		&&								&&\\
\bullet 				&& \bullet \ar[ddll] \ar[rr] 			&& \bullet \ar[dd] \ar[uu] \\
					&&								&& \\
\bullet \ar[uu] \ar[rr] 	&& \bullet \ar[uu, dotted, dash] 	&& \bullet \ar[ll] \\
\end{tikzcd}
\tab\item 
\begin{tikzcd} [column sep=tiny, row sep=tiny]
							&&												&& \bullet \ar[ddll]\\
							&&												&&\\
\bullet \ar[rr]  				&& \bullet \ar[ddll, dotted, dash] \ar[rr, dotted, dash] && \bullet \ar[dd] \ar[uu] \\
							&&												&& \\
\bullet \ar[uu] \ar[d] \ar[rr] 	&& \bullet \ar[uu] 								&& \bullet \ar[ll] \\
\bullet 						&&												&&\\
\end{tikzcd}
\tab\item  
\begin{tikzcd} [column sep=tiny, row sep=tiny]
					&&												& \bullet \ar[ddl] 				&\\
					&&												&							&\\
\bullet \ar[rr]  		&& \bullet \ar[ddll, dotted, dash] \ar[r, dotted, dash] 	& \bullet \ar[dd] \ar[uu] \ar[r] 	& \bullet \\
					&&												&							& \\
\bullet \ar[uu] \ar[rr] 	&& \bullet \ar[uu] 								& \bullet \ar[l] 				& \\
\end{tikzcd}
\tab\item  
\begin{tikzcd} [column sep=tiny, row sep=tiny]  
					&				&							&& \bullet \ar[d]\\
\bullet 				&				& \bullet \ar[dl] \ar[ddrr] 		&& \bullet \\
					& \bullet  \ar[dl] 	&							&&\\
\bullet \ar[uu] \ar[rr] 	&				& \bullet \ar[uu, dotted, dash] 	&& \bullet \ar[ll] \ar[uu]\\
\end{tikzcd}
\tab\item 
\begin{tikzcd} [column sep=tiny, row sep=tiny]
\bullet \ar[rr] &&\bullet \ar[ddll, dotted, dash] \ar[ddrr, dotted, dash] &&\bullet \ar[ll]\\
&&&&\\
\bullet \ar[uu] \ar[rr] &&\bullet \ar[uu] && \bullet \ar[ll] \ar[uu] \\
\bullet \ar[u] \ar[r] &\bullet &&&
\end{tikzcd}
\tab\item 
\begin{tikzcd} [column sep=tiny, row sep=tiny]  
					&& 												& \bullet \ar[ddl] 				&\\
					&&												&							&\\
\bullet \ar[dd] \ar[rr] 	&& \bullet \ar[dd, dotted, dash] \ar[r, dotted, dash] 	& \bullet \ar[dd]  \ar[r] \ar[uu] & \bullet \\
					&&												&							&\\
\bullet 				&& \bullet \ar[uull] \ar[r] 							& \bullet \ar[uul] 				&\\
\end{tikzcd}
\tab\item 
\begin{tikzcd} [column sep=tiny, row sep=tiny] 
					&& 								& \bullet \ar[dd] \\
					&&								&\\
\bullet \ar[ddrr] 		&& \bullet \ar[ll] \ar[uur] \ar[ddr] 	& \bullet \ar[l, dotted, dash] \\
					&&								&\\
\bullet \ar[uu] \ar[d] 	&& \bullet \ar[uu, dotted, dash]	& \bullet \ar[l] \ar[uu]\\
\bullet  				&&								&\\
\end{tikzcd}
\tab\item 
\begin{tikzcd} [column sep=tiny, row sep=tiny] 
&\bullet \ar[rr] \ar[dd] &&\bullet \ar[r] \ar[dd] &\bullet \ar[dd] &\\
\bullet \ar[d] &&&&& \\
\bullet &\bullet \ar[l] \ar[rr] &&\bullet \ar[r] \ar[ulul, dotted, dash] &\bullet \ar[ulu, dotted, dash] & \\
\end{tikzcd}
\tab\item  
\begin{tikzcd} [column sep=tiny, row sep=tiny] 
\bullet \ar[r] &\bullet &&\\
&\bullet \ar[u] &\bullet \ar[l] &\\
\bullet \ar[uru, dotted, dash] \ar[uu] \ar[r] &\bullet \ar[u] &\bullet \ar[l] \ar[u] \ar[r] \ar[ul, dotted, dash] &\bullet 
\end{tikzcd}
\tab\item  
\begin{tikzcd} [column sep=tiny, row sep=tiny] 
&\bullet \ar[rr] \ar[dd] &&\bullet \ar[r] \ar[dd] &\bullet \ar[dd] \ar[r] &\bullet \\
&&&&& \bullet \ar[u] \\
&\bullet \ar[rr] &&\bullet \ar[r] \ar[ulul, dotted, dash] &\bullet \ar[ulu, dotted, dash] & \\
\end{tikzcd}
\tab\item 
\begin{tikzcd} [column sep=tiny, row sep=tiny]  
							& \bullet \ar[d] 	&							&&\\
							& \bullet 		&							&&\\
\bullet \ar[dd] \ar[ur] 			&				& \bullet \ar[ll] \ar[dd] \ar[rr] 	&& \bullet \ar[dd]  \\
							&				&							&& \\
\bullet \ar[uurr, dotted, dash] 	&				& \bullet \ar[ll] \ar[rr] 		&& \bullet \ar[uull, dotted, dash]\\
\end{tikzcd}
\tab\item  
\begin{tikzcd} [column sep=tiny, row sep=tiny]
			&							& \bullet \ar[dl] 					&&\\
\bullet \ar[r] 	& \bullet \ar[r, dotted, dash] 	& \bullet \ar[u] \ar[ddl] \ar[ddrr] 	&& \bullet \\
			&							&								&&\\
			& \bullet \ar[uu] \ar[r] 		& \bullet \ar[uu, dotted, dash] 		&& \bullet \ar[ll] \ar[uu] \\
\end{tikzcd}
\tab\item 
\begin{tikzcd} [column sep=tiny, row sep=tiny] 
\bullet &\bullet \ar[l] \ar[rr] \ar[dd] &&\bullet \ar[r] \ar[dd] &\bullet \ar[dd] \ar[r] &\bullet \\
&&&&& \\
&\bullet \ar[rr] &&\bullet \ar[r] \ar[ulul, dotted, dash] &\bullet \ar[ulu, dotted, dash] & \\
\end{tikzcd}
\tab\item 
\begin{tikzcd} [column sep=tiny, row sep=tiny] 
&\bullet \ar[rr] \ar[dd] &&\bullet \ar[r] \ar[dd] &\bullet \ar[dd] & \\
&&&&& \\
\bullet \ar[r] &\bullet \ar[rr] &&\bullet \ar[r] \ar[ulul, dotted, dash] &\bullet \ar[ulu, dotted, dash] &\bullet \ar[l]  \\
\end{tikzcd}
\end{inparaenum}
\end{theorem}
\begin{proof} \label{provae8}
The scheme of the demonstration consists  in analyze each representation-finite trivial extension of Cartan class $\mathbb{E}_8$ \cite{fern}, displaying its relations of type $2$ and its elementary cycles. With this information, we apply in our computer program and we have the cutting sets that define the incidence algebras. Thus, we show in the list of our theorem only the non-hereditary incidence algebras. We will show this procedure to a trivial extension and the remaining work is analogous.

The following is a particular example:

\begin{tikzcd} [column sep=small, row sep=small]  
&\bullet \ar{ddl}[swap]{\alpha_4} &\bullet \ar{l}[swap]{\alpha_9} &\\
&&&\\
\bullet \ar{d}[swap]{\alpha_2} \ar{r}{\alpha_8} &\bullet \ar{r}{\alpha_1} \ar{uur}[near start]{\alpha_{10}} &\bullet \ar{r}{\alpha_5} &\bullet \ar{d}{\alpha_{11}} \ar{uull}[swap]{\alpha_7}\\
\bullet \ar{rrr}[swap]{\alpha_6} &&&\bullet \ar{ulll}{\alpha_3}\\
\end{tikzcd}\hfill%
\begin{minipage}{.45\linewidth}
The relations of type $2$ are: $r1=\alpha_3 \alpha_8 \alpha_1 \alpha_5 \alpha_7$, $r2= \alpha_4 \alpha_8 \alpha_1 \alpha_5 \alpha_{11}$, $r3= \alpha_9 \alpha_4 \alpha_8 \alpha_1$, $r4= \alpha_7 \alpha_4 \alpha_8 \alpha_{10}$, $r5= \alpha_3 \alpha_8 \alpha_{10}$, $r6= \alpha_{11} \alpha_3 \alpha_2$, $r7= \alpha_6 \alpha_3 \alpha_8$ and $r8= \alpha_4 \alpha_2$.

The elementary cycles are: $C_1 = \alpha_8 \alpha_1 \alpha_5 \alpha_7 \alpha_4$, $C_2 = \alpha_8 \alpha_1 \alpha_5 \alpha_{11} \alpha_3$, $C_3 = \alpha_8 \alpha_{10} \alpha_9 \alpha_4$ and $C_4 = \alpha_3 \alpha_2 \alpha_6$. 

Therefore, the program shows us the cutting set that defines the solutions $76$ and $77$ of the theorem. 
\end{minipage}
\vspace{.5cm}
\end{proof}

%% file: tame.tex
The PHI algebras of extended Dynkin type are derived equivalent to the category of coherent sheaves of certain weighted projective line. 
This category is derived equivalent to the category of modules over canonical algebras, see \cite{hub}, \cite{gei-len}, \cite{rin}. 

We do not have a complete description of the PHI algebras of extended Dynkin type as in the Dynkin type, but we have been able to identify some families.
One of them is the PHI concealed algebras of extended Dynkin type.
We can identify through the works of Happel-Vossieck \cite{hap-vos} and Bongartz \cite{bon}.

The other family of PHI algebras of extended Dynkin type are the members of a new set which we call the \emph{ANS family}, in reference to Assem, Nehring and Skowro\'nski.

The algebras which belong to this family are the algebras which are obtained by a cutting of the trivial extensions of concealed algebras of type $\overline{Q}$ where  $\overline{Q}=\widetilde{\mathbb{D}_n}$ for some $n$, or $\overline{Q}=\widetilde{\mathbb{E}_p}$ for some $p$.

In the Tsukuba Journal of Mathematics \cite{ass-neh-sko}, Assem, Nehring and Skowro\'nski published the work entitled ``Domestic trivial extensions of simply connected algebras''. The main result of this article has an important role in the characterization of the ANS family.

\begin{theorem} [Assem-Nehring-Skowro\'nski] \label{ANS}
Let $A$ be a finite-dimensional, basic and connected algebra over an algebraically closed field $K$. If $A$ is simply connected, then the following conditions are equivalent:
\begin{enumerate}
\item $A$ is an iterated tilted algebra of Euclidean type $\overline{Q}$, where  $\overline{Q}=\widetilde{\mathbb{D}_n}$ for some $n$, or $\overline{Q}=\widetilde{\mathbb{E}_p}$ for some $p$.
\item There exists a representation-infinite tilted algebra $B$ of the  type $\overline{Q}$ such that $T(A) \cong T(B)$.
\end{enumerate}
\end{theorem}

\begin{proposition}
The theorem above is valid without the hypothesis of $A$ simply connected.
\end{proposition}

\begin{proof}
 To show that item 1 implies item 2 we observe that since $A$ is iterated tilted of type $Q$ it is automatically simply connected  \cite{mar-mor} 
 and we get the validity of item 2.
 
 The implication of item 2 to item 1 is done as follows:
 
 The algebra $B$ is tilted of type $Q$ so it is simply connected, change in the item 1 the role of $A$ by $B$. Since $B$ is simply connected the first two items are equivalent, with the simply connected hypothesis, changing $A$ for
 $B$ so the validity of item 2 implies that item 1 is valid for $A$ and also that $A$ is simply connected.
\end{proof}

We will describe the  algebras $B$ which are the members of the ANS family.  We will work with the concealed algebras $A$ of Euclidean type $\widetilde{\mathbb{D}}_m$, $\widetilde{\mathbb{E}}_6$, $\widetilde{\mathbb{E}}_7$ ou $\widetilde{\mathbb{E}}_8$. Therefore, the difficulty is to obtain the PHI algebras $K\Delta$ such that $T(K\Delta) \cong T(A)$.

We will use the list of Happel-Vossieck \cite{hap-vos} to work with all the concealed algebras $A$ of Euclidean type, so the algebras $B$ come from the list of Happel-Vossieck. 

The algebras which we consider are the trivial extension of schurian algebras. Thus, from the list of Happel-Vossieck \cite{hap-vos} we will explore only the schurian algebras. In turn, we observed that the schurian concealed algebras $A$ of Euclidean type fit perfectly in the strategy used in the full description of PHI algebras of Dynkin type. The main difference is that we do not have a list of the trivial extensions of $A$. Therefore, the machinery to find PHI algebra $K\Delta$ will have one more step than that which was done in the previous sections. We will describe the quivers of the trivial extensions of $A$. An important observation is that we will not study the trivial extensions of $A^{op}$ since the trivial extension of $A^{op}$ is isomorphic to $T(A)^{op}$.

\begin{lemma}
Let $A$ be a finite dimensional algebra. We denote the trivial extension of $A$ by $T(A)$. We have $T(A)^{op} \cong T(A^{op})$.
\end{lemma}
\begin{proof}
We consider the morphism
\begin{equation*}
\begin{split}
\Phi \colon T(A)^{op} & \longrightarrow  T(A^{op}) \\
(a,f) & \longmapsto (a,f)
\end{split}
\end{equation*}

The reader can check that this is an isomorphism of algebras.
\end{proof}

Let $A$ schurian concealed algebra Euclidean type $\widetilde{\mathbb{D}}_m$, $\widetilde{\mathbb{E}}_6$, $\widetilde{\mathbb{E}}_7$ or $\widetilde{\mathbb{E}}_8$, we will do the following:
\begin{itemize}
\item compute the quiver of trivial extension $A$;
\item describe the elementary cycles of $Q_{T(A)}$;
\item describe the relations of type $2$ in presentation of $T(A)$;
\item identifies the arrows in these relations of type $2$;
\item use our computer program;
\item if there is a non-hereditary solution, then we need to verify if the cutting set defines a PHI algebra $K\Delta$ that is not a concealed algebra of the Euclidean type.
\end{itemize}

This will be the script of the proof of the characterization theorems of the PHI algebras of ANS family.

The chapter XIV of the book ``Elements of the Representation of Associative Algebras'' \cite{sim-sko2} gives a list of all frames of the list of Happel-Vossieck \cite{hap-vos}. We will use this list to write the series of results for the PHI algebras of type $\widetilde{\mathbb{D}}_m$, $\widetilde{\mathbb{E}}_6$, $\widetilde{\mathbb{E}}_7$ and $\widetilde{\mathbb{E}}_8$.

We will start the work of describing the PHI algebras of the ANS family of type $\widetilde{\mathbb{D}}_m$. First, we'll start with members of type $\widetilde{\mathbb{D}}_4$.

\begin{theorem} \label{d4til}
The algebras associated with the quivers with relations below:

\NumTabs{2}
\begin{inparaenum}
\noindent 
\item 
\begin{tikzcd} [column sep=small, row sep=small]
				&\bullet \ar[dl] \ar[dr] \ar[d] \ar[dd, dotted, dash, bend right=45] \ar[dd, dotted, dash, bend left=45]&\\
\bullet	\ar[dr]	&\bullet \ar[d]																							&\bullet \ar[dl]\\
				&\bullet 																								&				
\end{tikzcd}
\tab\item 
\begin{tikzcd} [column sep=small, row sep=small]
\bullet	\ar[drr] \ar[d]	\ar[ddr, dotted, dash]	&		&\bullet \ar[dll] \ar[d] \ar[ddl, dotted, dash]\\
\bullet	\ar[dr]									&		&\bullet \ar[dl]\\
												&\bullet&
\end{tikzcd}
\tab\item 
\begin{tikzcd} [column sep=small, row sep=small]
\bullet	\ar[dr] 		&													&\bullet \ar[ll] \ar[dd] \\
						&\bullet \ar[ur, dash, dotted] \ar[dl, dash, dotted]&\\
\bullet	\ar[uu] \ar[rr]	&													&\bullet \ar[ul]  			
\end{tikzcd}
\end{inparaenum}

are PHI algebras of the ANS family of type $\widetilde{\mathbb{D}}_4$.
\end{theorem}
\begin{proof} 
Given an algebra coming from an admissible operation on a frame, see \cite{sim-sko2}, we will use the trivial extension of that algebra and extract the necessary information to use our computer program. We will do this for each admissible operation in  each frame. Since we are only interested in type $\widetilde{\mathbb{D}}_4$, the only frame that satisfies this condition is $\mathcal{F}r2$.

Observe that we are not interested in the algebras whose trivial extensions only have cutting sets that define hereditary PHI algebra. This will be examined more carefully in frame $\mathcal{F}r2$. The adjacent quiver $\widetilde{\mathbb{D}}_4$ has $16$  possibilities of admissible operations $1$, that is, we have $16$ Euclidean graphs $\widetilde{\mathbb{D}}_4$. Now, our goal is to study the trivial extension of each graph and see if there is any cutting set that defines a non-hereditary incidence algebra. For this, the trivial extension must have at least two elementary cycles that have at least one common arrow in accordance with the lemma \ref{cortephia}. 
As an example we will describe here  this procedure to a particular  trivial extension the other cases follows the same pattern.

\begin{tikzcd} 
\bullet	\ar[dr]	&		&\bullet \\
		&\bullet \ar[ur]&\\
\bullet	\ar[ur]	&		&\bullet \ar[ul]					
\end{tikzcd} \hfill
\begin{tikzcd} [column sep=small, row sep=small]
\bullet	\ar{dr}[sloped, swap]{\alpha_3}	&&\bullet \ar{ll}[swap]{\alpha_5} \ar{dd}{\alpha_6} \ar[bend left]{ddll}[near end]{\alpha_7}\\
&\bullet \ar{ur}[sloped]{\alpha_1} &\\
\bullet	\ar{ur}[sloped]{\alpha_2} &&\bullet \ar{ul}[swap, sloped]{\alpha_4}
\end{tikzcd}

By symmetry, there are four quivers in this set. The relations of type $2$ are: $r1=\alpha_3 \alpha_1 \alpha_6$, $r2=\alpha_3 \alpha_1 \alpha_7$, $r3=\alpha_4 \alpha_1 \alpha_7$, $r4=\alpha_4 \alpha_1 \alpha_5$, $r5=\alpha_2 \alpha_1 \alpha_5$ and $r6=\alpha_2 \alpha_1 \alpha_6$. The elementary cycles are: $C_1 = \alpha_1 \alpha_5 \alpha_3$, $C_2 = \alpha_1 \alpha_6 \alpha_4$ and $C_3 = \alpha_1 \alpha_7 \alpha_2$. With this information, the program shows us the solution $1$.
\end{proof}

In general, it is difficult to describe the PHI algebras of type $\widetilde{\mathbb{D}}_m$. Thus, we begin with the type $\widetilde{\mathbb{D}}_4$ and the next type is $\widetilde{\mathbb{D}}_5$.

\begin{theorem} \label{d5til}

The  PHI algebras of the ANS family of type $\widetilde{\mathbb{D}}_5$ are described, below, as quiver and relations.

\NumTabs{2}
\begin{inparaenum} \setcounter{enumi}{3}

\noindent 
\item 
\begin{tikzcd} [column sep=small, row sep=small]
												&\bullet \ar[dl] \ar[dr] \ar[ddr, dotted, dash] \ar[ddl, dotted, dash]	&\\
\bullet	\ar[drr] \ar[d]	\ar[ddr, dotted, dash]	&&\bullet \ar[dll] \ar[d] \ar[ddl, dotted, dash]\\
\bullet	\ar[dr]									&&\bullet \ar[dl]\\
												&\bullet																&
\end{tikzcd}
\tab\item 
\begin{tikzcd} [column sep=small, row sep=small]
\bullet	\ar[drr] \ar[d]	\ar[ddr, dotted, dash]	&&\bullet \ar[dll] \ar[d] \ar[ddl, dotted, dash]\\
\bullet	\ar[dr]									&&\bullet \ar[dl]\\
												&\bullet \ar[d]&\\
												&\bullet&
\end{tikzcd}
\tab\item 
\begin{tikzcd} [column sep=small, row sep=small]
				&\bullet \ar[dl] \ar[dr] \ar[d] \ar[dd, dotted, dash, bend right=45] \ar[ddd, dotted, dash, bend left=25]	&\\
\bullet	\ar[dr]	&\bullet \ar[d]	&\bullet \ar[ddl]\\
				&\bullet \ar[d]	&\\
				&\bullet		&				
\end{tikzcd}
\tab\item  
\begin{tikzcd} [column sep=small, row sep=small]
\bullet	\ar[dr]	&														&\bullet\\
				&\bullet \ar[ur] \ar[dd, dash, dotted] \ar[dl] \ar[dr]	&\\
\bullet	\ar[dr]	&														&\bullet \ar[dl]\\
				&\bullet 												&
\end{tikzcd}
\tab\item 
\begin{tikzcd} [column sep=small, row sep=small]
\bullet			&																&\bullet\\
				&\bullet \ar[ul] \ar[ur] \ar[dd, dash, dotted] \ar[dl] \ar[dr]	&\\
\bullet	\ar[dr]	&																&\bullet \ar[dl]\\
				&\bullet 														&
\end{tikzcd}
\end{inparaenum}

\end{theorem}
\begin{proof}
We proceed as in the former theorem, now we are interested in the type $\widetilde{\mathbb{D}}_5$. The frame that satisfies this condition is $\mathcal{F}r2$, see \cite{sim-sko2}. The adjacent quiver $\widetilde{\mathbb{D}}_5$ has $32$ of possibilities of admissible operations $1$, that is, we have $32$ Euclidean graphs $\widetilde{\mathbb{D}}_5$. As before, to exemplify the work that has to be done. 
we will show this procedure in a particular case of  trivial extension.

\begin{tikzcd} [column sep=tiny, row sep=tiny]
\bullet	\ar[dr]	&				&						&\bullet \\
				&\bullet \ar[r]	&\bullet \ar[dr] \ar[ur]&\\
\bullet	\ar[ur]	&				&						&\bullet					
\end{tikzcd} \hfill
\begin{tikzcd} [column sep=tiny, row sep=tiny]
\bullet	&						&				&\bullet \ar[dl]\\
		&\bullet \ar[dl] \ar[ul]&\bullet \ar[l] &\\
\bullet	&						&				&\bullet \ar[ul]				
\end{tikzcd} \hfill
\begin{tikzcd} 
\bullet	\ar{dr}[sloped]{\alpha_5} &&&\bullet \ar{lll}[swap]{\alpha_7} \ar[bend right]{ddlll}[sloped]{\alpha_8}\\
&\bullet \ar{r}[swap]{\alpha_3}	&\bullet \ar{dr}[sloped]{\alpha_1} \ar{ur}[sloped, swap]{\alpha_2}&\\
\bullet	\ar{ur}[sloped, swap, near start]{\alpha_4} &&&\bullet \ar{lll}{\alpha_6} \ar[bend left]{uulll}[swap, sloped]{\alpha_9}
\end{tikzcd}

The relations of type $2$ are: $r1=\alpha_6 \alpha_4 \alpha_3 \alpha_2$, $r2=\alpha_4 \alpha_3 \alpha_1 \alpha_9$, $r3=\alpha_5 \alpha_3 \alpha_1 \alpha_6$, $r4=\alpha_9 \alpha_5 \alpha_3 \alpha_2$, $r5=\alpha_5 \alpha_3 \alpha_2 \alpha_8$, $r6=\alpha_7 \alpha_5 \alpha_3 \alpha_1$, $r7=\alpha_8 \alpha_4 \alpha_3 \alpha_1$ and $r8=\alpha_4 \alpha_3 \alpha_2 \alpha_7$. The elementary cycles are: $C_1 = \alpha_1 \alpha_6 \alpha_4 \alpha_3$, $C_2 = \alpha_1 \alpha_9 \alpha_5 \alpha_3$, $C_3 = \alpha_2 \alpha_7 \alpha_5 \alpha_3$ e $C_4 = \alpha_2 \alpha_8 \alpha_4 \alpha_3$. With this information, the program shows us the solutions $1$, $2$ and the opposit algebra of the algebra in item  $2$, above.
\end{proof}

After describing above  the PHI algebras of ANS family of type $\widetilde{\mathbb{D}}_4$ and $\widetilde{\mathbb{D}}_5$, we were able to generalize the procedure  and we give a description of  
the PHI algebras of ANS family of type $\widetilde{\mathbb{D}}_n$ for $n \geq 6$.

The following lemma is used several times in proof of the theorem describing the PHI algebras of ANS family of type $\widetilde{\mathbb{D}}_n$.

\begin{lemma} \label{corteAn}
Given the trivial extension $\Gamma$ of the hereditary algebra of type $A_n$, with at least two elementary cycles, we consider its diagram below, where $\beta_i$ and $\beta_{i + 1}$ 
have opposit order for every $i$.
\begin{center}
$\Gamma :$ 
\begin{tikzcd} [column sep=small, row sep=small]
\bullet \ar[rr, end anchor=165, bend left, "\alpha_2"] &\bullet \ar[l, "\alpha_1"] &\bullet \ar[l, dashed, "\beta_1"] *** \bullet \ar[r, swap, dashed, "\beta_n"] &\bullet \ar[r, swap, "\alpha_n"] &\bullet \ar[ll, bend right, swap, end anchor=15, "\alpha_{n+1}"]
\end{tikzcd}
\end{center}
We will use the dashed arrow to represent a path which may occur  in the quiver. Moreover, the symbol $***$ in the quiver represents the possibility of attaching subquivers of the form
\begin{tikzcd} 
\bullet	\ar[r, dashed]	&\bullet \ar[r]&\bullet \ar[ll, bend left]
\end{tikzcd}
or of the opposit form.

\begin{asparaenum} [\itshape i)]
\item  There are cutting sets $\Sigma$ of $\Gamma$ such that $KQ_\Gamma / <I_\Gamma \cup \Sigma>$  is an hereditary algebra whose quiver has underlying graph an $A_n$. 


\item We also assume that there is an elementary cycle with three arrows or more in $\Gamma$.

If $\alpha_1 \in \Sigma$ or $\alpha_n \in \Sigma$, then $KQ_\Gamma / <I_{\Gamma} \cup \Sigma>$ is not an incidence algebra.
If $\alpha_2 \in \Sigma$ or $\alpha_{n+1} \in \Sigma$, then $KQ_\Gamma / <I_{\Gamma} \cup \Sigma>$ is an incidence algebra with a quiver whose underlying graph is a  $A_n$.
\end{asparaenum}
\end{lemma}
\begin{proof}
The aim is to find a cutting set $\Sigma$ of $\Gamma$ where $KQ_\Gamma / <I_{\Gamma} \cup \Sigma>$ is an incidence algebra. 
First, let us assume that there is an elementary cycle with three arrows or more. 
We will show the existence of $\Sigma$ by induction in the number of elementary cycles. For clarity, we will begin the first step of the induction process with two elementary cycles:

\begin{center}
\begin{tikzcd} [column sep=small, row sep=small]
\bullet \ar[bend left]{rr}{\alpha_2} &\bullet \ar{l}{\alpha_1} &\bullet \ar[dashed]{l}{\beta_1} \ar[dashed]{r}[swap]{\beta_3} &\bullet \ar{r}[swap]{\alpha_3} &\bullet \ar[bend right]{ll}[swap]{\alpha_4}
\end{tikzcd}
\end{center}

The relations of type $2$ are: $r1=\alpha_2 \beta'_3$, where $\beta'_3$ is the first arrow of the path $\beta_3$, and $r2=\alpha_4 \beta'_1$, where $\beta'_1$ is the first arrow of the path $\beta_1\alpha_1$. 

The elementary cycles are: $C_1 = \alpha_2 \beta_1 \alpha_1$ and $C_2 = \alpha_4 \beta_3 \alpha_3$. 

Without loss of generality, we assume that exists the path $\beta_1$. Also, we start constructing the cutting set $\Sigma$ choosing the arrow of the elementary cycle $C_1$.

We consider $\alpha_2 \in \Sigma$. Then, eliminating the relation $r2=\alpha_4 \beta_1$, we get $\alpha_4 \in \Sigma$. So using $\Sigma = \{ \alpha_2, \alpha_4 \}$ we have the incidence algebra:
\begin{center}
$KQ_\Gamma / <I_{\Gamma} \cup \Sigma>:$
\begin{tikzcd} [column sep=small, row sep=small]
\bullet &\bullet \ar[l] &\bullet \ar[l, dashed] \ar[r, dashed] &\bullet \ar[r] &\bullet
\end{tikzcd}
\end{center}

Let $\Gamma$ be a trivial extension of the lemma with $m+1$ elementary cycles. We call the set $\Sigma=\{\alpha_2, \alpha_4, \dotsc, \alpha_{2m}\}$ for $m$ elementary cycles, where it satisfies the hypothesis of induction. We need to define which arrow of the elementary cycle $C_{m+1}$ would complete the cutting set $\Sigma$ of $\Gamma$:
\begin{center}
\begin{tikzcd} 
\bullet \ar[end anchor=165, bend left]{rr}{\alpha_2} &\bullet \ar{l}{\alpha_1} &\bullet \ar[dashed]{l}{\beta_1} *** \bullet \ar[bend left, start anchor=15]{rr}{\alpha_{2m}} &\bullet \ar[dashed]{l}{\beta_{2m-1}} &\bullet \ar[dashed]{r}[swap]{\beta_{2m+1}} \ar{l}{\alpha_{2m-1}} &\bullet \ar{r}[swap]{\alpha_{2m+1}} &\bullet \ar[bend right]{ll}[swap]{\alpha_{2m+2}}
\end{tikzcd}
\end{center}

The relations of type $2$ are: $r1=\alpha_2 \beta_3$, $r2=\alpha_4 \beta_1$, $\dotsc$, $rk=\alpha_{2m}\beta_{2m+1}$ and $r(k+1)=\alpha_{2m+2}\alpha_{2m-1}$, for some natural number $k$. 
The elementary cycles are: $C_1 = \beta_1 \alpha_1 \alpha_2$, $\dotsc$, $C_{m+1} = \beta_{2m+1} \alpha_{2m+1} \alpha_{2m+2}$.

Therefore, by the relation $r(k+1)=\alpha_{2m+2}\alpha_{2m-1}$, we determine the arrow $\alpha_{2m+2} \in \Sigma$. We conclude this part, and the proof is analogous to $\alpha_{n+1} \in \Sigma$.

Now, if the trivial extension $\Gamma$ have only elementary cycles with two arrows:
\begin{center}
\begin{tikzcd} [column sep=large]
\bullet \ar[end anchor=165, bend left]{r}{\alpha_2} &\bullet \ar{l}{\alpha_1} *** \bullet \ar[bend left, start anchor=15]{r}{\alpha_{2m}} &\bullet  \ar{r}[swap]{\alpha_{2m+1}} \ar{l}{\alpha_{2m-1}} &\bullet \ar[bend right]{l}[swap]{\alpha_{2m+2}}
\end{tikzcd}
\end{center}

The relations of type $2$ are: $r1=\alpha_2 \alpha_3$, $r2=\alpha_4 \alpha_1$, $\dotsc$, $rk=\alpha_{2m}\alpha_{2m+1}$ and $r(k+1)=\alpha_{2m+2}\alpha_{2m-1}$, for some natural number $k$. 
The elementary cycles are: $C_1 = \alpha_1 \alpha_2$, $\dotsc$, $C_{m+1} = \alpha_{2m+1} \alpha_{2m+2}$.

The process is analogous to have the cutting set $\Sigma=\{ \alpha_2, \alpha_4, \dotsc, \alpha_{2m+2} \}$ that define the incidence algebra. And we can verify that the cutting set $\Sigma=\{ \alpha_1, \alpha_3, \dotsc, \alpha_{2m+1} \}$ defines the opposit  incidence algebra.

The last part is missing. We assume that there is an elementary cycle with three arrows or more in $\Gamma$, with $\alpha_1 \in \Sigma$ or $\alpha_n \in \Sigma$. Without loss of generality, we assume that exists the path $\beta_1$. We consider the subquiver:
\begin{center}
\begin{tikzcd} [column sep=small, row sep=small]
\bullet \ar[bend left]{rr}{\alpha_2} &\bullet \ar{l}{\alpha_1} &\bullet \ar[dashed]{l}{\beta_1} \ar[dashed]{r}[swap]{\beta_3} &\bullet \ar{r}[swap]{\alpha_3} &\bullet \ar[bend right]{ll}[swap]{\alpha_4}
\end{tikzcd}.
\end{center}

The relations of type $2$ are: $r1=\alpha_2 \beta_3$ and $r2=\alpha_4 \beta_1$. 
The elementary cycles are: $C_1 = \alpha_2 \beta_1 \alpha_1$ and $C_2 = \beta_3 \alpha_3 \alpha_4$.

By hypothesis, we have $\alpha_1 \in \Sigma$. Then we have to choose an arrow from the elementary cycle $C_2=\beta_3 \alpha_3 \alpha_4$. Independently of  the choice, we cannot rule out the relation $r1=\alpha_2 \beta_3$ or $r2=\alpha_4 \beta_1$. Therefore, $KQ_\Gamma / <I_{\Gamma} \cup \Sigma>$ is not a incidence algebra. Similarly, we get the same the result if $\alpha_n \in \Sigma$.
\end{proof}

Next, we will display the PHI algebras of the ANS family of type $\widetilde{\mathbb{D}}_n$. The more complicated case in the proof  is in the case of the hereditary algebra $A$ of type $\widetilde{\mathbb{D}}_n$, that is when there is no relation. The possible quiver algebras, with a given frame is very large. So we can get a large amount of possible trivial extensions
associated with a given frame. 
\begin{theorem} \label{dntil}
The PHI algebras of the ANS family of type  $\widetilde{\mathbb{D}}_n$, for $n \geq 6$, are described below, via presentations of quivers and relations:
\NumTabs{2}
\begin{inparaenum} 

\noindent 
\item 
\begin{tikzcd} [column sep=scriptsize, row sep=scriptsize]
												&\bullet \ar[d]	&\\
												&\rvdots \ar[d]	&\\
												&\bullet \ar[dl] \ar[dr] \ar[ddr, dotted, dash] \ar[ddl, dotted, dash]	&\\
\bullet	\ar[drr] \ar[d]	\ar[ddr, dotted, dash]	&				&\bullet \ar[dll] \ar[d] \ar[ddl, dotted, dash]\\
\bullet	\ar[dr]									&				&\bullet \ar[dl]\\
												&\bullet		&
\end{tikzcd}
\tab\item 
\begin{tikzcd} [column sep=scriptsize, row sep=scriptsize]
						&\bullet \ar[d]	&\\
						&\rvdots \ar[d]	&\\
						&\bullet \ar[dl] \ar[dr] \ar[ddr, dotted, dash] \ar[ddl, dotted, dash] &\\
\bullet	\ar[drr] \ar[d]	&				&\bullet \ar[dll] \ar[d] \\
\bullet					&				&\bullet
\end{tikzcd}

\noindent
\item 
\begin{tikzcd} [column sep=scriptsize, row sep=scriptsize]
				&\bullet \ar[dl] \ar[dr] \ar[d] \ar[dd, dotted, dash, bend right=45] \ar[dddd, dotted, dash, bend left=16]	&\\
\bullet	\ar[dr]	&\bullet \ar[d]	&\bullet \ar[dddl]\\
				&\bullet \ar[d]	&\\
				&\rvdots \ar[d]	&\\
				&\bullet		&
\end{tikzcd}
\tab\item 
\begin{tikzcd} [column sep=tiny, row sep=scriptsize]
		&				&										&\bullet \ar[dl]&										&\bullet \\
\bullet &\dotsb \ar[l] &\bullet \ar[l] \ar[rr, dash, dotted] 	&				&\bullet \ar[ul] \ar[dl] \ar[dr] \ar[ur]&\\
		&				&										&\bullet \ar[ul]&										&\bullet
\end{tikzcd}

\noindent
\item 
\begin{tikzcd} [column sep=tiny, row sep=scriptsize]
								&\bullet \ar[dl]	&											&									&												&\bullet \\
\bullet \ar[rr, dash, dotted] 	&					&\bullet \ar[ul] \ar[dl] \ar[r, dash]		& \ar[l, dash] \dotsb \ar[r, dash]	&\bullet \ar[l,dash] \ar[ur,dash] \ar[dr,dash]	&\\
								&\bullet \ar[ul]	&											&									&												&\bullet
\end{tikzcd}

\noindent
\item 
\begin{tikzcd} [column sep=tiny, row sep=scriptsize]
&&&\bullet \ar[dl]	&&&&\bullet \\
\bullet	&\dotsb \ar[l] 	&\bullet \ar[l] \ar[rr, dash, dotted] 	&&\bullet \ar[ul] \ar[dl] \ar[r, dash]	& \ar[l, dash] \dotsb \ar[r, dash]	&\bullet \ar[l,dash] \ar[ur,dash] \ar[dr,dash]	&\\
&&&\bullet \ar[ul]	&&&&\bullet
\end{tikzcd}

\noindent
\item 
\begin{tikzcd} [column sep=tiny, row sep=scriptsize]
		&				&										&\bullet \ar[dl]&								&				&\bullet \\
\bullet	&\dotsb \ar[l]	&\bullet \ar[l] \ar[rr, dash, dotted] 	&				&\bullet \ar[ul] \ar[dl] \ar[r]	&\bullet \ar[ur]&\\
		&				&										&\bullet \ar[ul]&								&				&\bullet \ar[ul]
\end{tikzcd}

\noindent
\item 
\begin{tikzcd} [column sep=tiny, row sep=scriptsize]
		&				&										&\bullet \ar[dl]&										&\bullet \ar[dr] 	&										&				&\\
\bullet	&\dotsb \ar[l]	&\bullet \ar[l] \ar[rr, dash, dotted]	&				&\bullet \ar[ul] \ar[dl] \ar[ur] \ar[dr]&					&\bullet \ar[ll, dash, dotted] \ar[r] 	&\dotsb \ar[r] 	&\bullet \\
		&				&										&\bullet \ar[ul]&										&\bullet \ar[ur]	&										&				&
\end{tikzcd}

\noindent
\item 
\begin{tikzcd}[column sep=tiny, row sep=scriptsize]
&&&\bullet \ar[dl] &&&&&&\bullet \ar[dl] &&&\\
\bullet &\dotsb\ar[l] &\bullet \ar[l] \ar[rr, dash, dotted]	&&\bullet \ar[ul] \ar[dl] \ar[r] \arrow[%
	,dash
    ,thick
    ,start anchor=center
    ,end anchor=center
    ,decorate
    ,decoration={%
        ,brace
        ,amplitude=5pt
        ,raise=10pt
        ,mirror
        }
    ]{rrrr}[below=15pt]{\text{odd number}} &\bullet & \ar[l] \dotsb &\bullet \ar[l] \ar[r] &\bullet &&\bullet \ar[ll, dash, dotted] \ar[ul] \ar[dl] \ar[r] &\dotsb \ar[r] &\bullet \\
&&&\bullet \ar[ul] &&&&&&\bullet \ar[ul] &&&
\end{tikzcd}
\tab\item 
\begin{tikzcd} [column sep=tiny, row sep=scriptsize]
		&				&										&\bullet \ar[dl]&									&\bullet \ar[dl]\\
\bullet	&\dotsb \ar[l]	&\bullet \ar[l] \ar[rr, dash, dotted] 	&				&\bullet \ar[ul] \ar[dl] \ar[dr]	&\\
		&				&										&\bullet \ar[ul]&									&\bullet
\end{tikzcd}

\noindent
\item 
\begin{tikzcd} [column sep=scriptsize, row sep=scriptsize]
&\bullet \ar[dl] &&&&&&\bullet \\
\bullet \ar[rr, dash, dotted] &&\bullet \ar[ul] \ar[dl] \ar[r] 
\arrow[%
	,dash
    ,thick
    ,start anchor=center
    ,end anchor=center
    ,decorate
    ,decoration={%
        ,brace
        ,amplitude=5pt
        ,raise=10pt
        ,mirror
        }
    ]{rrr}[below=15pt]{\text{odd number}}&\bullet	& \ar[l] \dotsb \ar[r] &\bullet	&\bullet \ar[l] \ar[ur] &\\
&\bullet \ar[ul]&&&&&&\bullet \ar[ul]
\end{tikzcd}
\tab\item 
\begin{tikzcd} [column sep=scriptsize, row sep=scriptsize]
				&\bullet \ar[dr] \ar[d] \ar[dddd, dotted, dash, bend left=16] &&&\\
				&\rvdots \ar[d]	&\bullet \ar[dddl] &\dotsb \ar[l] &\bullet \ar[l]\\
				&\bullet \ar[dl] \ar[d] \ar[dd, dotted, dash, bend right=45]&&&\\
\bullet	\ar[dr]	&\bullet \ar[d]	&&&\\
				&\bullet		&&&
\end{tikzcd}
\end{inparaenum}

There are two cases in the substitution of subquiver 
\begin{tikzcd} [column sep=tiny, row sep=scriptsize]
						&						&									&\bullet \\
\bullet \ar[r, dash]	& \dotsb \ar[r, dash]	&\bullet \ar[ur,dash] \ar[dr,dash]	&\\
						&						&									&\bullet
\end{tikzcd}. 

If we have an even number of edges then we replace by
\begin{center}
\begin{tikzcd} [column sep=tiny, row sep=scriptsize]
				&			&						&			&								&\bullet \\
\bullet \ar[r]	&\bullet 	& \dotsb \ar[l] \ar[r]	&\bullet	&\bullet \ar[l] \ar[ur] \ar[dr]	&\\
				&			&						&			&								&\bullet
\end{tikzcd}
\end{center} 
otherwise, we replace by
\begin{tikzcd} [column sep=tiny, row sep=scriptsize]
				&			&				&						&			&\bullet \ar[dl]\\
\bullet \ar[r]	&\bullet 	& \dotsb \ar[l] &\bullet \ar[l] \ar[r]	&\bullet 	&\\
				&			&				&						&			&\bullet \ar[ul]
\end{tikzcd}

\end{theorem}
\begin{proof} \label{provadntil}
We describe the script of the proof for  each PHI algebra of the ANS family of type  $\widetilde{\mathbb{D}}_n$ for $n \geq 6$:
\begin{itemize}
\item we consider each schurian concealed algebra $A$ of type $\widetilde{\mathbb{D}}_n$;
\item we describe the trivial extensions of $A$;
\item we analyze the cutting sets that define an incidence algebra $K\Delta$;
\item we use the theorem \ref{ANS}, in order to proof that $K\Delta$ is a PHI algebra of type $\widetilde{\mathbb{D}}_n$.
\end{itemize}

We begin with the hereditary algebra $A$ of type $\widetilde{\mathbb{D}}_n$, with $n \geq 6$. First, we collect the configurations of the elementary cycles of the trivial extensions of $A$ of type $\widetilde{\mathbb{D}}_4$ and $\widetilde{\mathbb{D}}_5$. 

Next we use these configurations in the trivial extensions of the hereditary algebra $A$, of type $\widetilde{\mathbb{D}}_n$.\\
We list one trivial extension for each par of algebras $(A, A^{op})$. The label $\mathcal{F}r2$ is the same as the one on the list giving in  chapter XIV of the book ``Elements of the representation theory of associative algebras'' \cite{sim-sko2}.

The dashed arrow symbolizes an oriented path of whose  length depends on $n$.

We will show in detail two cases, the other cases can be obtained in an analogous way.

\begin{center}
\begin{tikzcd} 
\bullet	\ar{dr}[sloped, near start]{\alpha_4} &&&&&\bullet \ar{lllll}[swap, near end]{\alpha_7}  \ar[bend right]{ddlllll}{\alpha_{10}}\\
&\bullet \ar{r}[swap]{\alpha_1} &\bullet \ar[dashed]{r}[swap]{\beta} & \bullet \ar{r}[swap]{\alpha_2} &\bullet \ar{dr}[sloped, swap]{\alpha_5} \ar{ur}[sloped]{\alpha_6} &\\
\bullet	\ar{ur}[sloped, swap, near start]{\alpha_3} &&&&&\bullet \ar{lllll}[near end]{\alpha_8} \ar[bend left]{uulllll}[swap]{\alpha_9}
\end{tikzcd}
\end{center}

The relations of type $2$ are: $r1=\alpha_4 \alpha_1 \beta \alpha_2 \alpha_6 \alpha_{10}$, $r2=\alpha_3 \alpha_1 \beta \alpha_2 \alpha_6 \alpha_7$, $r3=\alpha_7 \alpha_4 \alpha_1 \beta \alpha_2 \alpha_5$, $r4=\alpha_9 \alpha_4 \alpha_1 \beta \alpha_2 \alpha_6$, $r5=\alpha_{10} \alpha_3 \alpha_1 \beta \alpha_2 \alpha_5$, $r6=\alpha_8 \alpha_3 \alpha_1 \beta \alpha_2 \alpha_6$, $r7=\alpha_3 \alpha_1 \beta \alpha_2 \alpha_5 \alpha_9$ and $r8=\alpha_4 \alpha_1 \beta \alpha_2 \alpha_5 \alpha_8$.

The elementary cycles are: $C_1 = \alpha_6 \alpha_{10} \alpha_3 \alpha_1 \beta \alpha_2$, $C_2 = \alpha_5 \alpha_8 \alpha_3 \alpha_1 \beta \alpha_2$, $C_3 = \alpha_5 \alpha_9 \alpha_4 \alpha_1 \beta \alpha_2$ and $C_4= \alpha_6 \alpha_7 \alpha_4 \alpha_1 \beta \alpha_2$.

As in the previous theorems referring to the type $\widetilde{\mathbb{D}}_4$ and $\widetilde{\mathbb{D}}_5$, there are four cutting sets defining the PHI algebras $1$ and $2$, and their opposit algebras. The PHI algebras $1$ and $2$ are defined from the cutting sets $\{\alpha_1\}$ e $\{\alpha_3,\alpha_4\}$, respectively.

The dashed arrows represent a possible oriented path to be added to the  quiver.
 
Moreover, the symbol $***$ in the quiver represents the possibility of placing subquivers of the form
\begin{tikzcd} 
\bullet	\ar[r, dashed]	&\bullet \ar[r]&\bullet \ar[ll, bend left]
\end{tikzcd}
or opposit form. For example, we would replace $***$ by
\begin{center}
\begin{tikzcd} 
\bullet	\ar[r, dashed]	&\bullet \ar[r]&\bullet \ar[r, bend right] \ar[ll, bend left]&\bullet \ar[l]
\end{tikzcd}
\end{center}

In the search for non-hereditary PHI algebras, in the trivial extensions, we are using the lemma \ref{cortephia}. 
This result requires that there is an arrow which belongs two elementary cycles, at least. 
We will show now the trivial extensions satisfying the following two hypothesis of lemma \ref{cortephia}:
\begin{itemize}
\item there is, at least, one pair of elementary cycles which share an arrow;
\item there is, at most, two pairs of elementary cycles, sharing one arrow.
\end{itemize}
We will show the possible configurations, assuming the restriction above:

\begin{tikzcd} [column sep=small, row sep=small]
\bullet	\ar[dr]	&						&				&\\
				&\bullet \ar[r, dashed]	&\bullet \ar[r]	&\bullet \ar[ulll, bend right] \ar[dlll, bend left]\\
\bullet	\ar[ur]	&						&				&
\end{tikzcd}\hfill
\begin{tikzcd} [column sep=small, row sep=small]
\bullet	\ar[dr]							&					&						&\\
										&\bullet \ar[dl] 	&\bullet \ar[l, dashed] &\bullet \ar[l]\\
\bullet	\ar[uu]	\ar[urrr, bend right]	&					&						&
\end{tikzcd}

The trivial extension has at least five elementary cycles. In this part of the proof, we will strongly use lemma \ref{corteAn}.

\begin{center}
\begin{tikzcd} 
\bullet	\ar{dr}[sloped, swap]{\alpha_4}	&&&&&&\bullet \ar[bend right, end anchor=15]{dl} \\
&\bullet \ar[dashed]{r}{\beta'} &\bullet \ar{r}{\alpha_2} &\bullet \ar[end anchor=165, bend left]{rr}{ \rho_2} \ar[bend right]{ulll}[swap]{\alpha_7} \ar[bend left]{dlll}{\alpha_8} &\bullet \ar{l}{\rho_1} &\bullet \ar[dashed]{l}{\beta} *** \bullet \ar[start anchor=10]{ur} \ar[start anchor=-10]{dr} &\\
\bullet	\ar{ur}[sloped]{\alpha_3} &&&&&&\bullet \ar[ul, bend left, end anchor=-18]
\end{tikzcd}
\end{center}

The relations of type $2$ are: $r1=\alpha_4 \beta' \alpha_2 \alpha_8$, $r2=\alpha_3 \beta' \alpha_2 \alpha_7$, $r3=\alpha_2 \rho_2$, $r4=\rho_1 \alpha_7$, $r5=\rho_1 \alpha_8$ all  other 
relations of type 2 have support on the right part of the quiver, (which we give bellow, in order to make it clear).

\begin{center}
\begin{tikzcd} 
&&&\bullet \ar[dl, bend right, end anchor=15]\\
\bullet \ar[end anchor=165, bend left]{rr}{\rho_2} &\bullet \ar{l}{\rho_1} &\bullet \ar[dashed]{l}{\beta} *** \bullet \ar[ur, start anchor=10] \ar[dr, start anchor=-10] &\\
&&&\bullet \ar[ul, bend left, end anchor=-18]
\end{tikzcd}
\end{center}

The elementary cycles are: $C_1 = \beta' \alpha_2 \alpha_7 \alpha_4$, $C_2 = \beta' \alpha_2 \alpha_8 \alpha_3$ and the others are part of the previous subquiver.

Next, we will examine  the possible  configuration, that is the possible elementary cycles and relations of type 2 in this trivial extension $\Gamma$.
First we assume that the path $\beta$ has  length greater than or equal to $1$ or that  the symbol $***$ is replaced by some elementary cycle with three arrows or more.\\
We will look for a cutting set $\Sigma$ such that $KQ_\Gamma / <I_\Gamma \cup \Sigma>$ is a non-hereditary incidence algebra.

In order to get that it is necessary to have $\alpha_2 \in \Sigma$. Otherwise, if the path $\beta'$, (in the picture), has length at least one, then for any arrow in the path $\beta'$ which is in
$\Sigma$, we get that  $\alpha_2 \rho_2$ or $\rho_1 \alpha_7$ belong to the ideal $<I_\Gamma \cup \Sigma>$, so the quotient would not be an incidence algebra. If  $\alpha_2 \in \Sigma$, then  $\rho_1 \in \Sigma $, another time, in this case using  lemma \ref{corteAn}, we obtain that $KQ_{\Gamma}/<I_\Gamma \cup \Sigma>$ is not an incidence algebra.

Next we look at configurations where the  path $\beta$ does not exist, and the symbol $***$ is  replaced only by elementary cycle(s) with two arrows. As seen in the previous paragraph, we need to have $\alpha_2 \in \Sigma$, implying $\rho_1 \in \Sigma$. Again we use the lemma \ref{corteAn}, implying a single cutting set. Consequently, we get a solution.

The solution has the quiver which depends on the length of the path $\beta'$ and the number of elementary cycles in the middle of the trivial extension $\Gamma$. For example,
If $\Gamma$ has the path $\beta'$ with length one and only one elementary cycle in the middle of $\Gamma$,  we get the following PHI algebra:
\begin{center}
\begin{tikzcd} [column sep=tiny, row sep=scriptsize]
		&										&\bullet \ar[dl]&								&			&\bullet \ar[dl]\\
\bullet	&\bullet \ar[l] \ar[rr, dash, dotted] 	&				&\bullet \ar[ul] \ar[dl] \ar[r]	&\bullet	&\\
		&										&\bullet \ar[ul]&								&			&\bullet \ar[ul]
\end{tikzcd}
\end{center}

We get two solutions. If the path $\beta'$, has length bigger or equal 1 then we have the solution $6$, otherwise we get the solution $5$.
\end{proof}

Next we will proof two theorems related  to the description of the PHI algebras of the ANS family of type  $\widetilde{\mathbb{E}}_6$ and  $\widetilde{\mathbb{E}}_7$.

\begin{theorem} \label{e6til}

The PHI algebras of the ANS family of type  $\widetilde{\mathbb{E}}_6$, are described as quiver and relations, as follows.
\NumTabs{3}
\begin{inparaenum}

\noindent 
\item 
\begin{tikzcd} [column sep=small, row sep=small]
				&\bullet \ar[dl] \ar[dr] \ar[ddd, dash , dotted] 	&\\
\bullet	\ar[d] 	&													&\bullet \ar[d]\\
\bullet \ar[dr]	&													&\bullet \ar[dl]\\
				&\bullet \ar[d]										&\\
				&\bullet 											&         
\end{tikzcd}
\tab\item 
\begin{tikzcd} [column sep=small, row sep=small]
				&\bullet \ar[d]										&\\
				&\bullet \ar[dl] \ar[dr] \ar[ddd, dash , dotted] 	&\\
\bullet	\ar[d] 	&													&\bullet \ar[d]\\
\bullet \ar[dr]	&													&\bullet \ar[dl]\\
				&\bullet											&
\end{tikzcd}
\tab\item 
\begin{tikzcd} [column sep=small, row sep=small]
				&\bullet 												&\\
				&\bullet \ar[dl] \ar[dr] \ar[u] \ar[ddd, dash , dotted] &\\
\bullet	\ar[d] 	&														&\bullet \ar[d]\\
\bullet \ar[dr]	&														&\bullet \ar[dl]\\
				&\bullet												&
\end{tikzcd}

\noindent 
\item 
\begin{tikzcd} [column sep=small, row sep=small]
				&\bullet \ar[dl] \ar[dr] \ar[ddd, dash , dotted] 	&\\
\bullet	\ar[d] 	&													&\bullet \ar[d]\\
\bullet \ar[dr]	&													&\bullet \ar[dl]\\
				&\bullet 											&\\
				&\bullet \ar[u]										&         
\end{tikzcd}
\tab\item 
\begin{tikzcd} [column sep=small, row sep=small]
								&\bullet \ar[ddl] \ar[ddr]		&\\
&&\\
\bullet \ar[ddr, dash, dotted]	&\bullet \ar[uu]				&\bullet \ar[ddl, dash, dotted]\\
&&\\
\bullet \ar[uu] 				&\bullet \ar[l]	\ar[r] \ar[uu] 	&\bullet \ar[uu] 
\end{tikzcd}
\tab\item 
\begin{tikzcd} [column sep=small, row sep=small]
						&\bullet \ar[dd]									&\\
						&													&\\
\bullet \ar[r] 			&\bullet \ar[dr, dash, dotted] \ar[dl, dash, dotted]&\bullet \ar[l]\\
\bullet \ar[u] \ar[dr]	&													&\bullet \ar[u] \ar[dl]\\
						&\bullet \ar[uu]									& 
\end{tikzcd}

\noindent 
\item 
\begin{tikzcd} [column sep=small, row sep=small]
\bullet \ar[ddr] 								&\bullet \ar[ddl] \ar[ddr] 	&\bullet \ar[ddl]\\
												&							&\\
\bullet \ar[uu] \ar[ddr] \ar[r, dash, dotted]	&\bullet 					&\bullet \ar[uu] \ar[ddl] \ar[l, dash, dotted]\\
												&							&\\
												&\bullet \ar[uu] 			& 
\end{tikzcd}
\tab\item 
\begin{tikzcd} [column sep=small, row sep=small]
\bullet \ar[ddr] 								&\bullet 			&\bullet \ar[ddl]\\
												&					&\\
\bullet \ar[uu] \ar[ddr] \ar[r, dash, dotted] 	&\bullet \ar[uu]	&\bullet \ar[uu] \ar[l, dash, dotted] \ar[ddl]\\
												&					&\\
												&\bullet \ar[uu]	& 
\end{tikzcd}
\tab\item 
\begin{tikzcd} [column sep=small, row sep=small]
\bullet \ar[d] &&\bullet \ar[d] \\
\bullet \ar[ddr] &\bullet \ar[ddl] \ar[l] \ar[r] \ar[dd, dash, dotted] &\bullet \ar[ddl] \\
&&\\
\bullet \ar[r] &\bullet &
\end{tikzcd}

\noindent 
\item 
\begin{tikzcd} [column sep=small, row sep=small]
										&\bullet \ar[dl]	&									&\\
\bullet	\ar[dd]	\ar[ddrr, dash, dotted]	&					&\bullet \ar[ll] \ar[dr]			&\\
						 				& 					&									&\bullet\\
\bullet \ar[rr, dash, dotted]			&					&\bullet \ar[dl] \ar[uu] \ar[uuul] 	&\\
										&\bullet \ar[ul]	&									&
\end{tikzcd}
\tab\item 
\begin{tikzcd} [column sep=small, row sep=small]
					&\bullet \ar[dr]		&								&\\
					&\bullet \ar[u] \ar[dd]	&\bullet \ar[l, dash, dotted] 	&\\
\bullet \ar[urr]	&						&								&\bullet \ar[dl]\\
					&\bullet \ar[r] \ar[ul] &\bullet \ar[uu] 				&
\end{tikzcd}

\end{inparaenum}

\end{theorem}
\begin{proof} \label{provae6til}
The list has five frames, see \cite{sim-sko2}: $\mathcal{F}r6$, $\mathcal{F}r7$, $\mathcal{F}r8$, $\mathcal{F}r9$ and $\mathcal{F}r10$. We will analyze only the schurian frames, satisfying the main hypothesis of the theory of the section \ref{sec6}. We will describe the trivial extension of each algebra originated from each one of the frames. That is, given a frame $\mathcal{F}r$, we will apply the admissible operation and then describe the trivial extension of that algebra. Next we will input the necessary information for each trivial extension in the computer program and show the non-hereditary solutions.

We note that the concealed algebra of the type $\widetilde{\mathbb{E}}_6$ is always one of the solutions and we will omit it. 
We will show the argument just for $\mathcal{F}r6$, the argument for the other algebras which we obtain, is analogous.

The adjacent Euclidean quiver $\widetilde{\mathbb{E}}_6$ has $64$ of possibilities of admissible operation $1$, that is, we can have $64$ Euclidean graphs $\widetilde{\mathbb{E}}_6$. Now, our goal is to analyze the trivial extension of each graph and see if there is any cutting set that defines a non-hereditary PHI algebra. For this, in the trivial extension, we must have at least two elementary cycles that have at least one common arrow, according to lemma \ref{cortephia}.

So we need to have two maximal paths that have at least length two and one arrow in common. Up to duality, there are three cases:

\NumTabs{2}
\begin{inparaenum}

\noindent 
\item
\begin{tikzcd} [column sep=small, row sep=small]
						&				&\bullet \ar[d, dash]	&				&\\
						&				&\bullet 				&				&\\
\bullet	\ar[r, dash]	&\bullet \ar[r]	&\bullet \ar[u]			&\bullet \ar[l]	&\bullet \ar[l, dash]
\end{tikzcd}
\tab\item
\begin{tikzcd} [column sep=small, row sep=small]
						&				&\bullet \ar[d, dash]	&			&\\
						&				&\bullet \ar[d]			&			&\\
\bullet	\ar[r, dash]	&\bullet \ar[r]	&\bullet \ar[r] 		&\bullet 	&\bullet \ar[l, dash]
\end{tikzcd}
\tab\item
\begin{tikzcd} [column sep=small, row sep=small]
						&				&\bullet \ar[d, dash]	&			&\\
						&				&\bullet 				&			&\\
\bullet	\ar[r, dash]	&\bullet \ar[r]	&\bullet \ar[r] \ar[u]	&\bullet	&\bullet \ar[l, dash]
\end{tikzcd}
\end{inparaenum}

Up to symmetry, we will make all combinations of admissible operations $1$ from the three previous cases.

We start with the frame
$\mathcal{F}r6.1$, above. With one possible orientation of the unoriented edge, on the frame, we get the following algebra and its trivial extension.

\begin{tikzcd} [column sep=small, row sep=small]
				&				&\bullet 		&				&\\
				&				&\bullet \ar[u]	&				&\\
\bullet	\ar[r]	&\bullet \ar[r]	&\bullet \ar[u]	&\bullet \ar[l]	&\bullet \ar[l]
\end{tikzcd}\hfill%
\begin{tikzcd} [column sep=small, row sep=small]
&&\bullet \ar{ddll}[sloped, swap]{\alpha_5} \ar{ddrr}[sloped]{\alpha_6} &&\\
&&\bullet \ar{u}{\alpha_1} &&\\
\bullet	\ar{r}[swap]{\alpha_7} &\bullet \ar{r}[swap]{\alpha_3} &\bullet \ar{u}{\alpha_2} &\bullet \ar{l}{\alpha_4} &\bullet \ar{l}{\alpha_8}
\end{tikzcd}

The relations of type $2$ are: $r1=\alpha_3 \alpha_2 \alpha_1 \alpha_6$ and $r2=\alpha_4 \alpha_2 \alpha_1 \alpha_5$.

The elementary cycles are: $C_1 = \alpha_7 \alpha_3 \alpha_2 \alpha_1 \alpha_5$ and $C_2 = \alpha_8 \alpha_4 \alpha_2 \alpha_1 \alpha_6$. 

Then, the computer program shows us the cutting sets that define the algebras $1$ and $2$ of the theorem.
\end{proof}

Before stating the theorem describing the PHI algebras of the ANS family of type  $\widetilde{\mathbb{E}}_7$, we need the following lemma. 

\begin{lemma} \label{cortenphia}
Let $\Gamma$ be a trivial extension of a schurian algebra with at least three elementary cycles with the following conditions:
\begin{asparaenum} [\itshape a)]
\item two elementary cycles have at least one arrow in common, we shall call the path in the intersection by  $\lambda$;
\item the third elementary cycle has a vertex in common with one of the two previous elementary cycles and has no intersection with the other;
\item the common vertex of the previous condition does not belong to the path $\lambda$;
\end{asparaenum}

These conditions are ilustrated in the following picture:
\begin{center}
\begin{tikzcd}
&\bullet \ar{d}{\beta'} &&\\
\bullet \ar[bend left, dashed]{ur}{\theta} &\bullet \ar{l}{\beta} \ar{r}{\gamma_2} &\dotsb \ar{r}{\gamma_{n \menas 1}} &\bullet \ar{d}{\gamma_n}\\
&\bullet \ar{d}[swap]{\alpha_1} \ar{u}{\gamma_1} &&\bullet \ar[dashed]{ll}{\lambda}\\
&\bullet \ar{r}[swap]{\alpha_2} &\dotsb \ar{r}[swap]{\alpha_{m \menas 1}} &\bullet \ar{u}[swap]{\alpha_m}
\end{tikzcd}
\end{center}

Let $\Sigma$ be a cutting set of $\Gamma$. 
If any arrow of $\lambda$ belongs to $\Sigma$ then $KQ_{\Gamma}/<I_\Gamma \cup \Sigma>$ is not an incidence algebra.
\end{lemma}
\begin{proof}
Without loss of generality, we will consider $\Gamma$ as a trivial extension with three elementary cycles satisfying the conditions of the statement of the lemma.

Thus, the relations of type $2$ are: $r1=\gamma_1 \beta$, $r2=\beta' \gamma_2$, $r3=\alpha_m \lambda \gamma_1$ e $r4=\gamma_n \lambda \alpha_1$. 
The elementary cycles are: $C_1 = \alpha_1 \alpha_2 \dotsc \alpha_{m - 1} \alpha_m \lambda$, $C_2 =  \lambda \gamma_1 \gamma_2 \dotsc \gamma_{n - 1} \gamma_n$ e $C_3 = \theta \beta' \beta$. 

Assume that some arrow in the path $\lambda$ belongs to the cutting set $\Sigma$.  We will see that, in this case,  independently of  the choice of the arrow of the elementary cycle $C_3$, the $KQ_{\Gamma}/<I_\Gamma \cup \Sigma>$ is not an incidence algebra.

We assume that $\beta' \in \Sigma$, then the relation $r1$ belongs to $<I_\Gamma \cup \Sigma>$. Consequently, $KQ_{\Gamma}/<I_\Gamma \cup \Sigma>$ is not an incidence algebra.

Analogously, we arrive at the same result if $\beta \in \Sigma$ or any arrow of $\theta$ belongs to the cutting set.
\end{proof}

\begin{theorem} \label{e7til}

The PHI algebras of the ANS family of type  $\widetilde{\mathbb{E}}_7$ are described below.

\NumTabs{2}

\begin{inparaenum}
\noindent 
\item 
\begin{tikzcd} [column sep=tiny, row sep=tiny]
				&\bullet \ar[dl] \ar[dr] \ar[dddd, dash , dotted] 	&\\
\bullet	\ar[d] 	&													&\bullet \ar[d]\\
\bullet	\ar[d] 	&													&\bullet \ar[d]\\
\bullet \ar[dr]	&													&\bullet \ar[dl]\\
				&\bullet											&
\end{tikzcd}
\tab\item 
\begin{tikzcd} [column sep=tiny, row sep=tiny]
				&\bullet \ar[dl] \ar[ddr] \ar[dddd, dash , dotted] 	&&\\
\bullet	\ar[d] 	&													&&\\
\bullet	\ar[d] 	&													&\bullet \ar[ddl] &\\
\bullet \ar[dr]	&													&&\\
				&\bullet											&\bullet \ar[l] &\bullet \ar[l]\\
\end{tikzcd}
\tab\item 
\begin{tikzcd} [column sep=tiny, row sep=tiny]
				&\bullet \ar[dl] \ar[ddr] \ar[dddd, dash , dotted] 	&\bullet \ar[l]\\
\bullet	\ar[d] 	&													&\\
\bullet	\ar[d] 	&													&\bullet \ar[ddl]\\
\bullet \ar[dr]	&													&\\
				&\bullet											&\bullet \ar[l]
\end{tikzcd}
\tab\item 
\begin{tikzcd} [column sep=tiny, row sep=tiny]
				&\bullet \ar[dl] \ar[ddr] \ar[dddd, dash , dotted] 	&\\
\bullet	\ar[d] 	&													&\\
\bullet	\ar[d] 	&													&\bullet \ar[ddl]\\
\bullet \ar[dr]	&													&\\
				&\bullet											&\\
                &\bullet \ar[u]	\ar[r]								&\bullet
\end{tikzcd}
\tab\item 
\begin{tikzcd} [column sep=tiny, row sep=tiny]
				&\bullet \ar[dl] \ar[ddr] \ar[dddd, dash , dotted] 	&\\
\bullet	\ar[d] 	&													&\\
\bullet	\ar[d] 	&													&\bullet \ar[ddl]\\
\bullet \ar[dr]	&													&\\
				&\bullet \ar[d]										&\\
                &\bullet \ar[d]										&\\
                &\bullet 											&
\end{tikzcd}
\tab\item 
\begin{tikzcd} [column sep=tiny, row sep=tiny]
                &\bullet \ar[d]										&\\
				&\bullet \ar[dl] \ar[ddr] \ar[dddd, dash , dotted] 	&\\
\bullet	\ar[d] 	&													&\\
\bullet	\ar[d] 	&													&\bullet \ar[ddl]\\
\bullet \ar[dr]	&													&\\
				&\bullet \ar[d]										&\\
                &\bullet 											&
\end{tikzcd}
\tab\item 
\begin{tikzcd} [column sep=tiny, row sep=tiny]
                &\bullet \ar[d]										&\\
                &\bullet \ar[d]										&\\
				&\bullet \ar[dl] \ar[ddr] \ar[dddd, dash , dotted] 	&\\
\bullet	\ar[d] 	&													&\\
\bullet	\ar[d] 	&													&\bullet \ar[ddl]\\
\bullet \ar[dr]	&													&\\
				&\bullet 											&
\end{tikzcd}
\tab\item 
\begin{tikzcd} [column sep=tiny, row sep=tiny]
                &\bullet 													&\\
				&\bullet \ar[dl] \ar[ddr] \ar[dddd, dash , dotted] \ar[u]	&\\
\bullet	\ar[d] 	&															&\\
\bullet	\ar[d] 	&															&\bullet \ar[ddl]\\
\bullet \ar[dr]	&															&\\
				&\bullet \ar[d]												&\\
                &\bullet 													&               
\end{tikzcd}
\tab\item 
\begin{tikzcd} [column sep=tiny, row sep=tiny]
                &\bullet \ar[d]												&\\    
                &\bullet 													&\\
				&\bullet \ar[dl] \ar[ddr] \ar[dddd, dash , dotted] \ar[u]	&\\
\bullet	\ar[d] 	&															&\\
\bullet	\ar[d] 	&															&\bullet \ar[ddl]\\
\bullet \ar[dr]	&															&\\
				&\bullet 													&            
\end{tikzcd}
\tab\item  
\begin{tikzcd} [column sep=tiny, row sep=tiny]
                &\bullet 													&\\    
                &\bullet \ar[u]												&\\
				&\bullet \ar[dl] \ar[ddr] \ar[dddd, dash , dotted] \ar[u]	&\\
\bullet	\ar[d] 	&															&\\
\bullet	\ar[d] 	&															&\bullet \ar[ddl]\\
\bullet \ar[dr]	&															&\\
				&\bullet 													&            
\end{tikzcd}
\tab\item 
\begin{tikzcd} [column sep=tiny, row sep=tiny]
&\bullet \ar[dl] \ar[ddr] \ar[ddddr, bend right, dash, dotted] &&\bullet \ar[ddddl, bend left, dash, dotted] \ar[ddl] \ar[dr] &\\
\bullet \ar[dd] &&&&\bullet \ar[dd]\\
&&\bullet \ar[dd] &&\\
\bullet \ar[drr] &&&&\bullet \ar[dll]\\
&&\bullet &&
\end{tikzcd}
\tab\item
\begin{tikzcd} [column sep=tiny, row sep=tiny]
&\bullet \ar[dl] \ar[d]\ar[ddl, dotted, dash, bend right=4] &\\
\bullet \ar[d] &\bullet \ar[dl] \ar[dd, dotted, dash] \ar[r] &\bullet \ar[d]\\
\bullet \ar[dr] &&\bullet \ar[dl]\\
&\bullet \ar[d] &\\
&\bullet &
\end{tikzcd}
\tab\item 
\begin{tikzcd}[column sep=tiny, row sep=tiny]
&\bullet \ar[d] &\\
&\bullet \ar[dl] \ar[d]\ar[ddl, dotted, dash, bend right=4] &\\
\bullet \ar[d] &\bullet \ar[dl] \ar[dd, dotted, dash] \ar[r] &\bullet \ar[d]\\
\bullet \ar[dr] &&\bullet \ar[dl]\\
&\bullet &
\end{tikzcd}
\tab\item
\begin{tikzcd}[column sep=tiny, row sep=tiny]
&\bullet \ar[dl] \ar[ddr] \ar[dddd, dotted, dash] &&\\
\bullet \ar[dd] &&&\\
&&\bullet \ar[ddl] &\bullet \ar[l]\\
\bullet \ar[dr] &&&\\
&\bullet \ar[d] &&\\
&\bullet \ar[d]&&\\
&\bullet &&
\end{tikzcd}
\tab\item
\begin{tikzcd}[column sep=tiny, row sep=tiny]
&\bullet \ar[dl] \ar[ddr] \ar[dddd, dotted, dash] &&\\
\bullet \ar[dd] &&&\\
&&\bullet \ar[ddl] &\bullet \ar[l]\\
\bullet \ar[dr] &&&\\
&\bullet &&\\
&\bullet \ar[u]&&\\
&\bullet \ar[u]&&
\end{tikzcd}
\tab\item 
\begin{tikzcd}[column sep=tiny, row sep=tiny]
&&\bullet \ar[dl] \ar[dr] \ar[dddd, dotted, dash]&\\
&\bullet \ar[d] &&\bullet \ar[d]\\
&\bullet \ar[d] &&\bullet \ar[ddl]\\
\bullet &\bullet \ar[l] \ar[dr]&&\\
&&\bullet &
\end{tikzcd}
\tab\item
\begin{tikzcd}[column sep=tiny, row sep=tiny]
&\bullet \ar[dl] \ar[ddr] \ar[dddd, dotted, dash] &&\\
\bullet \ar[dd] &&&\\
&&\bullet \ar[ddl] &\bullet \ar[l]\\
\bullet \ar[dr] &&&\\
&\bullet &&\\
&\bullet \ar[u] \ar[d]&&\\
&\bullet &&
\end{tikzcd}
\tab\item
\begin{tikzcd}[column sep=tiny, row sep=tiny]
&\bullet \ar[dl] \ar[d]\ar[ddl, dotted, dash, bend right=4] &\\
\bullet \ar[d] &\bullet \ar[dl] \ar[dd, dotted, dash] \ar[r] &\bullet \ar[d]\\
\bullet \ar[dr] &&\bullet \ar[dl]\\
&\bullet  &\\
&\bullet \ar[u]&
\end{tikzcd}
\tab\item 
\begin{tikzcd}[column sep=tiny, row sep=tiny]
&\bullet&\\
&\bullet \ar[dl] \ar[u] \ar[d]\ar[ddl, dotted, dash, bend right=4] &\\
\bullet \ar[d] &\bullet \ar[dl] \ar[dd, dotted, dash] \ar[r] &\bullet \ar[d]\\
\bullet \ar[dr] &&\bullet \ar[dl]\\
&\bullet  &
\end{tikzcd}
\tab\item
\begin{tikzcd}[column sep=tiny, row sep=tiny]
&\bullet \ar[dl] \ar[ddr] \ar[dddd, dotted, dash] &&\\
\bullet \ar[dd] &&&\\
&&\bullet \ar[ddl] &\bullet \ar[l]\\
\bullet \ar[dr] &&&\\
&\bullet \ar[d]&&\\
&\bullet &&\\
&\bullet \ar[u] &&
\end{tikzcd}
\tab\item
\begin{tikzcd}[column sep=tiny, row sep=tiny]
&\bullet \ar[dl] \ar[dr] \ar[ddd, dotted, dash] &&\\
\bullet \ar[d] &&\bullet \ar[d] &\\
\bullet \ar[dr] &&\bullet \ar[dl] \ar[dr] \ar[dd, dotted, dash]&\\
&\bullet \ar[dr] &&\bullet \ar[dl]\\
&&\bullet &
\end{tikzcd}
\tab\item 
\begin{tikzcd}[column sep=tiny, row sep=tiny]
&\bullet \ar[d]&&&\\
&\bullet \ar[dl] \ar[dr] \ar[ddd, dotted, dash]&&\\
\bullet \ar[d] &&\bullet \ar[d]&\\
\bullet \ar[dr] &&\bullet \ar[dl] \ar[r] &\bullet\\
&\bullet &&
\end{tikzcd}
\tab\item 
\begin{tikzcd}[column sep=tiny, row sep=tiny]
&\bullet &&&\\
&\bullet \ar[u] \ar[dl] \ar[dr] \ar[ddd, dotted, dash]&&\\
\bullet \ar[d] &&\bullet \ar[d]&\\
\bullet \ar[dr] &&\bullet \ar[dl] \ar[r] &\bullet\\
&\bullet &&
\end{tikzcd}
\tab\item 
\begin{tikzcd}[column sep=tiny, row sep=tiny]
&&\bullet \ar[dl] \ar[dr] \ar[dd, dotted, dash] &\\
&\bullet \ar[dl] \ar[dr] \ar[ddd, dotted, dash] &&\bullet \ar[dl]\\
\bullet\ar[d] &&\bullet \ar[d] &\\
\bullet \ar[dr] &&\bullet \ar[dl]&\\
&\bullet&&
\end{tikzcd}
\tab\item 
\begin{tikzcd}[column sep=tiny, row sep=tiny]
&\bullet \ar[dl] \ar[dr] \ar[ddd, dotted, dash]&&\\
\bullet \ar[d] &&\bullet \ar[d] &\bullet \ar[l]\\
\bullet \ar[dr] &&\bullet \ar[dl] &\\
&\bullet \ar[d] &&\\
&\bullet &&
\end{tikzcd}
\tab\item
\begin{tikzcd}[column sep=tiny, row sep=tiny]
&\bullet \ar[dl] \ar[dr] \ar[ddd, dotted, dash]&&\\
\bullet \ar[d] &&\bullet \ar[d] &\bullet \ar[l]\\
\bullet \ar[dr] &&\bullet \ar[dl] &\\
&\bullet &&\\
&\bullet \ar[u] &&
\end{tikzcd}
\tab\item
\begin{tikzcd}[column sep=tiny, row sep=tiny]
&\bullet \ar[ddl] \ar[dr] \ar[dddd, dotted, dash] &&&\\
&&\bullet \ar[d] \ar[r] &\bullet \ar[r] &\bullet\\
\bullet \ar[ddr] &&\bullet \ar[d] &&\\
&&\bullet \ar[dl] &&\\
&\bullet &&&
\end{tikzcd}
\tab\item 
\begin{tikzcd}[column sep=tiny, row sep=tiny]
&\bullet \ar[dl] \ar[dr] \ar[dddd, dotted, dash] &&\\
\bullet \ar[dd] &&\bullet \ar[d] &\\
&&\bullet \ar[d] &\\
\bullet \ar[dr] &&\bullet \ar[dl] &\bullet \ar[l] \\
&\bullet &&
\end{tikzcd}
\tab\item
\begin{tikzcd}[column sep=tiny, row sep=tiny]
&\bullet \ar[ddl] \ar[dr] \ar[dddd, dotted, dash] &&&\\
&&\bullet \ar[d] \ar[r] &\bullet &\bullet \ar[l]\\
\bullet \ar[ddr] &&\bullet \ar[d] &&\\
&&\bullet \ar[dl] &&\\
&\bullet &&&
\end{tikzcd}
\tab\item
\begin{tikzcd}[column sep=tiny, row sep=tiny]
&\bullet \ar[dl] \ar[dr] \ar[dddl, dotted, dash, bend left=15]&\\
\bullet \ar[d] &&\bullet \ar[ddll]\\
\bullet \ar[d] \ar[drr] \ar[dddr, dotted, dash, bend left=20]&&\\
\bullet \ar[d] &&\bullet \ar[d]\\
\bullet \ar[dr] &&\bullet \ar[dl]\\
&\bullet &
\end{tikzcd}
\tab\item 
\begin{tikzcd}[column sep=tiny, row sep=tiny]
&&\bullet \ar[d] &\\
&&\bullet \ar[dl] \ar[ddr] \ar[dddd, dotted, dash] &\\
\bullet \ar[r] &\bullet \ar[dd] &&\\
&&&\bullet \ar[ddl]\\
&\bullet \ar[dr] &&\\
&&\bullet \ar[d] &\\
&&\bullet &
\end{tikzcd}
\tab\item
\begin{tikzcd}[column sep=tiny, row sep=tiny]
&&\bullet \ar[d] &\\
&&\bullet \ar[dl] \ar[ddr] \ar[dddd, dotted, dash] &\\
&\bullet \ar[dd] &&\\
&&&\bullet \ar[ddl]\\
\bullet &\bullet \ar[l] \ar[dr] &&\\
&&\bullet \ar[d] &\\
&&\bullet &
\end{tikzcd}
\tab\item
\begin{tikzcd}[column sep=tiny, row sep=tiny]
&&\bullet \ar[d] &\\
&&\bullet \ar[dl] \ar[ddr] \ar[dddd, dotted, dash] &\\
&\bullet \ar[dd] &&\\
&&&\bullet \ar[ddl]\\
\bullet &\bullet \ar[l] \ar[dr] &&\\
&&\bullet &\\
&&\bullet \ar[u] &
\end{tikzcd}
\tab\item 
\begin{tikzcd}[column sep=tiny, row sep=tiny]
&&\bullet &\\
&&\bullet \ar[u] \ar[dl] \ar[ddr] \ar[dddd, dotted, dash] &\\
\bullet \ar[r] &\bullet \ar[dd] &&\\
&&&\bullet \ar[ddl]\\
&\bullet \ar[dr] &&\\
&&\bullet \ar[d] &\\
&&\bullet &
\end{tikzcd}
\tab\item
\begin{tikzcd}[column sep=tiny, row sep=tiny]
&\bullet \ar[dl] \ar[dr] \ar[dddd, dotted, dash] &&\\
\bullet \ar[dd] &&\bullet \ar[d] \ar[r] &\bullet \\
&&\bullet \ar[d] &\\
\bullet \ar[dr] &&\bullet \ar[dl] & \\
&\bullet &&
\end{tikzcd}
\tab\item
\begin{tikzcd}[column sep=tiny, row sep=tiny]
&\bullet \ar[ddl] \ar[dr] \ar[dddd, dotted, dash] &&&\\
&&\bullet \ar[d] &&\\
\bullet \ar[ddr] &&\bullet \ar[d] &&\\
&&\bullet \ar[dl] &\bullet \ar[l] &\bullet \ar[l]\\
&\bullet &&&
\end{tikzcd}
\tab\item 
\begin{tikzcd}[column sep=tiny, row sep=tiny]
&\bullet \ar[ddl] \ar[dr] \ar[dddd, dotted, dash] &&&\\
&&\bullet \ar[d] &&\\
\bullet \ar[ddr] &&\bullet \ar[d] &&\\
&&\bullet \ar[dl] &\bullet \ar[l] \ar[r] &\bullet \\
&\bullet &&&
\end{tikzcd}
\tab\item
\begin{tikzcd}[column sep=tiny, row sep=tiny]
&&&\bullet \ar[dl] \ar[ddr] \ar[dddd, dotted, dash] &&\\
&&\bullet \ar[dd] &&&\\
&&&&\bullet \ar[ddl] \ar[r] &\bullet \\
\bullet \ar[r] &\bullet \ar[r] &\bullet \ar[dr] &&&\\
&&&\bullet &&
\end{tikzcd}
\tab\item
\begin{tikzcd}[column sep=tiny, row sep=tiny]
&&&\bullet \ar[dl] \ar[ddr] \ar[dddd, dotted, dash] &&\\
&&\bullet \ar[dd] &&&\\
&&&&\bullet \ar[ddl] \ar[r] &\bullet \\
\bullet &\bullet \ar[l] \ar[r] &\bullet \ar[dr] &&&\\
&&&\bullet &&
\end{tikzcd}
\tab\item 
\begin{tikzcd}[column sep=tiny, row sep=tiny]
&\bullet \ar[dl] \ar[dr] \ar[ddl, dotted, dash, bend left=15]&&\\
\bullet \ar[d] \ar[dddr, dotted, dash, bend left=20] \ar[drr] &&\bullet \ar[dll] \ar[r] &\bullet\\
\bullet \ar[d] &&\bullet \ar[ddl] &\\
\bullet \ar[dr] &&&\\
&\bullet &&
\end{tikzcd}
\tab\item
\begin{tikzcd}[column sep=tiny, row sep=tiny]
&\bullet \ar[ddl] \ar[drr]  \ar[dddddr, dotted, dash] &&&\bullet  \ar[dddd, dotted, dash] \ar[dl] \ar[ddr] &\\
&&&\bullet \ar[dd] &&\\
\bullet \ar[dddrr] &&&&&\bullet \ar[ddl] \\
&&&\bullet \ar[dr] &&\\
&&&&\bullet \ar[dll] &\\
&&\bullet &&&
\end{tikzcd}
\tab\item
\begin{tikzcd}[column sep=tiny, row sep=tiny]
&&&\bullet \ar[dl] \ar[ddr] \ar[dddd, dotted, dash] &&\\
\bullet &\bullet \ar[l] &\bullet \ar[l] \ar[dd] &&& \\
&&&&\bullet \ar[ddl] &\bullet \ar[l] \\
&&\bullet \ar[dr] &&& \\
&&&\bullet &&
\end{tikzcd}
\tab\item 
\begin{tikzcd}[column sep=tiny, row sep=tiny]
&&&\bullet \ar[dl] \ar[ddr] \ar[dddd, dotted, dash] &&\\
\bullet \ar[r] &\bullet &\bullet \ar[l] \ar[dd] &&& \\
&&&&\bullet \ar[ddl] &\bullet \ar[l] \\
&&\bullet \ar[dr] &&& \\
&&&\bullet &&
\end{tikzcd}
\tab\item
\begin{tikzcd}[column sep=tiny, row sep=tiny]
&&&&\bullet  \ar[dddd, dotted, dash] \ar[dl] \ar[ddr] &\\
&&&\bullet \ar[dd] &&\\
&&&&&\bullet \ar[ddl] \\
&&&\bullet \ar[dl] \ar[dr] \ar[dd, dotted, dash] &&\\
&\bullet \ar[r] &\bullet \ar[dr] &&\bullet \ar[dl] &\\
&& &\bullet &&
\end{tikzcd}
\tab\item
\begin{tikzcd}[column sep=tiny, row sep=tiny]
&&\bullet \ar[dl] \ar[dddr] \ar[dddddddl, dotted, dash, bend left=5] &\\
&\bullet \ar[dl] \ar[dd] \ar[dddd, dotted, dash, bend right=35] &&\\
\bullet \ar[dd] &&&\\
&\bullet \ar[dd] &&\bullet \ar[ddddll] \\
\bullet \ar[dddr] \ar[dr] &&&\\
&\bullet &&\\
&&&\\
&\bullet &&
\end{tikzcd}
\tab\item 
\begin{tikzcd}[column sep=tiny, row sep=tiny]
&&&\bullet &\bullet \ar[l] \ar[dl] \ar[dr] \ar[dd, dotted, dash]&\\
&&&\bullet \ar[dl] \ar[dr] \ar[dd, dotted, dash] &&\bullet \ar[dl]\\
&\bullet \ar[r] &\bullet \ar[dr] &&\bullet \ar[dl] &\\
&& &\bullet &&
\end{tikzcd}
\tab\item
\begin{tikzcd}[column sep=tiny, row sep=tiny]
&\bullet \ar[d] &&&\\
&\bullet \ar[dl] \ar[ddr] \ar[ddddr, dotted, dash, bend right=35] &&\bullet \ar[ddl] \ar[ddr] \ar[ddddl, dotted, dash, bend left=25] &\\
\bullet \ar[dd] &&&&\\
&&\bullet \ar[dd] &&\bullet \ar[ddll] \\
\bullet \ar[drr] &&&&\\
&&\bullet &&
\end{tikzcd}
\tab\item
\begin{tikzcd}[column sep=tiny, row sep=tiny]
&&\bullet \ar[d] \ar[drr] \ar[dll] \ar[dddl, dotted, dash, bend right=25] \ar[ddr, dotted, dash, bend left=25]&&\\
\bullet\ar[d] &&\bullet \ar[d] \ar[dr] &&\bullet \ar[dl]\\
\bullet \ar[dr] &&\bullet \ar[dl] &\bullet &\\
&\bullet &&&
\end{tikzcd}
\tab\item 
\begin{tikzcd}[column sep=tiny, row sep=tiny]
&&\bullet \ar[dl] \ar[dr] \ar[ddd, dotted, dash] &&\\
&\bullet \ar[ddl] \ar[ddr] \ar[dddr, dotted, dash, bend right=25] &&\bullet \ar[d] &\\
&&&\bullet \ar[dl] \ar[dr] \ar[ddl, dotted, dash, bend left=10] &\\
\bullet \ar[drr] &&\bullet \ar[d] &&\bullet \ar[dll] \\
&&\bullet &&
\end{tikzcd}
\tab\item
\begin{tikzcd}[column sep=tiny, row sep=tiny]
&\bullet \ar[d] &&&\\
&\bullet \ar[ddl] \ar[ddr] \ar[ddddr, dotted, dash, bend right=25] &&\bullet \ar[ddl] \ar[ddr] \ar[ddddl, dotted, dash, bend left=25] &\\
&&&&\\
\bullet \ar[ddrr] &&\bullet \ar[dd] &&\bullet \ar[ddll] \\
&&&&\\
&&\bullet \ar[d] &&\\
&&\bullet &&
\end{tikzcd}
\tab\item
\begin{tikzcd}[column sep=tiny, row sep=tiny]
&&\bullet \ar[d] \ar[drr] \ar[dll] \ar[dddl, dotted, dash, bend right=25] \ar[dddrr, dotted, dash, bend right=25]&&\\
\bullet \ar[ddr] &&\bullet \ar[d] &&\bullet \ar[dd]\\
&&\bullet \ar[dl] \ar[dr] &&\\
&\bullet &&\bullet \ar[r] &\bullet
\end{tikzcd}
\tab\item 
\begin{tikzcd}[column sep=tiny, row sep=tiny]
&&&\bullet \ar[d] &\\
&\bullet \ar[ddl] \ar[ddr] \ar[ddddr, dotted, dash, bend right=25] &&\bullet \ar[ddl] \ar[ddr] \ar[ddddl, dotted, dash, bend left=25] &\\
&&&&\\
\bullet \ar[ddrr] &&\bullet \ar[dd] &&\bullet \ar[ddll] \\
&&&&\\
&&\bullet &&\\
&&\bullet \ar[u] &&
\end{tikzcd}
\tab\item
\begin{tikzcd}[column sep=tiny, row sep=tiny]
&&\bullet \ar[dl] \ar[dr] \ar[dd, dotted, dash] &&&\\
&\bullet \ar[dl] \ar[dr] \ar[dddr, dotted, dash, bend right=25] &&\bullet \ar[dl] &\bullet \ar[l] \ar[dr] \ar[dddll, dotted, dash, bend left=10] &\\
\bullet \ar[ddrr] &&\bullet \ar[dd] &&&\bullet \ar[ddlll]\\
&&&&&\\
&&\bullet &&&
\end{tikzcd}
\tab\item
\begin{tikzcd}[column sep=tiny, row sep=tiny]
&&\bullet \ar[dll] \ar[drr] \ar[d] \ar[ddddll, dotted, dash, bend right=25] \ar[ddr, dotted, dash, bend left=25] &&\\
\bullet \ar[ddd] &&\bullet \ar[dl] \ar[dr] \ar[dd, dotted, dash] &&\bullet \ar[dl]\\
&\bullet \ar[dr] \ar[ddl] &&\bullet \ar[dl] &\\
&&\bullet &&\\
\bullet &&&&
\end{tikzcd}
\tab\item 
\begin{tikzcd}[column sep=tiny, row sep=small]
&\bullet \ar[dl] \ar[dr] \ar[ddddrr, dotted, dash, bend right=15] &&\bullet \ar[dl] \ar[dr] \ar[dddd, dotted, dash] &&\bullet \ar[dl] \ar[dr] \ar[ddddll, dotted, dash, bend left=15] &\\
\bullet \ar[dddrrr] &&\bullet \ar[dddr] &&\bullet \ar[dddl] &&\bullet \ar[dddlll]\\
&&&&&&\\
&&&&&&\\
&&&\bullet &&&
\end{tikzcd}
\tab\item
\begin{tikzcd}[column sep=tiny, row sep=small]
&\bullet \ar[ddl] \ar[dr] \ar[ddddr, dotted, dash, bend right=25] &&\bullet \ar[dl] \ar[ddr] \ar[ddddl, dotted, dash, bend left=25] &&\\
&&\bullet \ar[d] &&&\\
\bullet \ar[ddrr] &&\bullet \ar[dd] &&\bullet \ar[ddll] &\bullet \ar[l]\\
&&&&&\\
&&\bullet &&&\\
\end{tikzcd}
\tab\item
\begin{tikzcd} [column sep=tiny, row sep=tiny]
&\bullet \ar[dr] \ar[dddl] \ar[dddddd, dotted, dash] &&\\
&&\bullet \ar[d] &\\
&&\bullet \ar[dd] &\\
\bullet \ar[dddr] &&&\\
&&\bullet \ar[d] &\\
&&\bullet \ar[dl] \ar[r] &\bullet \\
&\bullet &&
\end{tikzcd}
\tab\item 
\begin{tikzcd} [column sep=tiny, row sep=tiny]
&&\bullet \ar[d]&\\
&&\bullet \ar[d]&\\
&&\bullet \ar[d]&\\
&&\bullet \ar[dr] \ar[dl] \ar[dd, dotted, dash]&\\
\bullet \ar[r] &\bullet \ar[dr] &&\bullet \ar[dl]\\
&&\bullet &
\end{tikzcd}
\tab\item
\begin{tikzcd} [column sep=tiny, row sep=tiny]
&&\bullet &\\
&&\bullet \ar[d] \ar[u] &\\
&&\bullet \ar[d]&\\
&&\bullet \ar[dr] \ar[dl] \ar[dd, dotted, dash]&\\
\bullet \ar[r] &\bullet \ar[dr] &&\bullet \ar[dl]\\
&&\bullet &
\end{tikzcd}
\tab\item
\begin{tikzcd} [column sep=tiny, row sep=tiny]
&&&\bullet \ar[dr] \ar[ddl] \ar[dddd, dotted, dash] &&\\
&&&&\bullet \ar[dd] &\\
\bullet &\bullet \ar[l] &\bullet \ar[l] \ar[ddr] &&&\\
&&&&\bullet \ar[dl] \ar[r] &\bullet \\
&&&\bullet &&
\end{tikzcd}
\tab\item 
\begin{tikzcd}[column sep=tiny, row sep=tiny]
&&&\bullet \ar[d] &\\
&\bullet \ar[dl] \ar[ddr] \ar[ddddr, dotted, dash, bend right=25] &&\bullet \ar[ddl] \ar[ddr] \ar[ddddl, dotted, dash, bend left=25] &\\
\bullet \ar[dd] &&&&\\
&&\bullet \ar[dd] &&\bullet \ar[ddll]\\
\bullet \ar[drr] &&&&\\
&&\bullet &&\\
\end{tikzcd}
\tab\item
\begin{tikzcd} [column sep=tiny, row sep=tiny]
&&&\bullet \ar[dr] \ar[ddl] \ar[dddd, dotted, dash] &&\\
&&&&\bullet \ar[dd] &\\
\bullet \ar[r] &\bullet &\bullet \ar[l] \ar[ddr] &&&\\
&&&&\bullet \ar[dl] \ar[r] &\bullet \\
&&&\bullet &&
\end{tikzcd}
\tab\item
\begin{tikzcd}[column sep=tiny, row sep=tiny]
&&&\bullet &\\
&\bullet \ar[dl] \ar[ddr] \ar[ddddr, dotted, dash, bend right=25] &&\bullet \ar[u] \ar[ddl] \ar[ddr] \ar[ddddl, dotted, dash, bend left=25] &\\
\bullet \ar[dd] &&&&\\
&&\bullet \ar[dd] &&\bullet \ar[ddll]\\
\bullet \ar[drr] &&&&\\
&&\bullet &&\\
\end{tikzcd}
\tab\item 
\begin{tikzcd} [column sep=tiny, row sep=tiny]
&&\bullet \ar[dr] \ar[ddl] \ar[dddd, dotted, dash] &&\\
&&&\bullet \ar[d] &\\
\bullet &\bullet \ar[l] \ar[ddr] &&\bullet \ar[d] &\\
&&&\bullet \ar[dl] \ar[r] &\bullet \\
&&\bullet &&
\end{tikzcd}
\tab\item
\begin{tikzcd}[column sep=tiny, row sep=tiny]
&&&\bullet \ar[d] &\\
&&&\bullet \ar[d] &\\
&\bullet \ar[ddl] \ar[ddr] \ar[ddddr, dotted, dash, bend right=25] &&\bullet \ar[ddl] \ar[ddr] \ar[ddddl, dotted, dash, bend left=25] &\\
&&&&\\
\bullet \ar[ddrr] &&\bullet \ar[dd] &&\bullet \ar[ddll]\\
&&&&\\
&&\bullet &&\\
\end{tikzcd}
\tab\item
\begin{tikzcd} [column sep=tiny, row sep=tiny]
&\bullet \ar[dr] \ar[dl] \ar[ddd, dotted, dash] &&\\
\bullet \ar[d] &&\bullet \ar[d] \ar[r] &\bullet \\
\bullet \ar[dr] &&\bullet \ar[dl] &\\
&\bullet \ar[d] &&\\
&\bullet &&\\
\end{tikzcd}
\tab\item 
\begin{tikzcd} [column sep=tiny, row sep=tiny]
&\bullet \ar[d]&&\\
&\bullet \ar[dr] \ar[dl] \ar[ddd, dotted, dash] &&\\
\bullet \ar[d] &&\bullet \ar[d] & \\
\bullet \ar[dr] &&\bullet \ar[dl] &\bullet \ar[l]\\
&\bullet &&
\end{tikzcd}
\tab\item
\begin{tikzcd} [column sep=tiny, row sep=tiny]
&\bullet &&\\
&\bullet \ar[dr] \ar[dl] \ar[ddd, dotted, dash] \ar[u] &&\\
\bullet \ar[d] &&\bullet \ar[d] & \\
\bullet \ar[dr] &&\bullet \ar[dl] &\bullet \ar[l]\\
&\bullet &&
\end{tikzcd}
\tab\item
\begin{tikzcd} [column sep=tiny, row sep=tiny]
&\bullet \ar[ddl] \ar[dr] \ar[ddddr, dotted, dash, bend right=25] &&\bullet \ar[dl] \ar[ddr] \ar[ddddl, dotted, dash, bend left=25] &&\\
&&\bullet \ar[d] &&&\\
\bullet \ar[ddrr] &&\bullet \ar[dd] &&\bullet \ar[ddll] \ar[r] &\bullet \\
&&&&&\\
&&\bullet &&&\\
\end{tikzcd}
\tab\item 
\begin{tikzcd} [column sep=tiny, row sep=tiny]
&&&&\bullet \ar[dll] \ar[d] \ar[ddr] \ar[dddl, dotted, dash, bend right=25] \ar[dddd, dotted, dash, bend left=15] &\\
\bullet &\bullet \ar[l] &\bullet \ar[l] \ar[ddr] &&\bullet \ar[ddl] &\\
&&&&&\bullet \ar[ddl] \\
&&&\bullet \ar[dr] &&\\
&&&&\bullet &
\end{tikzcd}
\tab\item
\begin{tikzcd} [column sep=tiny, row sep=tiny]
&&&\bullet \ar[dll] \ar[d] \ar[ddrr] \ar[dddl, dotted, dash, bend right=25] \ar[dddddr, dotted, dash, bend left=15] &&\\
\bullet &\bullet \ar[l] \ar[ddr] &&\bullet \ar[ddl] &&\\
&&&&&\bullet \ar[dddl] \\
&&\bullet \ar[dr] &&&\\
&&&\bullet \ar[dr] &&\\
&&&&\bullet&
\end{tikzcd}
\tab\item
\begin{tikzcd} [column sep=tiny, row sep=tiny]
&&\bullet \ar[dl] \ar[dr] \ar[dd, dotted, dash] &&&\\
&\bullet \ar[dl] \ar[dr] \ar[dddr, dotted, dash, bend right=25] &&\bullet \ar[dddl, dotted, dash, bend left=10] \ar[dr] \ar[dl] &&\\
\bullet \ar[ddrr] &&\bullet \ar[dd] &&\bullet \ar[ddll] &\bullet \ar[l]\\
&&&&&\\
&&\bullet &&&
\end{tikzcd}
\tab\item
\begin{tikzcd} [column sep=tiny, row sep=small]
&&&\bullet \ar[dll] \ar[drr] \ar[d] \ar[ddl, dotted, dash, bend right=25] \ar[ddd, dotted, dash, bend left=25] &&&\\
\bullet &\bullet \ar[l] \ar[dr] &&\bullet \ar[dl] &&\bullet \ar[ddll] \ar[r] &\bullet \\
&&\bullet \ar[dr] &&&&\\
&&&\bullet &&&
\end{tikzcd}
\tab\item
\begin{tikzcd} [column sep=tiny, row sep=small]
&&&\bullet \ar[dll] \ar[drr] \ar[d] \ar[ddd, dotted, dash, bend right=25] \ar[ddd, dotted, dash, bend left=25] &&&\\
\bullet \ar[d] &\bullet \ar[l] \ar[ddrr] &&\bullet \ar[dd] &&\bullet \ar[ddll] \ar[r] &\bullet \\
\bullet &&&&&&\\
&&&\bullet &&&
\end{tikzcd}
\tab\item
\begin{tikzcd} [column sep=tiny, row sep=tiny]
&&\bullet \ar[dr] \ar[dl] \ar[dd, dotted, dash]&&\\
\bullet \ar[r] &\bullet \ar[dr] &&\bullet \ar[dl] \ar[r] &\bullet \ar[d] \\
\bullet \ar[u] &&\bullet &&\bullet
\end{tikzcd}
\tab\item
\begin{tikzcd} [column sep=tiny, row sep=tiny]
&&\bullet \ar[dr] \ar[dl] \ar[dd, dotted, dash]&&\\
\bullet \ar[r] &\bullet \ar[dr] &&\bullet \ar[dl] \ar[r] &\bullet \\
\bullet \ar[u] &&\bullet &&\bullet \ar[u]
\end{tikzcd}
\tab\item
\begin{tikzcd} [column sep=tiny, row sep=tiny]
&&\bullet \ar[dr] \ar[dl] \ar[dd, dotted, dash]&&\\
\bullet \ar[d] \ar[r] &\bullet \ar[dr] &&\bullet \ar[dl] \ar[r] &\bullet \\
\bullet &&\bullet &&\bullet \ar[u]
\end{tikzcd}
\tab\item
\begin{tikzcd} [column sep=tiny, row sep=tiny]
&&\bullet \ar[dr] \ar[dl] \ar[dd, dotted, dash]&&\\
\bullet \ar[d] \ar[r] &\bullet \ar[dr] &&\bullet \ar[dl] \ar[r] &\bullet \ar[d]\\
\bullet &&\bullet &&\bullet 
\end{tikzcd}
\end{inparaenum}

\end{theorem}
\begin{proof} \label{provae7til}
We consider the list of frames of concealed algebras, see \cite{ass-sko2}. In the frame part $\widetilde{\mathbb{E}}_7$  has $23$ frames: $\mathcal{F}r11$, $\dotsc$, $\mathcal{F}r32$. As we explained at the beginning of this section, we will use only schurian frames. For each frame, we will make the admissible operation resulting in a concealed schurian algebra $A$. Next, we will put the necessary information on the trivial extension $T(A)$ in the computer program. Finally, the members of this ANS family of type  $\widetilde{\mathbb{E}}_7$ will be the algebras $K\Delta$ defined by a cutting sets of $T(A)$. We will show the work on the case of $\mathcal{F}r11$, the necessary work on the other cases follows an analogous path.

We apply the admissible operation $1$ in the Euclidean frame $\widetilde{\mathbb{E}}_7$, resulting in the hereditary algebra of type $\widetilde{\mathbb{E}}_7$, this algebra is a concealed algebra. Thanks to the theorem \ref{ANS}, we will analyze the trivial extension of the concealed algebra by noting the existence of some cutting set which defines a non-hereditary PHI algebra. For this, in the trivial extension, we must have at least two elementary cycles that have at least one common arrow, according to lemma \ref{cortephia}.

Therefore, we must have two maximal paths in hereditary algebras that have at least length $2$ and one arrow in common. Up to dual graphs, the possibilities are:

\begin{tikzcd} [column sep=scriptsize, row sep=scriptsize]
&\bullet \ar[dl] \ar[drrr] &&&\\
***  \ar[r, dashed] &\bullet \ar[u] &\bullet \ar[l] &\bullet \ar[l] &\bullet \ar[l]
\end{tikzcd}
\hfill
\begin{tikzcd} [column sep=scriptsize, row sep=scriptsize]
&&&\bullet \ar[dlll] \ar[dr] &\\
\bullet	\ar[r] &\bullet \ar[r] &\bullet \ar[r] &\bullet \ar[u] &*** \ar[l, dashed]
\end{tikzcd}
\hfill
\begin{tikzcd} [column sep=scriptsize, row sep=scriptsize]
&\bullet \ar[dl] \ar[dr] &\\
*** \ar[r, dashed] &\bullet \ar[u] &*** \ar[l, dashed] 
\end{tikzcd}

The symbol $***$ can be replaced by 
\begin{tikzcd}[column sep=scriptsize, row sep=scriptsize]
\bullet \ar[r, bend right] &\bullet \ar[l, bend right]
\end{tikzcd}, or
\begin{tikzcd}[column sep=scriptsize, row sep=scriptsize]
\bullet \ar[r, bend right] &\bullet \ar[l, bend right] \ar[r, bend right] &\bullet \ar[l, bend right]
\end{tikzcd}, or
\begin{tikzcd}[column sep=scriptsize, row sep=scriptsize]
\bullet \ar[rr, bend right] &\bullet \ar[l] &\bullet \ar[l]
\end{tikzcd}, or
\begin{tikzcd}[column sep=scriptsize, row sep=scriptsize]
\bullet \ar[r] &\bullet \ar[r] &\bullet \ar[ll, bend right]
\end{tikzcd}.

We will investigate all the trivial extensions of the case 1. We will put aside the trivial extensions that satisfy the hypothesis of the lemma \ref{cortenphia}:

One last remark, the path
\begin{tikzcd}
\bullet \ar[r, dashed, dash] &\bullet
\end{tikzcd} 
has length one or two.

Therefore, as a consequence of this filter, we obtain only the trivial extension, on the right of the picture below:

\begin{tikzcd} [column sep=tiny, row sep=normal]
&&&\bullet &&&\\
\bullet	\ar[r] &\bullet \ar[r] &\bullet \ar[r] &\bullet \ar[u] &\bullet \ar[l] &\bullet \ar[l] &\bullet \ar[l]
\end{tikzcd}\hfill%
\begin{tikzcd} [column sep=tiny, row sep=normal]
&&&\bullet \ar[drrr] \ar[dlll] &&&\\
\bullet	\ar[r] &\bullet \ar[r] &\bullet \ar[r] &\bullet \ar{u}{\alpha} &\bullet \ar[l] &\bullet \ar[l] &\bullet \ar[l]
\end{tikzcd}

A direct application of the lemma \ref{cortephia}, we have the cutting set $\{ \alpha \}$ that defines the solution 1.
\end{proof}

A work for the future is to apply the theorem \ref{ANS} in frames $\mathcal{F}r33$, $\dotsc$, $\mathcal{F}r149$ from the list of Happel and Vossieck \cite{hap-vos} with an analogous demonstration of the previous theorem. This would solve the description of PHI algebras of the ANS family of type  $\widetilde{\mathbb{E}}_8$.

%% file: ape-programa2.tex
\section{The source code of the main page of the site} \label{ape:prog}

\begin{verbatim}
<!DOCTYPE html>
<html>
<head>
<title> 
The cutting sets of given trivial extension that 
define incidence algebras
</title>
<meta name="description" content="The cutting sets">
<meta http-equiv="Content-Type" content="text/html; 
charset=utf-8"> 
</head>
<body>
<script src="app.js"></script>
<h1> 
The cutting sets of given trivial extension that 
define incidence algebras
</h1> 
<h2>
Complete the form below to calculate the cutting sets
</h2>
<form name="Untitled-2" method="post">
	<fieldset>
		<legend>
			Number of relations of type 2		
		</legend>
			<table cellspacing="3">
  				<tr>
   					<td>
    					<label for="rels">
    						Relations
    					</label>
   					</td>
   					<td align="left">
    					<input type="text" name="rels" id="relsid" 
required="required" placeholder=0>
					</td>
				</tr>
			</table>
	</fieldset>
	<fieldset>
		<legend>
			Number of different arrows in these relations		
		</legend>
			<table cellspacing="3">
  				<tr>
   					<td>
    					<label for="flechas">
    						Arrows
    					</label>
   					</td>
   					<td align="left">
    					<input type="text" name="flechas" 
id="flechasid" required="required" placeholder=0>
					</td>
				</tr>
			</table>
	</fieldset>
	<fieldset>
		<legend>
			Number of elementary cycles		
		</legend>
			<table cellspacing="3">
  				<tr>
   					<td>
    					<label for="ciclos">
    						Cycles
    					</label>
   					</td>
   					<td align="left">
    					<input type="text" name="ciclos" id="ciclosid" 
required="required" placeholder=0>
					</td>
				</tr>
			</table>
	</fieldset>
	<input type="button" id="enviarid" onclick="Enviar();" 
value="Send">
	<input type="reset" onClick="Limpar();" value="Clean">
</form>
</body>
</html>
\end{verbatim}

\section{The source code for app.js}
\begin{verbatim}
function Enviar() {
	// input variables
	flechas = document.getElementById("flechasid");
	relacoes = document.getElementById("relsid");  
	ciclos = document.getElementById("ciclosid");
	// previous variables only reading
	document.getElementById('flechasid').readOnly = true;
	document.getElementById('relsid').readOnly = true;
	document.getElementById('ciclosid').readOnly = true;
	// transform these variables into integers 
	var fls = parseInt(flechas.value, 10);
	var rels = parseInt(relacoes.value, 10);
	var cics = parseInt(ciclos.value,10);	
	// verification of incoming information
	if (fls > 0 && fls < 30 && cics > 0 && cics < 10 &&
rels > 0 && rels < 50) { 
		// preparation for table creation
		var cab1 = [];
		var cab2 = [];
		var lin = [];
		for (var i=1; i<=rels; i++){
			cab1.push('r'+i);
		};
		for (var i=1; i<=cics; i++){
			cab2.push('c'+i);
		};
		for (var i=1; i<=fls; i++){
			lin.push('alfa'+i);
		};
		var cab = cab1;
		for (var i=0; i < cab2.length; i++){
			cab.push(cab2[i]);
		};
		var data = {
			celulas: cab,
			idLinhas: lin
		};
		var table = document.createElement('table');
		// table header
		var tr = document.createElement('tr');
		var th = document.createElement('th');
		tr.appendChild(th); 
		data.celulas.forEach(function (celula) {
  			var th = document.createElement('th');
    		th.innerHTML = celula;
    		tr.appendChild(th);
		});
		table.appendChild(tr);
		// body
		data.idLinhas.forEach(function (id) {
    		// create new line
    		var tr = document.createElement('tr');
    		tr.dataset.id = id;
    		// first cell with line name
    		var td = document.createElement('td');
    		td.innerHTML = id;
    		tr.appendChild(td);
    		// scroll array of TDs
    		data.celulas.forEach(function (celula) {
        		var td = document.createElement('td');
        		var input = document.createElement('input');
        		input.type = 'checkbox';
        		input.name = celula + [];
        		td.appendChild(input);
        		tr.appendChild(td);
			});
    		table.appendChild(tr);
		});
		// add table to document
		document.body.appendChild(table);
		// disable submit button
		document.getElementById("enviarid").disabled = true;
		// create process button
  		var botao = document.createElement('input');
  		botao.type = 'button';
 		botao.value = 'Ready!';
  		// process
  		botao.onclick = function() {
			// function to add the vectors
			Array.prototype.add = function( b ) {
    			var a = this,
        		c = [];
    			if( Object.prototype.toString.call( b ) === 
    			'[object Array]' ) {
        			if( a.length !== b.length ) {
            		throw "Array lengths do not match.";
        			} else {
            		for( var i = 0; i < a.length; i++ ) {
                		c[ i ] = a[ i ] + b[ i ];
            		}
        			}
    			} else if( typeof b === 'number' ) {
        			for( var i = 0; i < a.length; i++ ) {
            			c[ i ] = a[ i ] + b;
        			}
    			}
    			return c;
			};
			// function to make the combination
			function combinations( inpSet, qty ) {
  				var inpLen = inpSet.length;
  				var combi = [];
  				var result = []; 
  				var recurse = function( left, right ) {
    				if ( right === 0 ) {
      					result.push( combi.slice(0) );
    				} else {
      				for ( var i = left; i <= inpLen - right; i++ ) {
        					combi.push( inpSet[i] );
        					recurse( i + 1, right - 1 );
        					combi.pop();
      					}
    				}
  				};
  				recurse( 0, qty );
  				return result;
			}
  			// create data object
  			var et = {  };
  			for (var i=0; i < lin.length; i++){
				et['alfa'+(i+1)]=[];
  			};
			// transform checkbox in 0 or 1
			for (var i in cab){
				var cols = document.getElementsByName(cab[i]);
				var temp = [];
				for (var j = 0, l = cols.length; j < l; j++) {
 					if (cols[j].checked) {
						et['alfa'+(j+1)].push(1);
					} else {
						et['alfa'+(j+1)].push(0);
					};
				};
			};
			// column with the total of each arrow in the cycles
			for (var i in et){
    			var temp = [];
    			for (var j=rels; j < rels + cics; j++){
        			temp.push(et[i][j]);
    			}
    			et[i].push(
            	temp.reduce(
               		function(prev, cur) {
                    		return prev + cur;
                	}
            	)
    			);
			}
			// create the variable mat
			var mat = [];
			for (var i in et){
    			mat.push(et[i]);
			}
			// reference array with total column TotRe2
			var TotRe2 = [];
			for (var i in et){
    			var temp = [];
    			temp.push(i);
    			var temp2 = [];
    			for (var j=0; j < rels; j++){
        			temp2.push(et[i][j]);
    			}
    			temp.push(
            	temp2.reduce(
               	function(prev, cur) {
                  	return prev + cur;
                	}
            	)
    			);
    			TotRe2.push(temp);
			}
			// descending organization of TotRe2
			TotRe2.sort(function (a, b) {
  				if (a[1] < b[1]) {
					return 1;
				}
				if (a[1] > b[1]) {
					return -1;
				}
				// a must be equal to b
				return 0;
			});
			// selection of the arrow with more presence in the 
relations of type 2 and variables
			var ref = TotRe2[0][0];
			var corte = cics - et[ref][rels+cics];
			var soma = et[ref];
			var sol = [ref];
			// cutting set of trivial extensial
			while (cics * TotRe2[0][1] >= rels){
				// separate the arrows that are not in the cycles 
of the ref arrow
				for (var j= rels; j < rels + cics; j++){
					if (et[ref][j] !== 0){
						for (var i in mat){
							if (mat[i][j] !== 0 && mat[i] !== 'fora' && i !== ref){
								mat[i]='fora';
							}
						}
					}
				}
				// function that returns false if sum is not solution
				function solucao(){
					for (var i=0; i < rels; i++){
						if (soma[i] === 0){
							return false;						
						}					
					}
					for (var i=rels; i < rels + cics; i++){
						if (soma[i] !== 1){
							return false;						
						}	
					}
					return true;
				} 				
				// see if the initial sum is a solution 
				if (corte === 0){
					if (solucao()){
						alert('Answer ' + sol);
					} else {
						alert("Some typing error"); 
					}
				// otherwise go to the cutting process
				} else {
					// variables to make the combination
					var x = [];
					for (i in mat){
						if (mat[i] !== 'fora' && i !== ref){
							var aux = parseInt(i,10) + 1;
							x.push('alfa' + aux);
						}
					}
					// cutting set
					while (corte > 0){
						// combination
	        			var d = combinations(x,corte);
						// sum		
						for (var i=0; i < d.length; i++){
							for (var j=0; j < corte; j++){
								var a = et[d[i][j]];
								var c = soma.add( a );
								soma = c;
							}
            			// see if the sum is a solution 
            			if (solucao()){
							// organize the solution
							for (var j=0; j < corte; j++){
								sol.push(d[i][j]);
							}				
							// show the solution
							alert('Answer ' + sol);
						}
							soma = et[ref];
							sol = [ref];
						}
						corte = corte - 1;
					}
				}
				// preparation for a new reference
				et[ref] = 'fora';
				var mat = [];
				for (var i in et){
					mat.push(et[i]);
				}
				// remove the ref from the array to put the new ref
				TotRe2.shift();
				ref = TotRe2[0][0];
				corte = cics - et[ref][rels+cics];
				soma = et[ref];
				sol = [ref];
				if (TotRe2.length === 1){
					break;
				}
			}
  		};
		// Add that button
		document.body.appendChild(botao);
	} else {
		alert('Invalid information!');
		// free previous variables 
		document.getElementById('flechasid').readOnly = false;
		document.getElementById('relsid').readOnly = false;
		document.getElementById('ciclosid').readOnly = false;
	};
} 
function Limpar(){
	history.go(0);
	document.getElementById("enviarid").disabled = false;
}
\end{verbatim}